\documentclass[11pt,a4paper,reqno]{amsart}
\usepackage{amsthm,amsmath,amsfonts,amssymb,amsxtra,appendix,bookmark,dsfont,latexsym,bm,delarray,euscript,amsgen,amsbsy,amsopn,amscd,latexsym,mathrsfs}
\usepackage[usenames,dvipsnames]{color}
\definecolor{darkblue}{RGB}{0,0,170}
\definecolor{brickred}{RGB}{200,0,0}
\definecolor{darkblu}{RGB}{100,40,150}
\hypersetup{colorlinks=true, linkcolor=darkblue, citecolor=brickred}
\usepackage{enumitem}
\usepackage[margin=1in]{geometry}

\makeatletter
\def\@seccntformat#1{%
  \protect\textup{%
    \protect\@secnumfont
    \expandafter\protect\csname format#1\endcsname 
    \csname the#1\endcsname
    \protect\@secnumpunct
  }%
}

\makeatother

\makeatletter
\def\th@plain{%
  \thm@notefont{}
  \slshape 
}
\def\th@definition{%
  \thm@notefont{}
  \normalfont 
}
\makeatother
\theoremstyle{plain}
\newtheorem{theorem}{Theorem}
\numberwithin{theorem}{section}
\newtheorem{lemma}[theorem]{Lemma}
\newtheorem{proposition}[theorem]{Proposition}

\newtheorem{corollary}[theorem]{Corollary}

\theoremstyle{definition}
\newtheorem{definition}[theorem]{Definition}
\newtheorem{remark}[theorem]{Remark}

\newcommand{\Hbb}{\mathbb{H}}

\newcommand{\R}{\mathbb{R}}
\newcommand{\N}{\mathbb{N}}

\newcommand{\T}{\mathbb{T}}
\newcommand{\M}{\mathcal{M}}
\renewcommand{\L}{\mathbb{L}}
\newcommand{\Gbb}{\mathbb{G}}

\newcommand{\dist}{\hbox{dist}}

\newcommand{\va}[1]{\left\{\begin{array}{r@{\text{ }}ll}#1\end{array}\right.}
\newcommand{\inner}[1]{\left \langle #1 \right\rangle}
\newcommand{\norm}[1]{\left\| #1\right\|}
\DeclareMathOperator{\loc}{loc}

\DeclareMathOperator{\dom}{dom}
\DeclareMathOperator{\SFL}{SFL}
\DeclareMathOperator{\CFL}{CFL}
\DeclareMathOperator{\RFL}{RFL}
\DeclareMathOperator{\divv}{div}

\usepackage{amsthm}

\numberwithin{equation}{section}

\begin{document}

\title[Semilinear nonlocal elliptic equations]{\large Semilinear nonlocal elliptic equations with \\ source term and measure data}

\author[P.-T. Huynh]{Phuoc-Truong Huynh}
\address[P.-T. Huynh]{Department of Mathematics and Statistics, Masaryk University, Brno, Czech Republic.}
\email{\tt hphuoctruong.hcmue@gmail.com, xhuynhp@math.muni.cz}

\author[P.-T. Nguyen]{Phuoc-Tai Nguyen}
\address[P.-T. Nguyen]{Department of Mathematics and Statistics, Masaryk University, Brno, Czech Republic.}
\email{\tt ptnguyen@math.muni.cz}


\begin{abstract} Recently, several works have been carried out in attempt to develop a theory for linear or sublinear elliptic equations involving a general class of nonlocal operators characterized by mild assumptions on the associated Green kernel. In this paper, we study the Dirichlet problem for superlinear equation (E) ${\mathbb L} u = u^p +\lambda \mu$ in a bounded domain $\Omega$ with homogeneous boundary or exterior Dirichlet condition, where $p>1$ and $\lambda>0$. The operator ${\mathbb L}$ belongs to a class of nonlocal operators including typical types of fractional Laplacians and the datum $\mu$ is taken in the optimal weighted measure space. The interplay between the operator ${\mathbb L}$, the source term $u^p$ and the datum $\mu$ yields substantial difficulties and reveals the distinctive feature of the problem. We develop a unifying technique based on a fine analysis on the Green kernel, which enables us to construct a theory for semilinear equation (E) in measure frameworks. A main thrust of the paper is to provide a fairly complete description of positive solutions to the Dirichlet problem for (E).
In particular, we show that there exist a critical exponent $p^*$ and a threshold value $\lambda^*$ such that the multiplicity holds for $1<p<p^*$ and $0<\lambda<\lambda^*$, the uniqueness holds for $1<p<p^*$ and $\lambda=\lambda^*$, and the nonexistence holds in other cases. Various types of nonlocal operators are discussed to exemplify the wide applicability of our theory. 

\bigskip
	
	
\bigskip
	

\end{abstract}

\maketitle

\tableofcontents

\section{Introduction}
In the last decades, the research of nonlocal equations has attracted much attention because of its interest to numerous areas such as probability, harmonic analysis, potential theory and other scientific disciplines. A large number of publications have been devoted to various aspects of elliptic equations driven by different types of operators ranging from typical fractional Laplacians to more general integro-differential operators. Recently, several works have originated in attempt to deal with linear or sublinear elliptic equations involving a wide class of nonlocal operators determined by two-sided estimates on the associated Green kernel (see e.g. \cite{BonVaz_2014, BonVaz_2015, BonSirVaz_2015, BonFigVaz_2018, Aba_2019,GomVaz_2019,KimSonVon_2020}). In the present paper, we are interested in positive solutions of nonlocal elliptic equations with  superlinear source terms of the form 
\begin{equation} \label{eq:E1}
\L u = u^p + \mu \text{ in }\Omega,
\end{equation} 
with homogeneous Dirichlet boundary or exterior condition
\begin{equation} \label{zeroDir}
u = 0 \quad \text{on } \partial \Omega \text{ or in } \Omega^c \text{ if applicable}.
\end{equation}
Here $\Omega \subset \R^N$ is a bounded $C^2$ domain with the complement $\Omega^c = \R^N \setminus \Omega$, $p>1$, $\mu$ is a positive Radon measure on $\Omega$ and $\L$ is a nonlocal operator posed in $\Omega$. Specific assumptions on $\L$ are given in Subsection \ref{sec:pre_assumption}. Basic examples of $\L$ to keep in mind are the restricted fractional Laplacian, spectral fractional Laplacian and censored (or regional) fractional Laplacian, see Subsection \ref{sec:someexamples}; some further illustrative examples are provided in Section \ref{sec:other-examples}. \medskip

\noindent \textbf{A survey on related works.} Let us recall some known results in the literature concerning problem \eqref{eq:E1} and \eqref{zeroDir}. One of the first breakthrough achievements in this topic was due to Lions in \cite{Lio_1980}, in which nonnegative solutions to Lane-Emden equation
\begin{equation} \label{eq:E2} - \Delta u = u^p \quad \text{in } \Omega \setminus \{0\}
\end{equation}
were sharply depicted, where $\Omega$ contains the origin $0$. In particular, when $N > 2$ and $1<p<\frac{N}{N-2}$, every nonnegative solution of \eqref{eq:E2} solves the following equation in the whole domain 
$$ - \Delta u = u^p + k \delta_0 \quad \text{in } \Omega 
$$
for some $k \geq 0$, where $\delta_0$ denotes the Dirac mass concentrated at $0$. A rather full understanding of singular solutions to \eqref{eq:E2} was acquired thanks to the complete characterization result for the more complicated range $\frac{N}{N-2}<p<\frac{N+2}{N-2}$ of Gidas and Spruck in \cite{GidSpr_1981}. Many significant developments have been obtained since the celebrated papers of Lions \cite{Lio_1980} and Gidas and Spruck \cite{GidSpr_1981}, among them we refer to the paper of Bidaut-V\'eron and Yarur \cite{BidYar_2002} where various necessary and sufficient criteria for the existence of nonnegative solutions to the problem with positive measure datum $\mu$
\begin{equation} \label{eq:E3} \left\{ \begin{aligned} 
- \Delta u  &= u^p + \mu &&\text{ in }\Omega,\\ 
u  &= 0 &&\text{ on }\partial \Omega,
\end{aligned} \right. \end{equation}
were obtained for $1< p <\frac{N+1}{N-1}$. Interesting extensions of these results to semilinear equations involving linear second order differential operators in divergence form were established by V\'eron in \cite{Ver-handbook}. Further investigation dealing with the question of uniqueness and multiplicity for \eqref{eq:E3} and variants thereof were carried out by Ferrero and Saccon \cite{FerSac_2006} (see also Naito and Sato \cite{NaiSat_2007}).

Several results regarding isolated singularities were extended by Chen and Quaas \cite{CheQua_2018} to semilinear fractional equations 
\begin{equation} \left\{  \begin{aligned}
(-\Delta)^s u &= u^p + k\delta_0 &&\text{ in }\Omega,\\ 
u &= 0 &&\text{ in } \Omega^c,
\end{aligned} \right. \end{equation}
where $(-\Delta)^s$, $s \in (0,1)$, is the well-known fractional Laplacian defined by
$$ (-\Delta)^s u(x):=c_{N,s} P.V. \int_{\R^N} \frac{u(x)-u(y)}{|x-y|^{N+2s}}dy,\quad x \in \Omega.$$
Here $c_{N,s} > 0$ is a normalized constant and the abbreviation P.V. stands for ``in the principal value sense". Afterwards, Chen, Felmer and V\'eron \cite{CheVerFel_2014} obtained an existence result for a more general source term and measure data by using a different approach which is based on Schauder fixed point theorem and delicate estimates on the Green kernel of $(-\Delta)^s$.

Recently  equations involving general operators characterized by estimates on the associated Green functions have been intensively investigated. One of the first and important contributions in this research direction from the PDE point of view was given by Bonforte and V\'azquez in \cite{BonVaz_2015} where a concept of weak-dual solutions was introduced to investigate a priori estimates for porous medium equations. For the existence and uniqueness results, the reader is referred to Bonforte, Sire and V\'azquez \cite{BonSirVaz_2015}. The notion of weak-dual solutions is convenient to study not only parabolic equations, but also elliptic equations. Later on, Bonforte, Figalli and V\'azquez \cite{BonFigVaz_2018} developed a theory for semilinear elliptic equations with sublinear nonlinearities in the context of weak-dual solutions including a priori estimates, Harnack inequality and regularity results. Further investigations into general nonlocal operators have been treated in recent works, for example Abatangelo, G\'omez-Castro and V\'azquez \cite{Aba_2019} (for the boundary behavior of solutions), Chan, G\'omez-Castro and V\'azquez \cite{ChaGomVaz_2019_2020} (for the study of eigenvalue problems and equivalent types of solutions). Besides, it is worth mentioning a series of papers of Kim, Song and Vondra\v cek \cite{KimSonVon_2020, KimSonVon_2020-1,KimSonVon_2019} where nonlocal operators were studied  from the probabilistic point of view. \medskip

\noindent \textbf{Aim of the paper.} Motivated by the above-mentioned works, we aim to establish the existence, nonexistence,  uniqueness, multiplicity and qualitative properties of solutions to \eqref{eq:E1} and \eqref{zeroDir}. The class of operators $\L$ under investigation is required to satisfy a set of mild conditions and includes quite a few notable models. An underlying assumption on $\L$ is expressed in terms of two-sided estimates on the associated Green kernel. In this regard, we offer a unifying approach which relies mainly on the use of the Green kernel and hence allows to deal with various types of local and nonlocal operators.      \medskip

\noindent \textbf{Acknowledgements.}
The authors were supported by Czech Science Foundation, project GJ19-14413Y.  P.-T. Huynh gratefully acknowledges Prof. Jan Slov\'ak for the kind hospitality and great support during his study at Masaryk University. The authors would like to thank the anonymous referee for the comments which helped to improve the paper.

\section{Preliminaries, basic assumptions and definitions}\label{sec:preliminaries}
Throughout the present paper, we assume that $\Omega \subset \R^N$ ($N > 2s)$ is a bounded domain with $C^2$ boundary. Put $\delta(x)=\dist(x,\partial \Omega)$. We denote by $C,c_1, c_2, \ldots$ positive constants that may vary from line to line. If necessary, we will write $C = C(a,b,\ldots)$ to emphasize the dependence of $C$ on $a,b,\ldots$ For a number $q \in (1,\infty)$, we denote by $q'$ the conjugate exponent of $q$, namely $q' = q/(q-1)$. For two functions $f,g$, we write $f \lesssim g$ ($f \gtrsim g$) if there exists a constant $C > 0$ such that $f \le C g$ ($f \ge C g$). We write $f \sim g$ if $f \lesssim g$ and $g \lesssim f$. We also denote $a\vee b := \max\{ a,b\}$ and $a\wedge b := \min \{ a, b\}$.   Finally, for two Banach spaces $X,Y$ such that $X \subset Y$, we write $X \hookrightarrow Y$ if $X$ is continuously embedded in $Y$ and $X \hookrightarrow \hookrightarrow Y$ if $X$ is compactly embedded in $Y$.

\subsection{Functional spaces}\label{sec:LPspace}\text{}

\textbf{Weighted Lebesgue spaces, Marcinkiewicz spaces and spaces of Radon measures.} Recall that for a nonnegative function $\eta$ and $q \in [1,\infty)$, the weighted Lebesgue space $L^q(\Omega,\eta)$ is defined by
$$L^q(\Omega,\eta ):= \left\{ f \in L^q_{\loc}(\Omega): \int_{\Omega}|f|^q\eta dx < +\infty \right\}.$$ 
We also denote
$$\eta L^{\infty}(\Omega) := \left\{ f :\Omega \to \R:  \text{ there exists } v \in L^{\infty}(\Omega) \text{ such that } f = \eta v \right\}.$$
Next, the weighted Marcinkiewicz space (or weighted weak Lebesgue space) $M^q(\Omega,\eta )$, $1 \le q < \infty$, is defined by
$$M^q(\Omega,\eta ):= \left\{ f \in L^1_{\loc}(\Omega): \sup_{\lambda > 0} \lambda^q \int_{\Omega} \mathbf{1}_{\{ x \in \Omega: |f(x)| \ge \lambda \} } \eta dx < +\infty \right\}.$$
The key role of Marcinkiewicz spaces is to give optimal estimates of the Green kernel, see Section \ref{sec:Green}. When  $\eta \equiv 1$, we recover the usual Lebesgue spaces and Marcinkiewicz spaces. In this case, we simply write $L^q(\Omega)$ and $M^q(\Omega)$ for $L^q(\Omega,\eta )$ and $M^q(\Omega,\eta )$ respectively.

We denote by $\M(\Omega,\eta )$ the space of weighted Radon measures defined as
$$\M(\Omega,\eta ) := \left\{ \mu \text{ is a Radon measure on } \Omega: \int_{\Omega} \eta d|\mu| < +\infty \right\}.$$
 Also, $\M^+(\Omega,\eta )$ is defined as the cone of positive measures on $\Omega$. 
In the case $\eta \equiv 1$, we have the space of bounded Radon measures $\M(\Omega)$. 

The following strict inclusions hold (see for instance \cite[Subsection 2.2]{BidViv_2000} and \cite[Subsection 1.1]{Gra_2009})
$$L^q(\Omega,\eta) \hookrightarrow M^q(\Omega, \eta ) \hookrightarrow L^m(\Omega,\eta ) \hookrightarrow \M(\Omega,\eta) \text{ for all }m,q \text{ such that } 1 \le m < q < \infty.$$
For more details on these functional spaces, the reader is referred to \cite{Gra_2009,MarVer_2014}. \medskip

\textbf{Fractional Sobolev spaces.} For $s \in (0,1) $, the fractional Sobolev space $H^s(\Omega)$ is defined by
$$H^s(\Omega):= \left\{ u \in L^2(\Omega): [u]_{H^s(\Omega)}:= \int_{\Omega}\int_{\Omega} \dfrac{|u(x)-u(y)|^2}{|x-y|^{N+2s}}dx dy < +\infty \right\}.$$
This space is Hilbert space equipped with the inner product
$$\inner{u,v}_{H^s(\Omega)}:= \int_{\Omega} uv dx + \int_{\Omega}\int_{\Omega} \dfrac{(u(x)-u(y))(v(x)-v(y))}{|x-y|^{N+2s}}dxdy$$
and the induced norm
$$ \|u \|_{H^s(\Omega)} := \sqrt{\langle u, u \rangle_{H^s(\Omega)}}.
$$

The space $H^s_0(\Omega)$ is defined as the closure of $C^{\infty}_c(\Omega)$ with respect to the norm $\norm{\cdot}_{H^s(\Omega)}$, i.e.
$$
H^s_0(\Omega) = \overline{C^{\infty}_c(\Omega)}^{\norm{\cdot}_{H^s(\Omega)}}.
$$
 
Let $H^{s}_{00}(\Omega)$ be defined by
$$
H^s_{00}(\Omega) := \left\{ u \in H^s(\Omega): \dfrac{u}{\delta^s} \in L^2(\Omega) \right\}.
$$
It can be seen that $H^s_{00}(\Omega)$ equipped with the norm
$$
\norm{u}^2_{H^s_{00}(\Omega)} := \int_{\Omega} \left(1+\dfrac{1}{\delta^{2s}}\right)|u|^2 dx + \int_{\Omega} \int_{\Omega} \dfrac{|u(x)-u(y)|^2}{|x-y|^{N+2s}}dx dy < +\infty
$$
is a Banach space. Roughly speaking, $H^s_{00}(\Omega)$ is the space of functions in $H^s(\Omega)$ satisfying the Hardy inequality. By \cite[Subsection 8.10]{Bha_2012}, there holds
$$
H^s_{00}(\Omega) = \va{
&H^s(\Omega) = H^s_0(\Omega), &\text{ if }0 < s < \frac{1}{2}, \\[4pt] 
&H^{\frac{1}{2}}_{00}(\Omega) \subsetneq H^{\frac{1}{2}}_{0}(\Omega), &\text{ if }s = \frac{1}{2},\\[4pt]
&H^s_0(\Omega), &\text{ if } \frac{1}{2} < s < 1.}
$$
In fact, $H^{s}_{00}(\Omega)$ is the space of functions in $H^s(\R^N)$ supported in $\Omega$, or equivalently, the trivial extension of functions in $H^s_{00}(\Omega)$ belongs to $H^s(\R^N)$ (see \cite[Lemma 1.3.2.6]{Gris_2011}). Furthermore, $C^{\infty}_c(\Omega)$ is a dense subset of $H^s_{00}(\Omega)$. When $s = \frac{1}{2}$, $H^{\frac{1}{2}}_{00}(\Omega)$ is called the Lions--Magenes space (see \cite[Theorem 7.1]{LioMag_1972}). The strict inclusion $H^{\frac{1}{2}}_{00}(\Omega) \subsetneq H^{\frac{1}{2}}_0(\Omega)$ holds since $1 \in H^{\frac{1}{2}}_0(\Omega)$ but $1 \notin H^{\frac{1}{2}}_{00}(\Omega)$.

Alternatively, fractional Sobolev spaces can be viewed as interpolation spaces due to \cite[Chapter 1]{LioMag_1972}.  For more details on fractional Sobolev spaces, the reader is referred to \cite{Bha_2012,BonSirVaz_2015,Gris_2011, LioMag_1972, NezPalVal_2012} and references therein.

\subsection{Main assumptions}\label{sec:pre_assumption} 
In the present paper, we deal with a class of nonlocal operators $\L$ which covers three typical types of fractional Laplace operators, as well as the classical Laplacian. Assumptions related to $\L$ are listed below.

\textbf{Assumptions on $\L$.} Inspired by \cite{BonSirVaz_2015,BonFigVaz_2018,Aba_2019,ChaGomVaz_2019_2020},  we make the following assumptions on $\L$.
\begin{enumerate}[resume,label=(L\arabic{enumi}),ref=L\arabic{enumi}]
	\item \label{eq_L1} $\L: C^{\infty}_c(\Omega) \subset L^2(\Omega) \to L^2(\Omega)$ is a positive, symmetric operator.
\end{enumerate}

Under assumption \eqref{eq_L1}, we infer from the standard theory (see, for instance, \cite[Theorem 4.14]{Dav_1980}) that $\L$ admits a positive, self-adjoint extension $\widetilde{\L}$, which is the Friedrich extension of $\L$. Furthermore, $\Hbb(\Omega):= \dom(\widetilde{\L}^{\frac{1}{2}})$ is a Hilbert space equipped with the inner product
\begin{equation} \label{inprod-1}
(u,v) \mapsto \inner{u,v}_{L^2(\Omega)} + \mathcal{B}(u,v),\quad u,v \in \Hbb(\Omega),
\end{equation}
where the $\mathcal{B}$ is the bilinear form defined as
$$
\mathcal{B}(u,v):=\inner{\widetilde{\L}^{\frac{1}{2}} u, \widetilde{\L}^{\frac{1}{2}} v}_{L^2(\Omega)},\quad u,v \in \Hbb(\Omega).
$$
The inner product \eqref{inprod-1} induces the following norm on $\Hbb(\Omega)$
\begin{equation} \label{norm1}
\sqrt{\| u \|_{L^2(\Omega)}^2 + \mathcal{B}(u,u)}, \quad u \in \Hbb(\Omega).
\end{equation}
\textit{From now on,  we use the same notation $\L$ to denote the extension $\widetilde{\L}$.}  

We further assume that
\begin{enumerate}[resume,label=(L\arabic{enumi}),ref=L\arabic{enumi}] \setcounter{enumi}{1}
	\item \label{eq_L2} There exists a constant $\Lambda  > 0$ such that
	\begin{equation}\label{eq:lowerbound}
	\norm{u}_{L^2(\Omega)}^2 \le \Lambda\inner{\L u, u}_{L^2(\Omega)},\quad \forall u \in C^{\infty}_c(\Omega). 
	\end{equation}
\end{enumerate}
Assumption \eqref{eq_L2} is closely related to the positivity of the first eigenvalue of $\L$ (see \cite[Subsection 3.1.1]{BonSirVaz_2015}). It can be also seen that $C^{\infty}_c(\Omega)$ is dense in $\Hbb(\Omega)$. Thus, under assumption \eqref{eq_L2} and the density of $C^{\infty}_c(\Omega)$, one has

$$
\norm{u}_{L^2(\Omega)}^2 \leq \Lambda \inner{ \L^{\frac{1}{2}} u, \L^{\frac{1}{2}} u}_{L^2(\Omega)} = \Lambda\mathcal{B}(u,u),\quad \forall u \in \Hbb(\Omega).
$$
Therefore, one can define an equivalent norm to the one in \eqref{norm1} as follows 
$$\norm{u}_{\Hbb(\Omega)}:= \sqrt{\mathcal{B}(u,u)}.
$$ 

In addition to the above properties, more specific characteristics of $\Hbb(\Omega)$ are required in our analysis. It can be verified that for typical nonlocal operators such as the restricted fractional Laplacian, the spectral fractional Laplacian and the censored (or regional) fractional Laplacian, $\Hbb(\Omega)$ can be identified with $H^s_{00}(\Omega)$ (see Subsection \ref{sec:someexamples} for more details). This indicates and motivates us to assume that
\begin{enumerate}[resume,label=(L\arabic{enumi}),ref=L\arabic{enumi}]
\setcounter{enumi}{2}
\item \label{eq_L3} $\Hbb(\Omega) = H^s_{00}(\Omega)$.
\end{enumerate}

\textbf{Assumptions on the inverse of $\mathbb{L}$.}  We also require the existence of the inverse operator of $\L$. 

\begin{enumerate}[label=(G\arabic{enumi}),ref=G\arabic{enumi}]
	\item \label{eq_G1} There exists an operator $\Gbb^{\Omega}$ such that for every $f \in C^{\infty}_c(\Omega)$,  one has
\begin{equation}\label{eq:homolinear}
\L [\Gbb^{\Omega} [f]] = f \text{ in } L^2(\Omega).
\end{equation}
In other words,  $\Gbb^{\Omega}$ is a right inverse of $\L$  and for every $f \in C^{\infty}_c(\Omega)$, one has $\Gbb^{\Omega}[f] \in \dom(\L) \subset \Hbb(\Omega)$.
    \item \label{eq_G2} The operator $\Gbb^{\Omega}$ admits a Green kernel $G^{\Omega}$, namely
$$
\Gbb^{\Omega}[f](x) = \int_{\Omega} G^{\Omega}(x,y)f(y)dy,\quad x \in \Omega,
$$    
where $G^{\Omega}: \Omega \times \Omega \setminus \{(x,x): x \in \Omega \} \to (0,\infty)$ is symmetric and satisfies
\begin{equation} \label{G-est} 
G^{\Omega} (x,y) \sim \dfrac{1}{|x-y|^{N-2s}} \left( \dfrac{\delta(x)}{|x-y|} \wedge 1 \right)^{\gamma} \left( \dfrac{\delta(y)}{|x-y|} \wedge 1 \right)^{\gamma}, \quad x,y \in \Omega, x \neq y,
\end{equation}
with $s,\gamma \in (0,1]$ and $N>2s$.
\end{enumerate}

Under assumptions \eqref{eq_G1} and \eqref{eq_G2}, our first main result provides existence and nonexistence results for \eqref{eq:E1} expressed in terms of the value of the exponent $p$ and the norm of the datum $\mu$ (see Theorem \ref{th:main1} for the precise statement of this result).

Finally, we impose the following condition on $N,s$ and $\gamma$ 
\begin{enumerate}[label=(G\arabic{enumi}),ref=G\arabic{enumi}]
\setcounter{enumi}{2}
\item \label{eq_G3} $N \geq N_{s,\gamma}$ where 
\begin{equation} \label{Nsgamma} 
N_{s,\gamma}:= \left\{ 
\begin{aligned} 
&4s &&\text{ if } \gamma < 2s,\\ 
&4s(1+\gamma) - \gamma  &&\text{ if }\gamma \ge 2s.
\end{aligned} \right. 
\end{equation}
\end{enumerate}
When $\gamma<2s$, assumption \eqref{eq_G3} becomes $N>4s$. This condition has been required in numerous papers dealing with variational settings (see e.g. Dipierro, Medina, Peral and Valdinoci \cite{DipMedPerVal_2017}, Servadei \cite{Ser_2014} and references therein). When $\gamma \geq 2s$,  we have $N_{s,\gamma} \leq 3$ and hence \eqref{eq_G3} always holds if $N \geq 3$. In the present paper, assumption (G3) will be employed to obtain an important regularity result in a variational setting (see Subsection 8.1).

Under the assumptions \eqref{eq_L1}--\eqref{eq_L3} and \eqref{eq_G1}--\eqref{eq_G3}, we establish uniqueness and simplicity results for  \eqref{eq:E1} (see Theorem \ref{th:main2} for the precise statement). 

\subsection{Some examples} \label{sec:someexamples} The class of operators satisfying the assumptions \eqref{eq_L1}--\eqref{eq_L3}, \eqref{eq_G1}--\eqref{eq_G3} is exemplified by the typical fractional Laplace operators.   \medskip

\textbf{The restricted fractional Laplacian (RFL).} One of the main examples is the fractional Laplacian  defined, for $s \in (0,1)$, by
$$\L u(x) = (-\Delta)^s_{\RFL} u(x): = c_{N,s}\, P.V. \int_{\R^N} \frac{u(x)-u(y)}{|x-y|^{N+2s}}dy,\quad x \in \Omega,$$
restricted to functions that are zero outside $\Omega$ as mentioned in the introduction. This operator has been intensively studied in the literature, see for instance \cite{Aba_2015,RosSer_2014,SerVal_2012,SerVal_2014} and references therein. For any $u \in C^{\infty}_c(\Omega)$, we deduce from \cite[page 13]{Garofalo_2019} that $(-\Delta)^s_{\RFL} u \in C^{\infty}(\overline{\Omega}) \subset L^2(\Omega)$. Therefore, for every $u,v \in C^{\infty}_c(\Omega)$,
$$
\int_{\Omega} u (-\Delta)^s_{\RFL} v dx = \dfrac{c_{N,s}}{2} \int_{\R^N} \int_{\R^N} \dfrac{(u(x)-u(y))(v(x)-v(y))}{|x-y|^{N+2s}}dx dy = \int_{\Omega} v(-\Delta)^s_{\RFL} u dx,
$$
hence \eqref{eq_L1} is satisfied. \eqref{eq_L2} follows from \cite[Theorem 6.5]{NezPalVal_2012} since
$$
\norm{u}_{L^2(\Omega)}^2 \lesssim \int_{\R^N} \int_{\R^N} \dfrac{|u(x)-u(y)|^2}{|x-y|^{N+2s}} = \inner{(-\Delta)^s_{\RFL} u,  u}_{L^2(\Omega)},\quad \forall u\in C^{\infty}_c(\Omega).
$$
By \cite[page 585]{SonVon_2003},
$$
\Hbb(\Omega) =  \left\{ u \in H^s(\R^N): u = 0 \text{ a.e. in }\Omega \right\} = H^s_{00}(\Omega),
$$
namely \eqref{eq_L3} is fulfilled.  

Next,  \eqref{eq_G1} holds due to \cite[page 5733]{BonSirVaz_2015} (see also Lemma \ref{lem:infinitesimal}) and \eqref{eq_G2} holds for $\gamma = s$ in  (see \cite{CheSon_1998,RosSer_2014}). Finally, in this case, assumption \eqref{eq_G3} reads as $N \geq 4s$.
 \medskip

\textbf{The spectral fractional Laplacian (SFL).} The spectral fractional Laplacian has been introduced in \cite{SonVon_2003} and studied in, for instance, \cite{AbaDup_2017,DhiMaaZri_2011,CafSti_2016,BraColPabSan_2013}.
It is defined, for $s \in (0,1)$, by
$$\L u(x) = (-\Delta)^s_{\SFL} u(x):= P.V. \int_{\Omega} [u(x)-u(y)]J_s(x,y)dy + \kappa_s(x) u(x),\quad x \in \Omega,  $$
where 
$$J_s(x,y) := \dfrac{s}{\Gamma(1-s)} \int_0^{\infty} K_{\Omega} (t,x,y) \dfrac{dt}{t^{1+s}}dt, \quad x,y \in \Omega,$$
and 
$$\kappa_s(x) := \dfrac{s}{\Gamma(1-s)} \int_0^{\infty} \left(1- \int_{\Omega} K_{\Omega} (t,x,y) dy\right) \dfrac{dt}{t^{1+s}} \sim \dfrac{1}{\delta(x)^{2s}},\quad x \in \Omega.$$
Here $K_{\Omega}$ is the heat kernel of the Laplacian $-\Delta$ in $\Omega$.
A class of more general spectral type operators defined as fractional orders of the local operator $\divv (A \nabla u)$, where $A$ is a $C^{1,\alpha}$ vector-valued function,  
was investigated in \cite{CafSti_2016}.

For this type of operators,  one has $(-\Delta)^s_{\SFL} u \in C^1_0(\overline{\Omega}) \subset L^2(\Omega)$ (see \cite[Lemma 18]{AbaDup_2017}, or Section \ref{sec:other-examples} for a more delicate approach) and for any $u, v \in C^{\infty}_c(\Omega)$,
\begin{align*}
\int_{\Omega} u (-\Delta)^s_{\SFL} v dx &= \frac{1}{2}\int_{\Omega}\int_{\Omega} (u(x)-u(y))(v(x)-v(y))J_s(x,y)dxdy + \int_{\Omega} \kappa_s(x)u(x)v(x) dx \\
&= \int_{\Omega} v (-\Delta)^s_{\SFL} u dx,
\end{align*}
which gives \eqref{eq_L1}.  Next, \eqref{eq_L2} is satisfied since
\begin{align*}
\norm{u}^2_{L^2(\Omega)} 
&\lesssim \frac{1}{2}\int_{\Omega}\int_{\Omega} (u(x)-u(y))^2 J_s(x,y)dxdy + \int_{\Omega} \kappa_s(x)|u(x)|^2 dx \\ &= \inner{ u, (-\Delta)^s_{\SFL} u}_{L^2(\Omega)}.
\end{align*}
By \cite{CafSti_2016,SonVon_2003}, \eqref{eq_L3} is fulfilled where $\Hbb(\Omega) = H^s_{00}(\Omega)$ is defined as an interpolation space

It can be seen that \eqref{eq_G1} holds by \cite[Lemma 11]{AbaDup_2017} (see also the Lemma \ref{lem:infinitesimal}) and  \eqref{eq_G2} holds with $\gamma = 1$ (see \cite{BogBurChe_2003,CafSti_2016,SonVon_2003}). Finally, assumption \eqref{eq_G3} becomes $N \geq 4s$ if $s>\frac{1}{2}$ or $N \geq 8s-1$ if $s \in (0,\frac{1}{2}]$.
\medskip

\textbf{The censored fractional Laplacian (CFL).} The censored fractional Laplacian is defined for $ s > \frac{1}{2}$ by
$$\L u(x) = (-\Delta)^s_{\CFL} u(x) := a_{N,s}\, P.V. \int_{\Omega} \dfrac{u(x)-u(y)}{|x-y|^{N+2s}}dy, \quad x \in \Omega.$$
This operator has been studied in \cite{BogBurChe_2003,Che_2018,Fal_2020}.  We verify that if $u \in C^{\infty}_c(\Omega)$, then $(-\Delta)^s_{\CFL} u \in L^{\infty}(\Omega)$.  Indeed,
\begin{align*}
(-\Delta)^s_{\CFL} u(x) =  \int_{\Omega \backslash B(x,\delta(x))} \dfrac{u(x)-u(y)}{|x-y|^{N+2s}}dy + \int_{B(x,\delta(x))}  \dfrac{u(x)-u(y) - \nabla u(x) \cdot (x-y)}{|x-y|^{N+2s}}dy.
\end{align*}
Here the constant $a_{N,s}$ is omitted because it does not affect the argument. Since
$$u(x) - u(y) = \nabla u(x)\cdot (x-y) + \frac{1}{2} D^2 u(\xi) (x-y) \cdot (x-y) \text{ and } |\nabla u(x)| \leq M\delta(x)$$ 
for some $\xi$ between $x$ and $y$, and for positive constant $M > 0$, we deduce that
\begin{align*}
|(-\Delta)^s_{\CFL} u(x)| &\lesssim \int_{\Omega \backslash B(x,\delta(x))} \dfrac{\delta(x)|x-y| + \norm{D^2 u}_{L^{\infty}(\Omega)}|x-y|^2}{|x-y|^{N+2s}} dy \\ 
&+  \int_{B(x,\delta(x))} \dfrac{\norm{D^2 u}_{L^{\infty}(\Omega)}|x-y|^2}{|x-y|^{N+2s}}dy \\ &\lesssim \delta(x) \int_{\delta(x)}^{d_{\Omega}} t^{-2s} dt  + \int_0^{d_{\Omega}}t^{1-2s} dt \\  &\lesssim 1+ \delta(x)^{2-2s},
\end{align*}
where $d_{\Omega}$ denotes the diameter of $\Omega$. Thus $(-\Delta)^s_{\CFL} u \in L^{\infty}(\Omega) \subset L^2(\Omega)$. Moreover, for any $u,v \in C_c^\infty(\Omega)$,
$$\int_{\Omega} v (-\Delta)^s_{\CFL} u dx =  \int_{\Omega} u (-\Delta)^s_{\CFL} v dx = \dfrac{a_{N,s}}{2} \int_{\Omega} \int_{\Omega} \dfrac{(u(x)-u(y))(v(x)-v(y))}{|x-y|^{N+2s}}dx dy,$$
which, together with  the Hardy inequality with $s > \frac{1}{2}$ (see \cite[(1.10)]{CheSon_2003})
$$\int_{\Omega} |u(x)|^2 dx \lesssim \int_{\Omega} \dfrac{|u(x)|^2}{\delta(x)^{2s}}dx \lesssim \int_{\Omega} \int_{\Omega} \dfrac{|u(x)-u(y)|^2}{|x-y|^{N+2s}}dxdy,\quad u \in C^{\infty}_c(\Omega), $$
yields \eqref{eq_L1} and \eqref{eq_L2}. It can be seen from \cite{BogBurChe_2003,CheKimSon_2009} that $\Hbb(\Omega) = H^s_0(\Omega) = H^s_{00}(\Omega)$ for $s> \frac{1}{2}$, whence \eqref{eq_L3} is satisfied.  

Next, \eqref{eq_G1} follows from Lemma \ref{lem:infinitesimal} (see also \cite{Che_2018}). We infer from \cite{BogBurChe_2003} that \eqref{eq_G2} holds with $\gamma = 2s - 1$.  Finally, assumption \eqref{eq_G3} becomes $N \geq 4s$.

We emphasize that the class of operators under consideration also includes other types of operators such as the Laplacian perturbed by a Hardy potential, the sum of two RFLs, relativistic Schr\"odinger operators and operators constructed from the interpolation between RFL and SFL. The mentioned operators will be discussed in Section \ref{sec:other-examples}.
\section{Formulation of problems, features and main results}\label{sec:result}

\noindent \textbf{Formulation of the boundary value problem.} We start with boundary value problems for linear equations driven by $\L$ with measure data of the form
\begin{equation}\label{eq:Lmu} 
\left\{ \begin{aligned}
\L u &= \mu &&\text{ in }\Omega, \\
u &=0 &&\text{ on } \partial \Omega, \\
u &=0 &&\text{ in }\R^N \backslash \overline{\Omega} \text{ (if applicable)},
\end{aligned} \right.
\end{equation}
where $\mu$ is a Radon measure on $\Omega$. 

For a function $f \in L^1_{\loc}(\Omega)$, the Green operator acting on $f$ is defined as
	$$
	\Gbb^{\Omega}[f](x) := \int_{\Omega} G^{\Omega}(x,y)f(y)dy,\quad  x \in \Omega.
	$$
\begin{definition}[Weak-dual solutions of linear problems] \label{def:weak-dual-linear}
Assume that $\mu \in \M(\Omega,\delta^{\gamma} )$. A function $u$ is a weak-dual solution, or $\Gbb^{\Omega} [\delta^{\gamma} L^{\infty}(\Omega)]$-weak solution, of problem \eqref{eq:Lmu} if $u \in L^1(\Omega,\delta^{\gamma} )$ and it satisfies
	\begin{equation} \label{weakdual-formula}
	\int_{\Omega} u\xi dx = \int_{\Omega} \Gbb^{\Omega} [\xi] d\mu,\quad \forall \xi \in \delta^{\gamma}L^{\infty}(\Omega).
	\end{equation}
\end{definition}

The notion of weak-dual solutions was first introduced by Bonforte and V\'azquez \cite{BonVaz_2015}, then became more prevalent and was developed in many papers (see e.g. \cite{BonFigVaz_2018,Aba_2019,ChaGomVaz_2019_2020,chan2020singular}) because it does not require to specify the meaning of $\L \xi$ for an arbitrary smooth function $\xi$. This notion is equivalent to other notions of solutions given in \cite[Theorem 2.1]{ChaGomVaz_2019_2020} for which the space of test functions $\delta^{\gamma}L^{\infty}(\Omega)$ is replaced by $L^{\infty}_c(\Omega)$ or $C^{\infty}_c(\Omega)$. We note that the homogeneous boundary condition in \eqref{eq:Lmu} is encoded in the weak-dual formulation \eqref{weakdual-formula} which can be better interpreted under additional relation between the exponents $s$ and $\gamma$ (see Subsection \ref{sec:weakduallinear} for further description of the boundary condition in the case $\gamma \geq s - \frac{1}{2}$). 

For a Radon measure $\mu$ on $\Omega$,  put
\begin{equation} \label{eq:Greenop} 
\Gbb^{\Omega}[\mu](x):=  \int_{\Omega} G^{\Omega}(x ,y)d\mu(y),\quad x \in \Omega
\end{equation}

Note that $\Gbb^{\Omega}[\mu](x)$ is finite for a.e. $x \in \Omega$ if and only if $\mu \in \M(\Omega,\delta^\gamma )$ (see Lemma \ref{lem:Gfinite}). Therefore, $\M(\Omega,\delta^\gamma )$ is the optimal class of measure data for problem \eqref{eq:Lmu}.

Moreover, it can be seen from weak formulation \eqref{weakdual-formula} and  Lemma \ref{lem:integrationbypart} that, for any $\mu \in \M(\Omega,\delta^\gamma )$, $\Gbb^{\Omega}[\mu]$ is the unique weak-dual solution of problem \eqref{eq:Lmu}. Hence in the context of weak-dual solutions, problem \eqref{eq:Lmu} is well posed. Since the solution of \eqref{eq:Lmu} can be represented in term of $\Gbb^{\Omega}$, in order to study problem \eqref{eq:Lmu}, it is sufficient to investigate $\Gbb^\Omega$. Main properties of $\Gbb^\Omega$ are given in Section \ref{sec:Green}.

The theory of linear equation \eqref{eq:Lmu} provides a basis in the study of nonnegative solutions to the semilinear equation
\begin{equation}\label{eq_nonlinear_source} \left\{ \begin{aligned}
\L u &= u^p + \lambda \mu &&\text{ in }\Omega, \\
u  &= 0 &&\text{ on }\partial \Omega, \\
u &=0 &&\text{ in }\R^N \backslash \overline{\Omega} \text{ (if applicable)},
\end{aligned} \right.
\end{equation}
where $p>1$, $\lambda$ is a positive parameter and $\mu \in \M^+(\Omega,\delta^\gamma )$ with $\| \mu \|_{\M(\Omega,\delta^\gamma )}=1$. We study \eqref{eq_nonlinear_source} by employing the concept of weak-dual solutions introduced above.

\begin{definition} \label{def:weak-dual-nonlinear}
Suppose that $\mu \in \M(\Omega,\delta^{\gamma})$. A nonnegative function $u$ is a weak-dual solution of \eqref{eq_nonlinear_source} if $u \in L^p(\Omega,\delta^{\gamma} )$ and it satisfies
\begin{equation} \label{weak-dual-nonlinear} 
\int_{\Omega} u\xi dx = \int_{\Omega}u^p \Gbb^{\Omega}[\xi] dx + \lambda \int_{\Omega} \Gbb^{\Omega} [\xi] d\mu, \quad \forall \xi \in \delta^{\gamma}L^{\infty}(\Omega).
\end{equation}
\end{definition}

\noindent \textbf{Features.} Problem \eqref{eq_nonlinear_source} has the following main features.

\begin{itemize}
\item The operator $\L$ is taken in a large class of local and nonlocal operators. Although several significant results for linear equations involving operator $\L$ have been recently obtained in function settings or measure frameworks (\cite{Aba_2019,ChaGomVaz_2019_2020}), sharp estimates on the Green operator associated to $\L$ and regularization results in weighted measure framework are still lacking. 
\item The equation in \eqref{eq_nonlinear_source} is nonlinear, with a superlinear source term $u^p$. Heuristically, the solvability for \eqref{eq_nonlinear_source} depends not only on the value of the exponent $p$ but also depends on the value of the parameter $\lambda$. 
\item The datum $\mu$ is chosen in weighted measure space $\M(\Omega,\delta^\gamma )$ which is the optimal class of data. The presence of  the weight $\delta^\gamma$ enlarges the value of the critical exponent for the existence for \eqref{eq_nonlinear_source} and hence complicates the investigation.
\end{itemize}

The interplay between the features yields substantial difficulties and requires a fine analysis. \medskip

\noindent \textbf{Main results.} 

Our first result shows that there exist a \textit{critical exponent} given by
\begin{equation} \label{pgamma} p^* = p^*(N,\gamma,s):= \frac{N+\gamma}{N+\gamma-2s}
\end{equation} 
and a \textit{threshold value} of the parameter for the existence and nonexistence of problem \eqref{eq_nonlinear_source}.

\begin{theorem}[Existence and nonexistence] \label{th:main1} Assume that \eqref{eq_G1}--\eqref{eq_G2} hold and $p>1$. 
	
\noindent {\sc 1. Subcritical case: $1<p<p^*$.} Let $\mu \in \M^+(\Omega,\delta^{\gamma} )$ with $\| \mu \|_{\M(\Omega,\delta^{\gamma})} = 1$. Then there exists a threshold value $\lambda^*  \in (0,+\infty)$ such that the followings hold.
	
	(i) If $\lambda \in (0,\lambda^*]$ then problem \eqref{eq_nonlinear_source} has a minimal positive weak-dual solution $\underline{u}_{\lambda}$. Moreover, there holds
	\begin{equation} \label{estimate-minimalsol} \lambda\Gbb^{\Omega}[\mu] \leq  \underline u_{\lambda} \le T(\lambda)\Gbb^{\Omega} [\mu] \quad  \text{a.e. in } \Omega,
	\end{equation}
	where $T$ is an increasing function and $T(\lambda) \to 0$ as $\lambda \to 0$.
	
	(ii) If $\lambda \in (\lambda^*,+\infty)$ then problem \eqref{eq_nonlinear_source} does not admit any nonnegative weak-dual solution.
	
\noindent {\sc 2. Supercritical case: $p \geq p^*$.} Let $\lambda>0$. There exists  $\mu \in \M^+(\Omega,\delta^{\gamma} )$ with $\| \mu \|_{\M(\Omega,\delta^{\gamma})} = 1$ such that problem \eqref{eq_nonlinear_source} does not admit any nonnegative weak-dual solution. 
\end{theorem}

The existence of the minimal solution $\underline{u}_\lambda$ in the subcritical case $p \in (1,p^*)$ is based on the sub- and supersolution theorem established in Proposition \ref{prop_existence.nonlinear} for the framework of operator $\L$ and weak-dual solutions. The key step in the argument is the construction of a positive weak-dual supersolution of the form $T(\lambda)\Gbb^\Omega[\mu]$, which is obtained due to a subtle estimate on the Green kernel (see Proposition \ref{prop_priorestimate}). In the supercritical case $p \geq p^*$, to prove the nonexistence result, we make use of the fact that any positive weak-dual solution of problem \eqref{eq_nonlinear_source}, if exists, is bounded from below by $\lambda\Gbb^\Omega[\mu]$. Moreover, we take a careful care of the interaction between two terms appearing in the estimate \eqref{G-est} for the Green kernel: $|x-y|^{2s-N}$ (which describes the interior singularity of the Green kernel) and  $\delta(x)^\gamma$ (which characterize the behavior of the Green kernel near the boundary). By localizing and  `displacing' the singularity toward the boundary $\partial \Omega$, we are able to show that, when $p \geq p^*$, a positive weak-dual solution, if exists, would not belong to $L^p({\mathcal C}, \delta^\gamma )$ where ${\mathcal C} \subset \Omega$  is a cone of vertex at a point on $\partial \Omega$. This leads to a contradiction and hence we obtain the nonexistence result in the supercritical case. 

In order to go further in the investigation of the structure of solution set of \eqref{eq_nonlinear_source} in the subcritical case $p \in (1,p^*)$, we need additionally assumptions \eqref{eq_L1}--\eqref{eq_L3} and \eqref{eq_G1}--\eqref{eq_G3}. In the next result, we show a multiplicity result when $\lambda  \in (0,\lambda^*)$ and the uniqueness when  $\lambda = \lambda^*$.

\begin{theorem}[Uniqueness and multiplicity] \label{th:main2} Assume that \eqref{eq_L1}--\eqref{eq_L3} and \eqref{eq_G1}--\eqref{eq_G3} hold. Let $p \in (1,p^*)$ and $\mu \in \M^+(\Omega,\delta^{\gamma} )$ with $\| \mu \|_{\M(\Omega,\delta^{\gamma})} = 1$.  

{\sc 1. Uniqueness.} If $\lambda = \lambda^*$, then problem \eqref{eq_nonlinear_source} admits a unique positive weak-dual solution.

{\sc 2. Multiplicity.} If $\lambda \in (0, \lambda^*)$, then problem  \eqref{eq_nonlinear_source} admits a second positive weak-dual solution $\widetilde{u}_{\lambda} >  \underline{u}_{\lambda}$. Moreover, there exists a positive constant $C=C(\Omega, N,s,\gamma,p,\lambda_1)$ such that
\begin{equation} \label{estimate-secondsol} \lambda\Gbb^{\Omega}[\mu] \leq  \widetilde u_{\lambda} \le C(\Gbb^{\Omega} [\mu] +1) \quad  \text{a.e. in } \Omega.
\end{equation}
\end{theorem}
Here, $\lambda_1$ in Theorem \ref{th:main2} is the first eigenvalue of $\L$, see Section \ref{sec:lineartheory}. Let us discuss the main idea of the proof. A conspicuous feature of present paper in comparison with previous works  is the deep involvement of the weight $\delta^\gamma$ appearing in the optimal class of data $\M(\Omega,\delta^\gamma)$, which complicates the analysis.  Under assumptions \eqref{eq_L1}--\eqref{eq_L3} and \eqref{eq_G1}--\eqref{eq_G3}, we develop a new unifying technique which enables us to obtain a pivotal regularity result for weak-dual solutions to the following nonhomogeneous linear equations
\begin{equation}\label{nonhomogeneous} \left\{ \begin{aligned}
\L v &= a v + f &&\text{ in }\Omega, \\
v  &= 0 &&\text{ on }\partial \Omega, \\
v &=0 &&\text{ in }\R^N \backslash \overline{\Omega} \text{ (if applicable)},
\end{aligned} \right.
\end{equation}
where $a \in L^r(\Omega,\delta^{\gamma} ), r > \frac{N+\gamma}{2s}$ and $f \in L^m(\Omega,\delta^{\gamma} ), m > \frac{N+\gamma}{2s}$ (Notice that $\frac{N+\gamma}{2s} = (p^*)'$, where $p^*$ is the critical exponent given in \eqref{pgamma}). More precisely, if $v$ is a weak-dual solution of \eqref{nonhomogeneous} then $v \in L^\infty(\Omega)$ (see Proposition \ref{prop_bootstrap2}). This is a crucial ingredient in the derivation of the (semi)stability of the minimal solution $\underline{u}_\lambda$ when $N \geq N_{s,\gamma}$ (see Propositions \ref{prop_stability} and \ref{prop_eigenvalue1}). The technique is based on a bootstrap argument, combined with delicate estimates on the Green kernel and does not require any weighted Sobolev embeddings which are not available for general operators. In the case $\lambda=\lambda^*$, this regularity result implies that the minimal weak-dual solution $\underline{u}_{\lambda^*}$ is semi-stable and  the eigenvalue $\sigma(\lambda^*)$ of the eigenproblem with weight $\underline{u}_{\lambda^*}^{p-1}$ satisfies $\sigma(\lambda^*)=1$ (see Proposition \ref{prop_eigenvalue1}), which enables us to prove the uniqueness. In the case $\lambda \in (0,\lambda^*)$, to construct the second solution, we reduce problem \eqref{eq_nonlinear_source} to a related problem (see problem \eqref{eq_mountainpass}) with a variational structure for which the mountain-pass theorem can be applied thanks to the above-mentioned regularity. 

We notice that when $\gamma<2s$,  condition $N \geq 4s$ is required to ensure that solutions of \eqref{nonhomogeneous} belong to $\Hbb(\Omega)$ and can be relaxed when the datum $\mu$ in problem \eqref{eq_nonlinear_source} is more regular, for example in $L^r(\Omega, \delta^\gamma )$ for $r>\frac{N+\gamma}{2s}$. When $\gamma \geq 2s$, condition $N \geq 4s(1+\gamma)-\gamma$ is satisfied if $N \geq 3$.

It is worth adding that we also establish sharp estimates for the Green kernel and obtain regularization results in weighted function settings and weighted measure frameworks.

We stress that our results are new, even in the case $\L$ is the classical Laplacian, and cover, complement or extend several results in \cite{BidYar_2002,CheQua_2018,Aba_2019,ChaGomVaz_2019_2020}. Our theory is valid for various types of local and nonlocal operators.

\textbf{Organization of the paper.} The rest of the paper is organized as follows. In Section \ref{sec:Green}, we establish sharp estimates regarding the Green kernel and regularization results in weighted function settings and measure frameworks. In Section \ref{sec:lineartheory}, we study the linear problem associated to $\L$ in a measure framework and in a variational setting. Section \ref{sec:minimalsolution} is devoted to the proof of Theorem \ref{th:main1}. In Section \ref{sec:keyregularity}, we obtain a key regularity result which plays an important role in making the connection between the context of weak-dual solutions and the variational setting. This in turn enables us to show the semistability or stability of the minimal solution to \eqref{eq_nonlinear_source}. The proof of Theorem \ref{th:main2} is represented in Section \ref{sec:uni+multi}. Finally, we discuss further examples in Section \ref{sec:other-examples}.

\section{Green kernel} \label{sec:Green}
In this section, we assume that \eqref{eq_G1}--\eqref{eq_G2} hold.
\subsection{Basic properties} In this subsection, we prove some basic properties of the Green kernel in measure frameworks which are a counterpart of results for function settings established in \cite{Aba_2019,ChaGomVaz_2019_2020}. We start with providing upper bounds for the Green kernel. 
\begin{lemma} There exists a positive constant $C=C(N,\Omega,s,\gamma)$ such that
 for any $x,y \in \Omega, x \neq y$, there holds
 \begin{equation}\label{eq_Green0}
 G^{\Omega}(x,y) \lesssim C \min \left\{ \dfrac{1}{|x-y|^{N-2s}},\dfrac{\delta(y)^{\gamma}}{\delta(x)^{\gamma}}\dfrac{1}{|x-y|^{N-2s}} , \dfrac{\delta(y)^{\gamma}}{|x-y|^{N-2s+\gamma}},\dfrac{\delta(y)^{\gamma}\delta(x)^{\gamma}}{|x-y|^{N-2s+2\gamma}} \right\}.
 \end{equation}	
\end{lemma}
\begin{proof}
We infer from \eqref{eq_G2} that, for any $x,y \in \Omega, x \neq y$,
\begin{equation}\label{eq_equivalentgreen}
\begin{aligned}
G^{\Omega}(x,y) &\sim \dfrac{1}{|x-y|^{N-2s}}\left(1 \wedge \dfrac{\delta(x)\delta(y)}{|x-y|^2} \right)^{\gamma}\\
&\sim \dfrac{1}{|x-y|^{N-2s}}\cdot\dfrac{\delta(x)^{\gamma}\delta(y)^{\gamma}}{|x-y|^{2\gamma} \vee \delta(x)^{2\gamma} \vee \delta(y)^{2\gamma}} \\ 
&\sim \dfrac{1}{|x-y|^{N-2s}}\cdot\dfrac{\delta(x)^{\gamma}\delta(y)^{\gamma}}{|x-y|^{2\gamma} +\delta(x)^{2\gamma} + \delta(y)^{2\gamma} }.
\end{aligned}
\end{equation}
This implies \eqref{eq_Green0}. 
\end{proof}

Let $\Gbb^{\Omega}$ be the Green operator defined in \eqref{eq:Greenop}. 

\begin{lemma} \label{lem:Gfinite} $\Gbb^{\Omega}[\mu](x)$ is finite for a.e. $x \in \Omega$ if and only if $\mu \in \M(\Omega,\delta^\gamma )$.
\end{lemma}
\begin{proof} Without loss of generality, it is sufficient to prove this lemma for a positive measure $\mu$. 
	
Assume $\mu \in \M^+(\Omega,\delta^\gamma )$. We infer from \cite[Theorem 2.11]{Aba_2019} that $\Gbb^{\Omega}[\mu] \in L^1_{\loc}(\Omega)$, whence $\Gbb^{\Omega}[\mu](x)$ is finite for a.e. $x \in \Omega$.

Conversely, assume $\Gbb^{\Omega}[\mu](x)$ is finite for a.e. $x \in \Omega$. Let $x_0 \in \Omega$ such that $\Gbb^{\Omega}[\mu](x_0)<+\infty$. From \eqref{eq_equivalentgreen}, we derive that there exists $c=c(N,s,\gamma,\Omega)$ such that
$$ G^{\Omega}(x,y) \geq c\delta(x)^\gamma \delta(y)^\gamma, \quad \forall x,y \in \Omega, x\neq y.
$$
It follows that
\begin{equation} \label{eq:lowerHopf}
\Gbb^{\Omega}[\mu](x) \geq c\delta(x)^{\gamma} \int_{\Omega}\delta(y)^\gamma d\mu(y), \quad x \in \Omega.
\end{equation}
Taking $x=x_0$ in the above estimate yields
$$  \int_{\Omega}\delta(y)^\gamma d\mu(y) \leq  c\frac{\Gbb^{\Omega}[\mu](x_0)}{\delta(x_0)^\gamma} < + \infty,
$$ 
whence $\mu \in \M^+(\Omega,\delta^\gamma )$. The proof is complete.
\end{proof}

We remark that estimate \eqref{eq:lowerHopf} can be regarded as a lower Hopf estimate in measure framework. Moreover, Lemma \ref{lem:Gfinite} indicates that $\M(\Omega,\delta^\gamma )$ is the optimal class of measure data for those operators $\Gbb^{\Omega}$ satisfying \eqref{eq_G1} and \eqref{eq_G2}.

The following results were given in \cite{ChaGomVaz_2019_2020} for integrable data and can be easily adapted to measure frameworks,  therefore we omit the proofs.

\begin{lemma}[Maximum principle]\label{lem_maximum} Assume that $\mu \in \M^+(\Omega,\delta^{\gamma} )$. Then either $\Gbb^{\Omega} [\mu] > 0$ a.e. in $\Omega$ or $\Gbb^{\Omega} [\mu] \equiv 0$.
\end{lemma}

\begin{lemma}[Integration by parts formula with $\Gbb^{\Omega}$]\label{lem:integrationbypart} Assume that $\mu \in \M(\Omega,\delta^{\gamma} )$ and $\xi \in \delta^{\gamma}L^{\infty}(\Omega)$. Then we have
	$$\int_{\Omega} \Gbb^{\Omega} [\xi] d\mu = \int_{\Omega} \xi \Gbb^{\Omega}[\mu] dx.$$
\end{lemma}

\subsection{Estimates on the Green kernels}
We next establish some estimates for the Green kernel. The following results were proved in \cite{Aba_2019}.
\begin{proposition}[{\cite[Theorem 3.9]{Aba_2019}}] \label{prop_L1estimate} Suppose that $\alpha > (\gamma-2s) \vee (-\gamma-1)$. Then the mapping $\Gbb^{\Omega}\colon L^1(\Omega,\delta^{\gamma} ) \longrightarrow L^1(\Omega,\delta^{\alpha})$
is well-defined and continuous. In particular, if $\gamma < 2s$ then
$\Gbb^{\Omega} \colon L^1(\Omega,\delta^{\gamma} ) \longrightarrow L^1(\Omega)$ is continuous.
\end{proposition}
Some improvements of the above result have been achieved  for the SFL in \cite{AbaDup_2017}, for RFL in \cite{CheVer_2014} and  for a more general class of nonlocal operators in function or measure settings. The next result provides sharp estimates for the Green operator which covers or improves the above-mentioned results. Our approach involves weighted Marcinkiewicz spaces introduced in Subsection \ref{sec:LPspace}. Put
\begin{equation} \label{p-alpha}
p_{\gamma,\alpha}:= \frac{N+\alpha}{N+\gamma-2s}.
\end{equation}
Notice that if $\alpha>\gamma - 2s$ then $p_{\gamma,\alpha}>1$. In particular, $p^* = p_{\gamma,\gamma} > 1$.

\begin{proposition}\label{prop_marcinkiewicz}Let $\alpha$ be such that
$$\max\left\{-\gamma-1,\gamma - 2s, - \frac{\gamma N}{N-2s+\gamma} \right\} < \alpha < \frac{\gamma N}{N-2s}.$$	
Then $\Gbb^{\Omega} : \M(\Omega,\delta^{\gamma} ) \longrightarrow M^{p_{\gamma,\alpha}}(\Omega,\delta^{\alpha} )$
is continuous, namely there exists a positive constant $C = C(\Omega,N,s,\gamma,\alpha)$ such that
\begin{equation} \label{Green-meas}
\| \Gbb^{\Omega}[\mu] \|_{M^{p_{\gamma,\alpha}}(\Omega,\delta^{\alpha} )} \leq C\| \mu  \|_{\M(\Omega,\delta^{\gamma} )}, \quad \forall \, \mu \in \M(\Omega,\delta^{\gamma} ).
\end{equation}
\end{proposition}

The proof of this result relies on the following key lemma.
\begin{lemma}[{\cite[Lemma 2.4]{BidViv_2000}}]\label{lem_marcin}
Let $\nu$ be a nonnegative bounded Radon measure on $D= \Omega$ or $\partial \Omega$ and $\eta \in C(\Omega)$ a positive weight function. Let $E$ be a continuous nonnegative function on $\left\{ (x,y) \in \Omega \times D: x \neq y \right\}$. Set
$$ {\mathbb E}[\nu](x): = \int_D E(x,y)d\nu(y), \quad x \in \Omega.
$$
For any $\lambda > 0$ and $y \in D$, we set
$$A_{\lambda}(y) := \left\{ x \in \Omega \backslash \{ y\}: E(x,y) > \lambda \right\} \quad \text{and}  \quad m_{\lambda}(y):= \int_{A_{\lambda}(y)} \eta dx.$$
Suppose that there exist $q>1$ and  $\bar C> 0$ such that
$$m_{\lambda}(y) \leq \bar  C \lambda^{-q}\quad \text{for all } y \in D \text{ and } \lambda > 0.$$
Then ${\mathbb E}[\nu] \in M^q(\Omega, \eta )$ and there holds
$$\norm{\mathbb{E}[\nu]}_{M^p(\Omega, \eta )} \le \left(1 + \frac{q}{q-1}\bar C \right) \norm{\nu}_{\M(D)}.$$
\end{lemma} 

\begin{proof}[\sc Proof of Proposition \ref{prop_marcinkiewicz}] We write
$$\Gbb^{\Omega} [\mu](x) = \int_{\Omega} \frac{G^{\Omega} (x,y)}{\delta(y)^{\gamma}} \delta(y)^{\gamma} d \mu(y), \quad x \in \Omega.$$
We will apply Lemma \ref{lem_marcin} with $D=\Omega$, $\eta=\delta^\alpha$, $d\nu=\delta^\gamma d\mu$ and $E(x,y)=\dfrac{G^{\Omega} (x,y)}{\delta(y)^{\gamma}}$.	
	
For $\lambda>0$ and $y \in \Omega$, denote
$$A_{\lambda}(y):= \left\{ x \in \Omega \backslash \{ y\}: \dfrac{G^{\Omega} (x,y)}{\delta(y)^{\gamma}} > \lambda \right\} \quad \text{and} \quad m_\lambda(y):=\int_{A_\lambda(y)}\delta^\alpha dx.$$
In order to estimate $m_\lambda(y)$, we consider successively two cases.

\textbf{Case 1:} $0 \le \alpha < \frac{\gamma N}{N-2s}$. 
By \eqref{eq_Green0}, for every $x \in A_{\lambda}(y)$, we have
$$\lambda <  \dfrac{G^{\Omega} (x,y)}{\delta(y)^{\gamma}} \le  \dfrac{C}{\delta(x)^{\gamma}|x-y|^{N-2s}} \quad \text{ and } \quad \lambda <  \dfrac{G^{\Omega} (x,y)}{\delta(y)^{\gamma}} \le \dfrac{C}{|x-y|^{N-2s+\gamma}} $$
for some $C = C(\Omega,N,s,\gamma) > 0$, which implies
$$\delta(x)^{\gamma} < C \lambda^{-1} |x-y|^{2s-N} \quad \text{ and } \quad |x-y| < C \lambda^{-\frac{1}{N-2s+\gamma}}.$$
Thus,
\begin{align*}
m_{\lambda} (y) = \int_{A_{\lambda}(y)} \delta(x)^{\alpha} dx &\le C \int_{ |x-y| \le C \lambda^{-\frac{1}{N-2s+\gamma}}} \left(\lambda^{-1} |x-y|^{2s-N} \right)^{\frac{\alpha}{\gamma}} dy \\
& \le  C \lambda^{-\frac{\alpha}{\gamma}}  \int_0^{\lambda^{-\frac{1}{N-2s+\gamma}}} z^{(2s-N)\frac{\alpha}{\gamma} + N-1} dz.
\end{align*}
Since $\gamma - 2s < \alpha < \frac{\gamma N}{N-2s}$, it follows that
$$(2s-N)\frac{\alpha}{\gamma} + N-1  > -1 \quad \text{and} \quad \dfrac{N+\alpha}{N+\gamma - 2s } > 1.$$
Therefore we obtain $m_{\lambda} (y) \le C \lambda^{-p_{\gamma,\alpha}}$.

\textbf{Case 2:} $\max\left\{ -\frac{\gamma N}{N+\gamma-2s}, \gamma - 2s,-\gamma-1\right\} < \alpha < 0$. Again by \eqref{eq_Green0}, for $x \in A_{\lambda}(y)$,
$$ \lambda < \dfrac{G^{\Omega}(x,y)}{\delta^{\gamma}(y)} \le C \dfrac{\delta(x)^{\gamma}}{|x-y|^{N-2s+2\gamma}},$$
which implies $\delta(x)^{\gamma} > C \lambda |x-y|^{N-2s+2\gamma}$.
From this, since $\alpha < 0$,
\begin{align*}
m_{\lambda} (y) = \int_{A_{\lambda}(y)} \delta(x)^{\alpha} dx & \le \int_{|x-y| \le C \lambda^{-\frac{1}{N-2s+\gamma}}} (\lambda|x-y|^{N-2s+2\gamma})^{\frac{\alpha}{\gamma}}dx \\
&\le C\lambda^{\frac{\alpha}{\gamma}} \int_0^{\lambda^{-\frac{1}{N-2s+\gamma}}} z^{\frac{\alpha}{\gamma}(N+2\gamma-2s)+N-1} dz  \\
&\le C \lambda^{-\frac{N+\alpha}{N-2s+\gamma}} = C \lambda^{-p_{\gamma,\alpha}}.
\end{align*}

In all cases, one has $$m_{\lambda} (y) \le C \lambda^{-p_{\gamma,\alpha}} \text{ for any }y \in \Omega \text{ and }\lambda>0.$$
Using Lemma \ref{lem_marcin} we derive that $\Gbb^{\Omega} [\mu] \in M^{p_{\gamma,\alpha}}(\Omega,\delta^{\gamma} )$ and \eqref{Green-meas} holds. We complete the proof.
\end{proof}

We note that our result is optimal in the sense that for $\mu \in \M(\Omega,\delta^{\gamma})$, $\Gbb^{\Omega}[\mu] \in M^{p^*}(\Omega,\delta^{\gamma})$ but in general $\Gbb^{\Omega}[\mu] \notin L^{p^*}(\Omega,\delta^{\gamma})$.

As a direct consequence of Proposition \ref{prop_marcinkiewicz}, we obtain the following result. We recall that $p^*$ is the exponent given in \eqref{pgamma}.
\begin{corollary} \label{cor:GLL} 

(i) For any $q \in [1,p^*)$,  $\Gbb^{\Omega}: L^1(\Omega,\delta^{\gamma} ) \longrightarrow L^q(\Omega,\delta^{\gamma} )$ is continuous, namely there exists a positive constant $C=C(N,\Omega,s,\gamma,q)$ such that
\begin{equation} \label{Gr-est-1}
\| \Gbb^{\Omega}[f] \|_{L^q(\Omega,\delta^\gamma )} \leq C \| f \|_{L^1(\Omega,\delta^\gamma )}, \quad \forall f \in L^1(\Omega,\delta^\gamma ).
\end{equation}

(ii) Suppose $\gamma < 2s$. For any $q \in \left[1,\frac{N}{N+\gamma-2s}\right)$, 
$\Gbb^{\Omega}: L^1(\Omega,\delta^{\gamma} ) \longrightarrow L^q(\Omega)$ is continuous, namely there exists a positive constant $C=C(N,\Omega,s,\gamma,q)$ such that
\begin{equation} \label{Gr-est-2}
\| \Gbb^{\Omega}[f] \|_{L^q(\Omega)} \leq C \| f \|_{L^1(\Omega,\delta^\gamma )}, \quad \forall f \in L^1(\Omega,\delta^\gamma ).
\end{equation}
\end{corollary}
\begin{proof}
The proof of statement (i) follows directly from Proposition \ref{prop_marcinkiewicz} and the continuous embeddings $L^1(\Omega,\delta^\gamma ) \hookrightarrow \M(\Omega,\delta^\gamma )$, $M^{p^*}(\Omega,\delta^\gamma ) \hookrightarrow L^q(\Omega,\delta^\gamma )$ due to $q<p^*$. The proof of statement (ii) is similar and we omit it. The condition $\gamma<2s$ in statement (ii) is required to ensure that $\frac{N}{N+\gamma -2s}>1$.
\end{proof}

Another useful consequence of Proposition \ref{prop_marcinkiewicz} is the following delicate estimate, which will be employed in the derivation of regularity results.
\begin{corollary} \label{G/delta} Assume $q \in [1,p^*)$. There exists a constant $C=C(\Omega,N,s,\gamma,q)$ such that
\begin{equation} \label{cor-Mar}
\int_{\Omega} \left( \frac{G^{\Omega}(x,y)}{\delta(y)^\gamma} \right)^q \delta(x)^\gamma dx \leq C, \quad \forall y \in \Omega.
\end{equation}
\end{corollary}
\begin{proof} Let  $y \in \Omega$.  In Proposition \ref{prop_marcinkiewicz}, taking $\mu = \delta_y \in \M(\Omega) \subset \M(\Omega,\delta^{\gamma})$ and $\alpha=\gamma$, one has 
$$\norm{ \Gbb^{\Omega} [\delta_y]}_{M^{p^*}(\Omega,\delta^{\gamma} )} \le C(\Omega,N,s,\gamma) \norm{\delta_y}_{\M(\Omega,\delta^{\gamma} )},$$
which, together with the continuous embedding $M^{p^*}(\Omega,\delta^\gamma ) \hookrightarrow L^q(\Omega,\delta^{\gamma})$ for $q < p^*$, implies
$$ \left( \int_{\Omega} G^{\Omega}(x,y)^q \delta(x)^{\gamma}dx\right)^{\frac{1}{q}} \leq C(\Omega,N,s,\gamma,q)\delta(y)^{\gamma}.$$
Thus \eqref{cor-Mar} follows. We complete the proof.
\end{proof}

\begin{remark} By a similar argument to that used in the proof of Proposition \ref{prop_marcinkiewicz}, we can show that, for any $\alpha \in \left(-\frac{N\gamma}{N+\gamma-2s}, 0\right]$, $\Gbb^{\Omega}: \M(\Omega) \longrightarrow M^{\frac{N+\alpha}{N-2s}}(\Omega,\delta^\gamma )$ is continuous. This extends the result in \cite[Proposition 3.2]{ChaGomVaz_2019_2020} to the framework of weighted spaces.
\end{remark}

The next propositions provide higher regularity results which are a weighted global counterpart of the local estimates in \cite[Proposition 3.6]{ChaGomVaz_2019_2020}.  

\begin{proposition}\label{prop_Linftyreg} For any $r > \frac{N+\gamma}{2s}$, $ \Gbb^{\Omega}: L^r(\Omega,\delta^\gamma ) \longrightarrow L^\infty(\Omega)$ is continuous, namely
\begin{equation} \label{eq_GLinfty} \left\| \Gbb^{\Omega}[f] \right\|_{L^{\infty}(\Omega)} \leq C\| f \|_{L^r(\Omega,\delta^\gamma )}, \quad \forall f \in L^r(\Omega,\delta^\gamma ).
\end{equation}
\end{proposition}
\begin{proof}
Let $f \in L^r(\Omega,\delta^\gamma )$. By H\"older's inequality, we have
\begin{equation} \label{Gf-1} \begin{aligned} 
\left|\Gbb^{\Omega}[f](x)\right| &= \left| \int_{\Omega} G^{\Omega}(x,y) f(y) dy \right| \\
&\le \left( \int_{\Omega} \left(\dfrac{G^{\Omega}(x,y)}{\delta(y)^{\frac{\gamma}{r}}}\right)^{r'} dy \right)^{\frac{1}{r'}}\left( \int_{\Omega} |f(y)|^r \delta(y)^{\gamma} dy \right)^{\frac{1}{r}}.
\end{aligned} \end{equation}
By \eqref{G-est} and since $r>1$, we deduce that  
$$G^{\Omega} (x,y) \le C \delta(y)^{\frac{\gamma}{r}}|x-y|^{-N+2s-\frac{\gamma}{r}}, \quad  \forall x,y \in \Omega, x \neq y.$$
Plugging it into \eqref{Gf-1} yields
\begin{align*}
\left|\Gbb^{\Omega}[f](x)\right|&\le C\left(\int_{\Omega} |x-y|^{-r'\left(N-2s+\frac{\gamma}{r}\right)}dy \right)^{\frac{1}{r'}}\left( \int_{\Omega} |f(y)|^r \delta(y)^{\gamma} dy \right)^{\frac{1}{r}} \le C\| f \|_{L^r(\Omega,\delta^\gamma )}.
\end{align*}
Here for  the second inequality, we have used the fact that $\int_{\Omega} |x-y|^{-r'\left(N-2s+\frac{\gamma}{r}\right)} dy < \infty$
since $r' \left(N-2s+\frac{\gamma}{r}\right) < N$. Therefore we obtain \eqref{eq_GLinfty}.
\end{proof}

\begin{proposition}\label{prop_higherregularity} Let $m \in \left(1,\frac{N+\gamma}{2s}\right)$ and $\varrho$ be given by $$ \frac{1}{\varrho}=\frac{1}{m}-\frac{2s}{N+\gamma}.$$ Then $\Gbb^{\Omega}: L^m(\Omega,\delta^\gamma ) \longrightarrow L^\varrho(\Omega,\delta^\gamma )$ is continuous, namely there exists a positive constant $C$ such that
\begin{equation} \label{eq_Grq} \norm{ \Gbb^{\Omega}[f] }_{L^\varrho(\Omega,\delta^\gamma )} \leq C\| f \|_{L^m(\Omega,\delta^\gamma )}, \quad \forall\, f \in L^m(\Omega,\delta^\gamma ).
\end{equation}
\end{proposition}

\begin{proof}

From Proposition \ref{prop_marcinkiewicz} and Proposition \ref{prop_Linftyreg}, we obtain that for any $r > \frac{N+\gamma}{2s}$ and $1 < q <p^*$,
$$
\begin{array}{lll}
\Gbb^{\Omega}: &L^1(\Omega,\delta^\gamma ) \longrightarrow L^q(\Omega,\delta^\gamma ) \\
& L^r(\Omega,\delta^\gamma ) \longrightarrow L^\infty(\Omega,\delta^{\gamma})
\end{array}
$$
are continuous. Here, $L^{\infty}(\Omega,\delta^{\gamma})$ denotes the space of essentially  bounded functions with respect to the measure $\delta^{\gamma} dx$, which satisfies $L^{\infty}(\Omega) \hookrightarrow L^{\infty}(\Omega,\delta^{\gamma})$. 

By Riesz-Thorin Interpolation Theorem (see, for instance, \cite{Tri_1978}), we derive that, for any $m \in (1,\frac{N+\gamma}{2s})$ and $\varrho$ is given by $\frac{1}{\varrho}=\frac{1}{m}-\frac{2s}{N+\gamma}$, the operator $\Gbb^{\Omega}: L^m(\Omega,\delta^\gamma ) \longrightarrow L^\varrho(\Omega,\delta^\gamma )$ is continuous, i.e. estimate \eqref{eq_Grq} holds. The proof is complete.
\end{proof}

Our next result is a 3G-inequality, whose proof follows line by line from the proof of \cite[Proposition 4.2]{CheSon_2002}.

\begin{proposition}[3G--inequality]\label{prop_3G} There exists a positive constant $C = C(\Omega,N,s,\gamma)$ such that for any pairwise different points  $x,y,z \in \Omega$, there holds 
\begin{align*}
\dfrac{G^{\Omega}(x,y)G^{\Omega}(y,z)}{G^{\Omega}(x,z)} \le C\left( \dfrac{\delta(y)^{\gamma}}{\delta(x)^{\gamma}} G^{\Omega} (x,y) + \dfrac{\delta(y)^{\gamma}}{\delta(z)^{\gamma}} G^{\Omega}(y,z)\right).
\end{align*}
In particular, one has
\begin{align*}
\dfrac{G^{\Omega}(x,y)G^{\Omega}(y,z)}{G^{\Omega}(x,z)} \le C\left( |x-y|^{2s-N} + |y-z|^{2s-N}\right).
\end{align*}
\end{proposition}

\begin{proposition} \label{prop_priorestimate} Suppose that $q \in \left[1, p^*\right)$. Then there exists a positive constant $C = C(\Omega,N,s,\gamma,q)$ such that for any $\mu \in \M^+(\Omega,\delta^{\gamma} )$ with $\norm{\mu}_{\M(\Omega,\delta^{\gamma} )} = 1$, there holds
	\begin{equation}\label{eq_5.Greenestimate}
	\Gbb^{\Omega} \left[ \Gbb^{\Omega} [\mu]^q \right] \le C \Gbb^{\Omega} [ \mu]\text{ a.e. in } \Omega.
	\end{equation}
\end{proposition}

\begin{proof}

Since $q \ge 1$, applying Jensen's inequality with respect to the measure $\delta^{\gamma}d\mu$ satisfying $\norm{\mu}_{\M(\Omega,\delta^{\gamma} )} = 1$, one has
\begin{align*}
\Gbb^{\Omega} [\mu ] (x)^q = \left[\int_{\Omega} \dfrac{G^{\Omega}(x,y)}{\delta(y)^{\gamma}}\delta(y)^{\gamma} d\mu(y) \right]^q &\le \int_{\Omega} \left( \dfrac{G^{\Omega}(x,y)}{\delta(y)^{\gamma}} \right)^q \delta(y)^{\gamma} d\mu(y)  \\
& =  \int_{\Omega} G^{\Omega}(x,y) \delta(y)^{\gamma(1-q)} d\mu(y).
\end{align*}
It follows that
\begin{equation} \label{eq_Green00}
\begin{aligned}
\Gbb^{\Omega} \left[\Gbb^{\Omega} [\mu]^q\right](x)  
&= \int_{\Omega} G^{\Omega} (x,z) \Gbb^{\Omega}[\mu] (z)^q dz \\
& \le \int_{\Omega} \int_{\Omega} G^{\Omega} (x,z) G^{\Omega}(y,z)^q \delta(y)^{\gamma(1-q)} d\mu(y) dz \\
& \le \int_{\Omega} \int_{\Omega} G^{\Omega}(x,z) G^{\Omega} (y,z) \left[ \dfrac{G^{\Omega} (y,z)}{\delta(y)^{\gamma}} \right]^{q-1} d\mu(y) dz.
\end{aligned}
\end{equation}
By \eqref{eq_Green0}, we have
\begin{equation}\label{eq_green1}
\frac{G^{\Omega} (y,z)}{\delta(y)^{\gamma}} \le \dfrac{C}{|y-z|^{N-2s+\gamma}},\quad \forall y,z \in \Omega, y \neq z.
\end{equation}
On the other hand, using 3G--inequality (Proposition \ref{prop_3G}), one has
\begin{equation}\label{eq_green2}
G^{\Omega} (x,z)G^{\Omega} (y,z) \le CG^{\Omega}(x,y) \left( |x-z|^{2s-N} + |y-z|^{2s-N} \right).
\end{equation}
Plugging \eqref{eq_green1} and \eqref{eq_green2} into \eqref{eq_Green00} leads to
\begin{equation} \label{eq_green1.1}  \Gbb^{\Omega} \left[\Gbb^{\Omega} [\mu]^q\right](x) \leq C \int_{\Omega}G^{\Omega}(x,y)I(x,y)d\mu(y),
\end{equation}
where 
\[ I(x,y):=  \int_{\Omega} \dfrac{|x-z|^{2s-N} + |y-z|^{2s-N}}{|y-z|^{(q-1)(N+\gamma-2s)}} dz. \]
Since $q < p^*$, we have $(N-2s)+(q-1)(N+\gamma-2s) < N$ and therefore
\begin{equation} \label{eq_green1.2} I(x,y) \le \int_{\Omega} \left(\dfrac{1}{|x-z|^{{(N-2s)+(q-1)(N+\gamma-2s)}}} + \dfrac{1}{|y-z|^{{(N-2s)+(q-1)(N+\gamma-2s)}}} \right) dz <\widetilde{C},
\end{equation}
where $\widetilde{C}$ is independent of $x,y$. Combining \eqref{eq_green1.1} and \eqref{eq_green1.2}, we conclude that \eqref{eq_5.Greenestimate} holds true. We complete the proof.
\end{proof}

\section{Linear equations}\label{sec:lineartheory}
\subsection{Weak-dual solutions}\label{sec:weakduallinear}

In this subsection, we assume that \eqref{eq_G1}--\eqref{eq_G2} hold. Boundary value problems with measures for linear equations involving usual fractional Laplacians have been extensively studied in the literature and recently investigated in a general framework, see \cite{Aba-PhD,BonVaz_2015,Aba_2019}.

Let $\mu \in \M(\Omega,\delta^{\gamma})$. Recall that a function $u$ is a weak-dual solution of \eqref{eq:Lmu} if $u \in L^1(\Omega,\delta^{\gamma} )$ and it satisfies
\begin{equation} \label{weak-dual-linear}\int_{\Omega} u\xi dx = \int_{\Omega} \Gbb^{\Omega} [\xi] d\mu, \quad \forall \xi \in \delta^{\gamma}L^{\infty}(\Omega).
\end{equation}
By Lemma \ref{lem:integrationbypart}, for any $\mu \in \M(\Omega,\delta^{\gamma} )$, the unique weak-dual solution of \eqref{eq:Lmu} is $\Gbb^{\Omega}[\mu]$.

The sense of the boundary condition in problem \eqref{eq:Lmu} depends on the relation between the exponents $s$, $\gamma$ and the smoothness of $\mu$. Below we interpret the sense of the boundary condition in case $\gamma \geq s - \frac{1}{2}$.

Let $W$ be a weight function defined by
\begin{equation} \label{eq:weight} W(x): = \left\{  \begin{aligned}
&\delta(x)^{2s-\gamma-1} \quad &&\text{if } \gamma>s-\frac{1}{2}, \\
&\delta(x)^{2s-\gamma-1}(1+|\ln\delta(x)|) &&\text{if } \gamma=s-\frac{1}{2}.
\end{aligned} \right.
\end{equation}

If $\mu=f \in C^{\infty}_c(\Omega)$ then the homogeneous boundary condition in \eqref{eq:Lmu} can be understood in the pointwise sense in connection with the weight $W$ as  
	\begin{equation} \label{G/W} \lim_{x \in \Omega, x \to \partial \Omega} \frac{\Gbb^\Omega[f](x)}{W(x)} = 0.
	\end{equation}
	Indeed, by \cite[Theorem 2.10]{Aba_2019}, $|\Gbb^{\Omega}[f](x)| \leq C\norm{f}_{L^{\infty}(\Omega)} \delta(x)^{\gamma}$ for every $x \in \Omega$, which yields
	\begin{equation} \label{G/W1} \left|\frac{\Gbb^\Omega[f](x)}{W(x)} \right| \leq C\norm{f}_{L^{\infty}(\Omega)}\frac{\delta(x)^\gamma}{W(x)}, \quad \forall x \in \Omega.
	\end{equation}
	By taking into account the definition of $W$ in \eqref{eq:weight} and $\gamma \geq s - \frac{1}{2}$, we deduce \eqref{G/W} from \eqref{G/W1}.
	
If $\mu \in \M(\Omega,\delta^\gamma )$, the homogeneous boundary condition in  \eqref{eq:Lmu} cannot be understood in the pointwise sense, but in  the following trace sense
\begin{equation} \label{trace}
\lim_{\varepsilon \to 0} \frac{1}{\varepsilon} \int_{ \{x \in \Omega: \delta(x) < \varepsilon  \} }\frac{\Gbb^\Omega[\mu](x)}{W(x)}dx = 0.
\end{equation}
Indeed, this can be obtained by using an argument similar to that of \cite[Lemma 3.8 a) and b)]{Aba_2019}. Although the proof of \cite[Lemma 3.8]{Aba_2019} deals with data in $L^1(\Omega,\delta^\gamma )$, the extension to data in $\M(\Omega,\delta^\gamma )$ is trivial. The case $\gamma = s - \frac{1}{2}$ is slightly different but follows straightforward from \cite[Lemma 3.8 b)]{Aba_2019}.

\subsection{Variational setting} \label{sec:variational} In this subsection, we assume that \eqref{eq_L1}--\eqref{eq_L3} and \eqref{eq_G1}--\eqref{eq_G3} hold.  For $f \in L^2(\Omega)$, consider the problem

\begin{equation}\label{eq_Lf} \left\{ \begin{aligned}
\L u &= f &&\text{ in }\Omega, \\
u  &= 0 &&\text{ on }\partial \Omega, \\
u &=0 &&\text{ in }\R^N \backslash \overline{\Omega} \text{ (if applicable)},
\end{aligned} \right.
\end{equation}

\begin{definition}[Variational solutions] \label{def:varsol} Assume $f \in L^2(\Omega)$. A function $u$ is a variational solution of \eqref{eq_Lf} if $u \in \Hbb(\Omega)$ and
\begin{equation}\label{eq_variation}
\inner{u,\xi}_{\Hbb(\Omega)} = \int_{\Omega} f\xi dx, \quad \forall \xi \in \Hbb(\Omega).
\end{equation}
\end{definition}
The existence and uniqueness of the variational solution to \eqref{eq_Lf} follows from the Lax--Milgram theorem since 
$$\norm{u}_{L^2(\Omega)}^2 \leq \mathcal{B}(u,u),\quad \forall u \in \Hbb(\Omega),$$
by the assumption \eqref{eq_L2}. Moreover,  the variational solution $u$ of \eqref{eq_Lf} satisfies
\begin{equation}\label{eq_L2estimate}
\norm{u}_{L^2(\Omega)} \le \norm{u}_{\Hbb(\Omega)} \lesssim\norm{f}_{L^2(\Omega)}.
\end{equation}
We note that the notion of variational solutions in Definition \ref{def:varsol} is equivalent to the one in \cite[Definition 2.5]{ChaGomVaz_2019_2020}.

Next we show that the variational solution of \eqref{eq_Lf} is a weak-dual solution.

\begin{proposition}\label{prop_variationalsolution} For any $f \in L^2(\Omega)$, $\Gbb^{\Omega}[f]$ is the variational solution of \eqref{eq_Lf}. 
\end{proposition}
\begin{proof}
First we assume that $f \in C^{\infty}_c(\Omega)$. Then $\Gbb^{\Omega} [f]$ is the weak-dual solution of \eqref{eq_Lf} and $\Gbb^{\Omega} [f] \in \Hbb(\Omega)$ by \eqref{eq_G1}. Furthermore, for any $v \in C^{\infty}_c(\Omega)$, one has
$$\inner{\Gbb^{\Omega}[f], v}_{\Hbb(\Omega)} = \inner{ \L[\Gbb^{\Omega}[f]], v}_{L^2(\Omega)} = \inner{ f, v}_{L^2(\Omega)}.$$
This implies that $\Gbb^{\Omega}[f]$ is the variational solution of \eqref{eq_Lf}. 
	
Next, assume $f \in L^2(\Omega)$ and let $u$ be the variational solution of \eqref{eq_Lf}. We will show that $u=\Gbb^{\Omega}[f]$. 

Let $\{f_n\} \subset C^{\infty}_c(\Omega)$ be a sequence converging to $f$ in $L^2(\Omega)$. Let $u_n = \Gbb^{\Omega}[f_n]$ be the variational solution of \eqref{eq_Lf} with $f$ replaced by $f_n$. Using the fact that $\Gbb^{\Omega}: L^2(\Omega) \longrightarrow L^2(\Omega)$ is continuous (see \cite[Theorem 2.2]{ChaGomVaz_2019_2020}), we obtain
$$\norm{\Gbb^{\Omega}[f_n] - \Gbb^{\Omega}[f]}_{L^2(\Omega)} = \| \Gbb^{\Omega}[f_n-f]\|_{L^2(\Omega)} \lesssim \norm{f_n - f}_{L^2(\Omega)}, \quad \forall n \in \mathbb{N},$$
which implies $\Gbb^{\Omega}[f_n] \to \Gbb^{\Omega}[f]$ in $L^2(\Omega)$. On the other hands, since $\Gbb^{\Omega}[f_n]-u$ is the variational solution of \eqref{eq_Lf} with $f$ replaced by $f_n-f$, we deduce from \eqref{eq_L2estimate} that
$$\norm{\Gbb^{\Omega}[f_n] - u}_{L^2(\Omega)}  \le \norm{\Gbb^{\Omega}[f_n] - u}_{\Hbb(\Omega)} \lesssim \norm{f_n - f}_{L^2(\Omega)}, \quad \forall n \in \mathbb{N}.$$
This implies $\Gbb^{\Omega}[f_n] \to u$ in $L^2(\Omega)$. By the uniqueness of the limit, $u = \Gbb^{\Omega}[f]$.
\end{proof}

Under the assumptions \eqref{eq_G1}--\eqref{eq_G2},  the operator $\Gbb^{\Omega} : L^2(\Omega) \to L^2(\Omega)$ is compact (see \cite[Proposition 5.1]{BonFigVaz_2018}).  Thus, there exist an orthogonal basis of $L^2(\Omega)$ consisting of eigenfunctions of $\Gbb^{\Omega}$, and a sequence $\{ \lambda_n \}$ satisfying $ 0 < \lambda_1 < \lambda_2 \le \cdots \le \lambda_n \nearrow + \infty$ such that
$$
\lambda_n \Gbb^{\Omega}[\varphi_n] = \varphi_n,\quad \forall n \in \mathbb{N}.
$$
In other words, $\lambda_n^{-1}$ is the eigenvalue of $\Gbb^{\Omega}$ associated to $\varphi_n$ for every $n \in \mathbb{N}$. By Proposition \ref{prop_variationalsolution}, $\varphi_n \in \Hbb(\Omega),\forall n \in \mathbb{N}$ and
$$
\inner{\varphi_n, \varphi_m}_{\Hbb(\Omega)} = \lambda_n \inner{\Gbb^{\Omega}[\varphi_n],\varphi_m}_{\Hbb(\Omega)} = \lambda_n \inner{\varphi_n, \varphi_m}_{L^2(\Omega)},\quad \forall n,m \in \mathbb{N}.
$$
Thus, for every $u,v \in \Hbb(\Omega)$, one has
$$
\mathcal{B}(u,v) = \inner{ u,v}_{\Hbb(\Omega)} = \inner{\sum_{n=1}^{\infty} \widehat{u_n} \varphi_n, \sum_{n=1}^{\infty} \widehat{v_n} \varphi_n}_{\Hbb(\Omega)} = \sum_{n=1}^{\infty}\lambda_n \widehat{u_n}\widehat{v_n},
$$
where $\widehat{u_n}:= \inner{u,\varphi_n}_{L^2(\Omega)},\forall n \in \mathbb{N}$.  Thus, one has the following characterization of $\Hbb(\Omega)$ 
$$
\Hbb(\Omega) = \left\{ u \in L^2(\Omega): \sum_{n=1}^{\infty} \lambda_n |\widehat{u_n}|^2 < \infty \right\}.
$$
We note that the space $\Hbb(\Omega)$ has been also recently investigated in \cite{BonSirVaz_2015} and \cite{ChaGomVaz_2019_2020}.

\section{Semilinear equations: existence and nonexistence} \label{sec:minimalsolution}

In this section, we assume that \eqref{eq_G1}--\eqref{eq_G2} hold. The aim of this section is to prove the existence of minimal solutions of \eqref{eq_nonlinear_source}. To this end, we first prove a general existence result for the semilinear problem
\begin{equation}\label{eq_general_nonlinear_source}
\left\{ \begin{aligned}  \L u &= g(x,u) + \mu &&\text{ in }\Omega, \\
u  &= 0 &&\text{ on }\partial \Omega, \\
u &=0 &&\text{ in }\R^N \backslash \overline{\Omega} \text{ (if applicable)},
\end{aligned} \right. \end{equation}
where $g:\Omega \times \R \to \R$ is a Caratheodory function such that $g(x,\cdot)$ is increasing in the second variable, for a.e. $x \in \Omega$ and $\mu \in \M(\Omega,\delta ^{\gamma} )$.  

Following Bonforte and V\'azquez \cite{BonVaz_2015} (see also \cite{BonFigVaz_2018,Aba_2019,ChaGomVaz_2019_2020}), we give the definition of weak-dual sub- and supersolutions. 
\begin{definition} \label{subweakdual} We say that a function $u$ is a weak-dual subsolution (resp. supersolution) of \eqref{eq_general_nonlinear_source} if $u \in L^1(\Omega,\delta^{\gamma} )$, $g(\cdot, u) \in L^1(\Omega,\delta^{\gamma} )$ and
\begin{equation}\label{eq_subsol}
\int_{\Omega} u\xi dx \leq \, (\text{resp.}\,  \geq)\, \int_{\Omega} g(x,u) \Gbb^{\Omega}[\xi] dx + \int_{\Omega} \Gbb^{\Omega}[\xi] d\mu, \quad \forall \xi \in \delta^{\gamma}L^{\infty}(\Omega), \xi \ge 0.
\end{equation}
A function $u$ is a weak-dual solution of \eqref{eq_general_nonlinear_source} if it is both subsolution and supersolution of \eqref{eq_general_nonlinear_source}.
\end{definition}

\begin{remark}
It can be seen that a function $u$ is a weak-dual subsolution of \eqref{eq_general_nonlinear_source} if and only if $u \in L^1(\Omega,\delta^{\gamma} )$,  $g(\cdot, u) \in L^1(\Omega,\delta^{\gamma} )$ and
\begin{equation}\label{eq_weaksubsol}
u \le \Gbb^{\Omega} [g(\cdot, u)] + \Gbb^{\Omega} [\mu] \quad \text{ a.e. in }\Omega.
\end{equation}
Indeed, if $u$ is a weak-dual subsolution of \eqref{eq_general_nonlinear_source}, then by Definition \ref{subweakdual}, $u \in L^1(\Omega,\delta^{\gamma} )$ and  $g(\cdot, u) \in L^1(\Omega,\delta^{\gamma} )$. Moreover, using the integration by parts formula in Lemma \ref{lem:integrationbypart}, for any $\xi \in \delta^{\gamma}L^{\infty}(\Omega)$, $\xi \ge 0$, one has
\begin{align*}
\int_{\Omega} u \xi dx \le \int_{\Omega} g(x,u) \Gbb^{\Omega}[\xi] dx + \int_{\Omega} \Gbb^{\Omega}[\xi] d\mu 
= \int_{\Omega} \xi \Gbb^{\Omega} [g(x,u)]dx  + \int_{\Omega} \xi \Gbb^{\Omega}[\mu] dx,
\end{align*}
 which implies 
 $$\int_{\Omega} \left( u - \Gbb^{\Omega} [g(\cdot,u)] - \Gbb^{\Omega}[\mu]  \right) \xi dx \le 0,\quad \forall \xi \in \delta^{\gamma}L^{\infty}(\Omega),\xi \ge 0.$$
This gives \eqref{eq_weaksubsol}. The converse statement can be established in an analogous way.

A similar observation also holds true for a weak-dual supersolution of \eqref{eq_general_nonlinear_source}. Finally, it can be seen that a function $u$ is a weak-dual solution of \eqref{eq_general_nonlinear_source} if and only if $u \in L^1(\Omega,\delta^{\gamma} )$, $g(\cdot, u) \in L^1(\Omega,\delta^{\gamma} )$ and
\begin{equation} \label{represent} u = \Gbb^{\Omega} [g(\cdot, u)] +\Gbb^{\Omega} [\mu] \quad \text{a.e. in } \Omega. 
\end{equation}
For interesting results regarding the equivalence of different types of solutions of \eqref{eq_general_nonlinear_source}, we refer to \cite[Theorem 2.1]{ChaGomVaz_2019_2020}. 
\end{remark}

\begin{remark} \label{gam<2s} From the representation \eqref{represent} and Marcinkiewicz estimates (Proposition \ref{prop_marcinkiewicz}), we see that if $u$ is a solution of \eqref{eq_general_nonlinear_source} then $u \in L^p(\Omega,\delta^{\gamma} )$, $p \in [1,p^*)$ and $u \in L^q(\Omega), q \in \left[1, \frac{N}{N-2s+\gamma}\right)$  if $\gamma < 2s$.
\end{remark}

Our general existence result is based on the sub- and supersolution method. The reader is referred to \cite[Theorem 5.4.1]{MarVer_2014} and \cite{MonPon_2008} for results in local frameworks.

\begin{proposition}\label{prop_existence.nonlinear} Let $v $ and $w$ be a weak-dual subsolution and a weak-dual supersolution of \eqref{eq_general_nonlinear_source}, respectively, satisfying $v \le w$ a.e. in $\Omega$.
Then there exist weak-dual solutions $\underline{u} \le \overline{u}$ of \eqref{eq_general_nonlinear_source}  such that for any weak-dual solution $u$ of \eqref{eq_general_nonlinear_source} satisfying $v \le u \le w$ a.e. in $\Omega$, there holds 
$$v \le \underline{u} \le u \le \overline{u} \le w \text{ a.e. in }\Omega.$$
\end{proposition}

Here we note that only the existence of the minimal weak-dual solution $\underline{u}$ is needed in our paper. However, we still prove the existence of the maximal solution $\overline{u}$ for the completeness. 

\begin{proof}For simplicity, in this proof by a solution (subsolution, supersolution, respectively) we mean a weak-dual solution (subsolution, supersolution, respectively).

We define an iterative sequence
$$u_0:= v, \quad u_{n+1} := \Gbb^{\Omega} [g(\cdot, u_n)] + \Gbb^{\Omega} [\mu ], \quad n \in \mathbb{N}.$$
It can be seen that $u_0 = v \le w$, therefore
\begin{align*}
v \le u_1 = \Gbb^{\Omega}[g(\cdot, u_0)] +\Gbb^{\Omega}[\mu] \le \Gbb^{\Omega}[g(\cdot, w)] +\Gbb^{\Omega}[\mu] \le w,
\end{align*}
since $v, w$ are a subsolution and a supersolution of \eqref{eq_general_nonlinear_source}, respectively, and $g(x,\cdot)$ is nondecreasing.
By induction and the monotonicity of $g(x,\cdot)$, one has $v \le u_n \le u_{n+1} \le w$ for every  $n \in \mathbb{N}$. 
This means $\{ u_n\}$ is an increasing sequence bounded from above by $w \in L^1(\Omega,\delta^{\gamma} )$. Using the Monotone Convergence Theorem, there exists a function $\underline{u} \in L^1(\Omega,\delta^{\gamma} )$ such that $u_n \to \underline{u}$  a.e. in $\Omega$  and in $L^1(\Omega,\delta^{\gamma} )$.
Furthermore, since $g(x,\cdot)$ is continuous and increasing, one has $g(\cdot, u_n) \to g(\cdot,\underline{u})$ a.e. in $\Omega$ and also in $L^1(\Omega,\delta^{\gamma} )$. Thus, since
$$\int_{\Omega} u_{n+1}\xi  dx  = \int_{\Omega} g(\cdot,u_n) \Gbb^{\Omega}[ \xi ]dx + \int_{\Omega} \Gbb^{\Omega}[ \xi ] d\mu, \quad \forall \xi \in \delta^{\gamma}L^{\infty}(\Omega),$$
letting $n \to \infty$ and using the Dominated Convergence Theorem, we deduce
$$\int_{\Omega} \underline{u}\xi  dx  = \int_{\Omega} g(\cdot,\underline{u}) \Gbb^{\Omega}[\xi] dx + \int_{\Omega} \Gbb^{\Omega}[ \xi ] d\mu, \quad \forall \xi \in \delta^{\gamma}L^{\infty}(\Omega),$$
from which we conclude that $\underline{u}$ is a solution to \eqref{eq_general_nonlinear_source}.

We prove that $\underline{u}$ is the minimal solution among all weak-dual solutions of \eqref{eq_general_nonlinear_source} being greater than $v$. Assume $u$ is another solution of \eqref{eq_general_nonlinear_source} such that $ v = u_0 \le u$. By induction, one has 
$$u = \Gbb^{\Omega} [g(\cdot, u)] + \Gbb^{\Omega} [\mu ] \ge \Gbb^{\Omega} [g(\cdot, u_n)] + \Gbb^{\Omega} [\mu ] = u_{n+1}, \quad \forall n \in \mathbb{N}.$$ 
Letting $n \to \infty$, we have $u \ge \underline{u}$, i.e. $\underline{u}$ is the desired minimal solution.

We next prove the existence of the maximal solution $\overline{u}$ of \eqref{eq_nonlinear_source}. Denote
$$ M := \sup \left\{ \int_{\Omega} u \delta^{\gamma} dx : u \text{ is a solution of }\eqref{eq_general_nonlinear_source} \text{ and }v \le u \le w \right\}< \infty. $$
Consider a sequence of solutions $\{u_n\}$ of \eqref{eq_general_nonlinear_source} such that $\int_{\Omega} u_n \delta^{\gamma} dx \to M$.
We see that if $v_1,v_2$ are two solutions of \eqref{eq_general_nonlinear_source}, then $v_1 \vee v_2$ is a subsolution since
$$v_i = \Gbb^{\Omega}[g(\cdot, v_i)] + \Gbb^{\Omega} [\mu] \le \Gbb^{\Omega} [g(\cdot, v_1 \vee v_2)]  + \Gbb^{\Omega}[\mu], \quad i = 1,2.$$ By this remark, the sequence $\{ u_n\}$ can be chosen to be increasing. Indeed, if $\{u_n\}$ is not increasing, one can always construct a sequence $\{v_n\}$ with $v_1:= u_1$ and $v_n$ being a solution of \eqref{eq_general_nonlinear_source} satisfying $\max\{ u_n, v_{n-1} \} \le v_n \le w$ (such a solution $v_n$ exists by the same argument leading to the existence of the minimal solution $\underline{u}$ above). Using the Monotone Convergence Theorem, there exists a function $\overline{u} \in L^1(\Omega,\delta^{\gamma} )$ such that 
$$u_n \to \overline{u} \text{ a.e. in }\Omega \text{ and in }L^1(\Omega,\delta^{\gamma} ) \text{ with } \int_{\Omega} \overline{u} \delta^{\gamma} dx = M.$$
Since $g(\cdot, u_n) \to g(\cdot, \overline{u})$ a.e. in $\Omega$, by the Monotone Convergence Theorem, we have $g(\cdot, u_n) \to g(\cdot, \overline{u})$ in $L^1(\Omega,\delta^{\gamma} )$. The same argument as in the first part shows that $\overline{u}$ is a solution of \eqref{eq_general_nonlinear_source}. 
Finally, we prove that $\overline{u}$ is the maximal solution. Indeed, for an arbitrary solution $u$ of \eqref{eq_general_nonlinear_source} such that $v \le u \le w$, we have $\overline{u}\vee u$ is a subsolution. Thus, there exists a solution $\widetilde{u}$ of \eqref{eq_general_nonlinear_source} such that $\overline{u} \le \overline{u} \vee u \le \widetilde{u} \le w$. This implies
$$\int_{\Omega} (\overline{u} \vee u) \delta^{\gamma} dx \le \int_{\Omega} \widetilde{u} \delta^{\gamma} dx \le  M = \int_{\Omega} \overline{u} \delta^{\gamma} dx,$$
which implies $\overline{u} \vee u = \overline{u}$. We complete the proof.
\end{proof}

We need the following lemma for the proof of Theorem \ref{th:main1}.
\begin{lemma}\label{lem_technicaleq} Consider the equation $t = c_pt^p + \lambda$, where $c_p > 0$ is a positive constant and $p > 1$. Then for every $\lambda \ge 0$ small enough, this equation admits a unique positive solution $t = T(\lambda) \le t_0 := (pc_p)^{-\frac{1}{p-1}}$, where $T$ is an increasing function and $T(0) = 0$.
\end{lemma}
\begin{proof}
It can be seen that the function $f(t) = t - c_p t^p$ is continuous and monotone increasing on $[0,t_0]$ ($t_0$ is given as above) since $f$ attains a unique maximum at $t = t_0$. Thus, for any $\lambda \in [0,f(t_0)]$, we have 
	$$
	f(t) = \lambda \Longleftrightarrow t = f^{-1}(\lambda) =: T(\lambda).
	$$
	Since $T$ is the inverse of $f$, it is continuous, monotone increasing with $T(0) = 0$. We have the desired result.
\end{proof}

\smallskip

\begin{proof}[{\sc Proof of Theorem \ref{th:main1} (1)}] ~
	
\textbf{Step 1.} We first prove the existence of the minimal solution of \eqref{eq_nonlinear_source} for $\lambda > 0$ small enough.

Put $v_\lambda:= \lambda \Gbb^{\Omega}[\mu]$, $w_\lambda:= T(\lambda) \Gbb^{\Omega}[\mu]$ for some function $T$ which will be specified later. It can be seen that $v_\lambda$ is a weak-dual subsolution of \eqref{eq_nonlinear_source} since $v_\lambda = \lambda \Gbb^{\Omega}[\mu] \in L^p(\Omega,\delta^{\gamma} )$ and $v_\lambda \le \Gbb^{\Omega}[ (v_\lambda )^p ] + \lambda \Gbb^{\Omega}[\mu]$. By Proposition \ref{prop_existence.nonlinear}, it is enough to find a suitable function $T$ such that $w_\lambda$ is a weak-dual  supersolution and $w_\lambda \ge v_\lambda$ a.e. in $\Omega$. 

Indeed, Lemma \ref{lem_technicaleq} shows that for all $\lambda > 0$ small enough, there exists a $T(\lambda) > 0$ such that $T$ is increasing with respect to $\lambda$, $T(\lambda) \to 0$ as $\lambda \to 0$ and $T(\lambda) = c_p T(\lambda)^p +\lambda$, where $c_p$ is the constant $C$ in \eqref{eq_5.Greenestimate}. Thus, it follows from Proposition \ref{prop_priorestimate} that
\begin{equation}\label{eq_minimal1}
\begin{aligned}
\Gbb^{\Omega}[(w_\lambda)^p] = \Gbb^{\Omega} \left[\left(T(\lambda)\Gbb^{\Omega}[\mu] \right)^p\right] = T(\lambda)^p\Gbb^{\Omega} [\Gbb^{\Omega}[\mu]^p] &\le c_pT(\lambda)^p \Gbb^{\Omega}[\mu] \\ &= (T(\lambda) - \lambda) \Gbb^{\Omega}[\mu],
\end{aligned}
\end{equation}
from which we deduce that
$$ w_\lambda = T(\lambda) \Gbb^{\Omega}[\mu] \ge \Gbb^{\Omega}[(w_\lambda)^p] + \lambda\Gbb^{\Omega}[\mu].$$
This implies $w_\lambda$ is a weak-dual supersolution. Furthermore, it is clear that $w_\lambda = T(\lambda) \Gbb^{\Omega}[\mu] \ge \lambda \Gbb^{\Omega}[\mu] = v_\lambda$. We deduce from Proposition \ref{prop_existence.nonlinear} that there exists a minimal nonnegative weak-dual solution $\underline u_{\lambda}$ of \eqref{eq_nonlinear_source} which satisfies $\lambda\Gbb^{\Omega}[\mu] \le  \underline u_{\lambda} \le T(\lambda)\Gbb^{\Omega} [\mu]$ a.e. in $\Omega$.

\textbf{Step 2.} We prove the existence of a threshold value $\lambda^* > 0$  stated in Theorem \ref{th:main1}.

It can be seen that if \eqref{eq_nonlinear_source} has a minimal positive weak-dual solution $u_{\lambda}$ for $\lambda > 0$, then it also admits a minimal positive weak-dual solution for all $\lambda' \in (0,\lambda)$. Furthermore, the map $\lambda \mapsto u_{\lambda}$ is increasing and continuous. It follows from the fact that the sequence defining the minimal solution $\underline u_{\lambda}$, namely
\begin{equation}\label{eq_sequence_minimal}
u_{\lambda,0} = \lambda \Gbb^{\Omega} [\mu], \quad u_{\lambda,n+1}  = \Gbb^{\Omega} [(u_{\lambda,n})^p] + \lambda \Gbb^{\Omega}[\mu], \quad  n \in \mathbb{N},
\end{equation}
satisfies $u_{\lambda,n} > u_{\lambda',n}$, and the map $\lambda \mapsto u_{\lambda,n}$ is continuous and increasing for all $n \in \mathbb{N}$. Thus, we may define
\begin{equation} \label{lambda*} \lambda^*:= \sup \left\{ \lambda > 0: \eqref{eq_nonlinear_source} \text{ has a positive weak-dual solution} \right\}.
\end{equation}

We prove that $\lambda^* < \infty$. Indeed, let $\lambda>0$ be such that problem \eqref{eq_nonlinear_source} admits a positive weak-dual solution  $u_\lambda$. Then $u_\lambda \in L^p(\Omega,\delta^\gamma )$ and 
\begin{equation} \label{eq_ulambda} \int_{\Omega} u_{\lambda}\xi dx = \int_{\Omega}(u_{\lambda})^p \xi dx + \lambda \int_{\Omega} \Gbb^{\Omega} [\xi] d\mu, \quad \forall \xi \in \delta^{\gamma}L^{\infty}(\Omega).
\end{equation}

Let $\lambda_1$ and $\varphi_1$ be the first eigenvalue and the corresponding positive eigenfunction of $\L$, respectively. Recall that $\varphi_1 = \lambda_1 \Gbb^{\Omega}[\varphi_1] \sim \lambda_1\delta^{\gamma}$ by \cite[Proposition 5.3]{BonFigVaz_2018}. Choosing $\xi = \varphi_1 \in \delta^{\gamma}L^{\infty}(\Omega)$ as a test function in \eqref{eq_ulambda}, one has
\begin{equation} \label{ueigen}
\int_{\Omega} (u_{\lambda})^p \Gbb^{\Omega}[\varphi_1] dx + \lambda \int_{\Omega} \Gbb^{\Omega}[\varphi_1]d\mu  
=  \int_{\Omega} u_{\lambda} \varphi_1 dx 
= \lambda_1 \int_{\Omega} u_{\lambda} \Gbb^{\Omega}[\varphi_1] dx.
\end{equation}
This and H\"older's inequality imply
\begin{equation}\label{eq_boundedness1}
\begin{aligned}
\int_{\Omega} (u_{\lambda})^p \Gbb^{\Omega}[\varphi_1] dx + \lambda \int_{\Omega} \Gbb^{\Omega}[\varphi_1]d\mu  \le \lambda_1 \left(\int_{\Omega} (u_{\lambda})^p \Gbb^{\Omega}[\varphi_1] dx \right)^{\frac{1}{p}} \left( \int_{\Omega} \Gbb^{\Omega}[\varphi_1] dx \right)^{1-\frac{1}{p}}.
\end{aligned}
\end{equation}
This leads to
\begin{equation}\label{eq_boundedness}
\int_{\Omega} (u_{\lambda})^p \Gbb^{\Omega}[\varphi_1] dx \le \lambda_1^{\frac{p}{p-1}} \int_{\Omega}\Gbb^{\Omega}[\varphi_1] dx.
\end{equation}
Combining \eqref{eq_boundedness1}, \eqref{eq_boundedness} and the fact that $ \int_{\Omega} \Gbb^{\Omega}[\varphi_1]d\mu \sim \int_{\Omega}\delta^{\gamma} d\mu = 1$, 
one has
\begin{equation} \label{tildelambda}
\lambda \le \frac{1}{\int_{\Omega} \Gbb^{\Omega}[\varphi_1] d\mu} \lambda_1^{\frac{p}{p-1}} \int_{\Omega} \Gbb^{\Omega}[\varphi_1] dx \leq  C\lambda_1^{\frac{p}{p-1}} \int_{\Omega} \delta^{\gamma} dx=:\tilde \lambda<+\infty.
\end{equation}
Notice that $\tilde \lambda$ depends on $N,\Omega,s,\gamma, p,\lambda_1$. Therefore $\lambda^* < +\infty$. 

It can be seen that 
$$
\lambda^*= \sup \left\{ \lambda > 0: \eqref{eq_nonlinear_source} \text{ has a minimal positive weak-dual solution} \right\}.
$$
This is obtained due to the observation that if  \eqref{eq_nonlinear_source} has a positive weak-dual solution $u_\lambda$ for $\lambda > 0$ then  \eqref{eq_nonlinear_source} has a minimal positive weak-dual solution. Indeed, since $u_{\lambda}$ is a weak-dual solution, it is also a weak-dual supersolution \eqref{eq_nonlinear_source}, while $0$ is a weak-dual subsolution of \eqref{eq_nonlinear_source}. Using Proposition \ref{prop_existence.nonlinear}, we know that there exists a minimal weak-dual solution $\underline{u}_{\lambda}$ such that $0 \le \underline{u}_{\lambda} \le u_{\lambda}$. 

\textbf{Step 3.} We prove the existence of the minimal positive weak-dual solution of \eqref{eq_nonlinear_source} for $\lambda = \lambda^*$.

First, observing \eqref{eq_boundedness1} and the estimate $\Gbb^{\Omega}[\varphi_1] \sim \delta^\gamma$, we derive that
\begin{equation} \label{Lpuni}
\int_{\Omega} (u_\lambda)^p \delta^\gamma dx \leq C\lambda_1^{\frac{p}{p-1}}\int_{\Omega}\delta^\gamma dx < C\lambda_1^{\frac{p}{p-1}} < C.
\end{equation}
From \eqref{ueigen}, \eqref{Lpuni} and again the estimate $\Gbb^{\Omega}[\varphi_1] \sim \delta^\gamma$, we derive
\begin{equation} \label{L1uni}
\int_{\Omega} u_\lambda \delta^\gamma dx \leq C\left(\int_{\Omega}(u_\lambda)^p \delta^\gamma dx + \lambda \int_{\Omega}\delta^\gamma d\mu \right) \leq C.
\end{equation}

From Step 2, for each $\lambda \in (0,\lambda^*)$, there exists a minimal positive weak-dual solution $\underline{u}_{\lambda}$ of \eqref{eq_nonlinear_source}. By \eqref{Lpuni} and \eqref{L1uni}, the sequence $\{ \underline{u}_\lambda \}$  is uniformly bounded in $L^1(\Omega,\delta^{\gamma} )$ and in $L^p(\Omega,\delta^{\gamma} )$ with respect to $\lambda$. Moreover, $\{ \underline{u}_\lambda \}$ is increasing with respect to $\lambda$. By the Monotone Convergence Theorem,  $\{ \underline{u}_\lambda \}$ converges to a function $\underline{u}_{\lambda^*}$ a.e. in $\Omega$, in $L^1(\Omega,\delta^\gamma )$ and in $L^p(\Omega,\delta^\gamma )$. 
For any  $\lambda < \lambda^*$ and $\xi \in \delta^{\gamma}L^{\infty}(\Omega)$, one has
$$\int_{\Omega} \underline{u}_{\lambda} \xi dx = \int_{\Omega}(\underline{u}_{\lambda})^p \Gbb^{\Omega}[\xi] dx + \lambda \int_{\Omega} \Gbb^{\Omega}[\xi] d\mu.$$
By letting $\lambda \nearrow \lambda^*$, one has
$$ \int_{\Omega} \underline{u}_{\lambda^*} \xi dx = \int_{\Omega}(\underline{u}_{\lambda^*})^p \Gbb^{\Omega}[\xi]dx + \lambda^* \int_{\Omega} \Gbb^{\Omega}[\xi ] d\mu.$$
Therefore $\underline{u}_{\lambda^*}$ is a weak-dual solution of \eqref{eq_nonlinear_source} for $\lambda = \lambda^*$.

We show that $\underline{u}_{\lambda^*}$ is the minimal solution of \eqref{eq_nonlinear_source} with $\lambda=\lambda^*$. It can be seen that if $u_{\lambda^*}$ is an arbitrary solution of \eqref{eq_nonlinear_source} with $\lambda=\lambda^*$, then $u_{\lambda^*} \geq \underline{u}_{\lambda}$ for $\lambda < \lambda^*$ since $u_{\lambda^*} \geq \underline{u}_{\lambda,n}$ for all $n \in \mathbb{N}$. Letting $\lambda \nearrow \lambda^*$, we obtain $u_{\lambda^*} \geq \underline{u}_{\lambda^*}$. The proof is complete.
\end{proof}

\begin{remark} \label{rmk:uniformbound} 
	
(i) We notice that the argument leading to \eqref{Lpuni} in Step 2 of the proof of Theorem \ref{th:main1}(1) is valid for every $p>1$. Therefore, for any $p>1$, if $u$ is a positive weak-dual solution of \eqref{eq_nonlinear_source} then 
$\| u \|_{L^p(\Omega,\delta^\gamma )} \leq C$ where $C=C(N,\Omega,s,\gamma,p,\lambda_1)$. Moreover, from the representation
\begin{equation} \label{repr2} u = \Gbb^{\Omega}[u^p] + \lambda\Gbb^{\Omega}[\mu],
\end{equation}
Proposition \ref{prop_marcinkiewicz}, Corollary \ref{cor:GLL} (i) and \eqref{Lpuni}, we have, for any $q \in [1,p^*)$, there hold
\begin{equation} \label{univer-2} \begin{aligned} \| u \|_{L^q(\Omega,\delta^\gamma )} &= \| \Gbb^{\Omega}[u^p] + \lambda\Gbb^{\Omega}[\mu] \|_{L^q(\Omega,\delta^\gamma )} \\
&\leq C(\| u \|_{L^p(\Omega,\delta^\gamma )}^p + \lambda\| \mu \|_{\M(\Omega,\delta^\gamma )}) \leq C,
\end{aligned}\end{equation}
where $C=C(N,\Omega,s,\gamma,\lambda_1,p,q)$.
\end{remark}

\begin{remark} \label{rmk:subsuperup} We can infer from Proposition \ref{prop_existence.nonlinear} that for $p>1$ and $0<\lambda<\lambda'$, if problem \eqref{eq_nonlinear_source} with $\lambda=\lambda'$ admits a positive weak-dual solution $u_{\lambda'}$ then problem \eqref{eq_nonlinear_source} admits a positive weak-dual solution $u_{\lambda}$ such that $u_{\lambda} < u_{\lambda '}$. In other words, if problem \eqref{eq_nonlinear_source} does not admit any positive weak-dual solution then neither does \eqref{eq_nonlinear_source} with $\lambda=\lambda'$.
\end{remark} \smallskip


\begin{proof}[{\sc Proof of Theorem \ref{th:main1} (2)}] ~ We observe from Remark \ref{rmk:subsuperup} that it is sufficient to consider $\lambda < \tilde \lambda$ where $\tilde \lambda$ is defined in \eqref{tildelambda}. 

Suppose by contradiction that for every $\mu \in \M(\Omega,\delta^{\gamma} )$ with $\norm{\mu}_{\M(\Omega,\delta^{\gamma} )} = 1$, there exists a positive weak-dual solution to \eqref{eq_nonlinear_source}. Let $y^* \in \partial \Omega$ and $\{y_n\} \subset \Omega$ such that $y_n \to y^* \in \partial \Omega$.  Put $\mu_n = \delta^{-\gamma}\delta_{y_n}$ then $\| \mu_n \|_{\M(\Omega,\delta^\gamma )} = 1$. By the supposition, for each $n \in \N$, there exists a positive weak-dual solution $u_n$ of \eqref{eq_nonlinear_source} with datum $\lambda \mu_n$. 

	
From the representation $ u_n = \Gbb^\Omega[(u_n)^p] + \lambda \Gbb^\Omega[\mu_n]$,
we deduce
	\begin{equation} \label{un-2} u_n(x) \geq  \lambda \Gbb^\Omega[\mu_n](x) = \lambda \int_{\Omega} G^\Omega(x,y)d\mu_n(y) = \lambda\frac{G^\Omega(x,y_n)}{\delta(y_n)^\gamma}. 
	\end{equation}
	Since
	$$G^{\Omega}(x,y) \sim \dfrac{1}{|x-y|^{N-2s}}\cdot \dfrac{\delta(x)^{\gamma}\delta(y)^{\gamma}}{\delta(x)^{2\gamma}+\delta(y)^{2\gamma} + |x-y|^{2\gamma}}, \quad x,y \in \Omega,$$
it follows from Fatou lemma that
	\begin{equation}\label{eq:bound1}
	\begin{aligned}
	\liminf_{ n \to \infty} \int_{\Omega} (u_n)^p \delta^{\gamma} dx \geq \int_{\Omega} (\liminf_{n \to \infty} (u_n)^p) \delta^{\gamma} dx 
	\gtrsim \lambda \int_{\Omega} \left( \dfrac{\delta(x)^{\gamma}}{|x-y^*|^{N-2s+2\gamma}} \right)^p \delta(x)^{\gamma} dx.
	\end{aligned}
	\end{equation}

Since $\Omega$ is a $C^2$ domain, it satisfies the interior cone condition, hence there exists $r_0 > 0$ small enough such that the circular cone at vertex $y^*$
	$$
	{\mathcal C}_{r_0}(y^*):=\left\{ x \in B_{r_0}(y^*):  (x-y^*)\cdot {\bf n}_{y^*} > \frac{1}{2}|x-y^*| \right\}  \subset \Omega,
	$$
where ${\bf n}_{y^*}$ denotes the inward unit normal vector to $\partial \Omega$ at $y^*$.

Without loss of generality, suppose that the coordinates are placed so that $y^* = 0 \in \partial \Omega$, the tangent hyperplane to $\partial \Omega$ at $0$ is $\{ x=(x_1,\ldots,x_{N-1},x_N) \in \R^N: x_N=0\}$ and ${\bf n}_0 = (0,\ldots, 0,1)$. We can choose $r_0$ small enough such that 
$$\delta(x) \geq \alpha |x|, \quad \forall x \in \mathcal{C}_{r_0}(0),$$
for some $\alpha \in (0,1)$.
Then we have 
	\begin{equation} \label{intr0}
	\begin{aligned}
\int_{\Omega} \left( \dfrac{\delta(x)^{\gamma}}{|x|^{N-2s+2\gamma}} \right)^p \delta(x)^{\gamma} dx 
& \gtrsim \int_{{\mathcal C}_{r_0}(0)} \dfrac{1}{|x|^{(N-2s+\gamma)p-\gamma}}dx \\
& \gtrsim \int_{B_{r_0}(0) } \dfrac{1}{|x|^{(N-2s+\gamma)p-\gamma}}dx \\
& \sim \int_0^{r_0} t^{N-1-(N-2s+\gamma)p +\gamma} dt.
	\end{aligned} 
	\end{equation}
Since $p \geq p^*$, the last integral in \eqref{intr0} is divergent. This, joint with \eqref{eq:bound1} and \eqref{intr0}, yields
	$$ \liminf_{n \to \infty}\int_{\Omega}(u_n)^p \delta^\gamma dx  =\infty,
	$$
which leads to a contradiction since the sequence $\{u_n\}$ is uniformly bounded in $L^p(\Omega,\delta^{\gamma})$ by Remark \ref{rmk:uniformbound}. The proof is complete.
\end{proof}

\section{Eigenproblem and key regularity results} \label{sec:keyregularity}
In this section, we assume that assumptions \eqref{eq_L1}--\eqref{eq_L3} and \eqref{eq_G1}--\eqref{eq_G3} hold. We study weighted eigenproblem associated to $\L$ in the variational framework.

\subsection{Embeddings}

We prove some embedding results. We first recall a classical result.
\begin{lemma}\label{lem_sobolevembedding} $\Hbb(\Omega) \hookrightarrow L^q(\Omega)$ for all $q \in \left[1, \frac{2N}{N-2s} \right]$, and $\Hbb(\Omega) \hookrightarrow \hookrightarrow L^q(\Omega)$ for all $q \in \left[1, \frac{2N}{N-2s} \right)$.
\end{lemma}

\begin{proof} Since  $\Hbb(\Omega) \hookrightarrow H^s(\Omega)$ by \eqref{eq_L3},  we obtain $$\Hbb(\Omega) \hookrightarrow L^q(\Omega), \quad \forall \, q \in \left[1,\frac{2N}{N-2s}\right]$$
and
$$\Hbb(\Omega) \hookrightarrow \hookrightarrow L^q(\Omega), \quad\forall \, q \in \left[1,\frac{2N}{N-2s}\right)$$ 
due to \cite[Theorem 6.7, Corollary 7.2]{NezPalVal_2012}.
\end{proof}
\begin{lemma}\label{lem_interpolation} Assume that $\alpha \in [0,2s)$ and $q \in \left[1+\frac{\alpha}{2s}, \frac{2N-2\alpha}{N-2s}\right]$. Then for any $u \in \Hbb (\Omega)$, we have $u \in L^{\frac{2(qs-\alpha)}{2s-\alpha}}(\Omega) \cap  L^q(\Omega,\delta^{-\alpha} )$ and furthermore, the following inequality holds
\begin{equation}\label{eq_interpolation}
\| u \|_{L^q(\Omega,\delta^{-\alpha})}  \lesssim \norm{u}_{\Hbb(\Omega)}^{\frac{\alpha}{qs}} \norm{u}^{1-\frac{\alpha}{qs}}_{L^{\frac{2(qs-\alpha)}{2s-\alpha}}(\Omega)}.
\end{equation}
As a consequence, $\Hbb(\Omega)\hookrightarrow L^q(\Omega,\delta^{-\alpha} )$ for all $q \in \left[1+\frac{\alpha}{2s}, \frac{2N-2\alpha}{N-2s}\right]$, and $\Hbb(\Omega)\hookrightarrow \hookrightarrow L^q(\Omega,\delta^{-\alpha} )$ for all $q \in \left[1+\frac{\alpha}{2s}, \frac{2N-2\alpha}{N-2s}\right)$. 
\end{lemma}
Notice that we recover Lemma \ref{lem_sobolevembedding} if $\alpha = 0$.
\begin{proof}[\sc Proof of Lemma \ref{lem_interpolation}]  By the assumption on $q$, one has $1\le \frac{2(qs-\alpha)}{2s-\alpha} \le \frac{2N}{N-2s}$. Therefore, for any $u \in \Hbb(\Omega)$, it follows from Lemma \ref{lem_sobolevembedding} that $u \in L^{\frac{2(qs-\alpha)}{2s-\alpha}}(\Omega)$ and
\begin{equation} \label{embed0} \norm{u}_{L^{\frac{2(qs-\alpha)}{2s-\alpha}}(\Omega)} \lesssim\norm{u}_{\Hbb(\Omega)}, \quad \forall \, u \in \Hbb(\Omega).
\end{equation}
	
Since $\alpha < 2s$, using H\"older's inequality and the fact that $\norm{\cdot}_{\Hbb(\Omega)}$ is equivalent to $\norm{\cdot}_{H^s_{00}(\Omega)}$ by \eqref{eq_L3}, one has, for any $u \in \Hbb(\Omega)$,
\begin{equation}\label{eq_compactsource1}
\begin{aligned}
\| u \|_{L^q(\Omega,\delta^{-\alpha})}^q  &= \int_{\Omega} \left(\frac{|u|^2}{\delta^{2s}}\right)^{\frac{\alpha}{2s}} \cdot |u|^{q - \frac{\alpha}{s}} dx \\ 
&\leq \left(\int_{\Omega} \dfrac{|u|^2}{\delta^{2s}}dx \right)^{\frac{\alpha}{2s}}\left(\int_{\Omega} |u|^{\frac{2(qs-\alpha)}{2s-\alpha}}dx \right)^{1-\frac{\alpha}{2s}}\\
&\lesssim \norm{u}^{\frac{\alpha}{s}}_{\Hbb(\Omega)} \norm{u}^{q-\frac{\alpha}{s}}_{L^{\frac{2(qs-\alpha)}{2s-\alpha}}(\Omega)}.
\end{aligned}
\end{equation}
Combining \eqref{embed0} and \eqref{eq_compactsource1} yields \eqref{eq_interpolation}. 

From \eqref{eq_interpolation} and \eqref{eq_compactsource1}, we obtain
\begin{equation} \label{embed1} \| u \|_{L^q(\Omega,\delta^{-\alpha})}  \lesssim \norm{u}_{\Hbb(\Omega)}, \quad \forall \, u \in \Hbb(\Omega).
\end{equation}
Hence $\Hbb(\Omega)$ is continuously embedded in $L^q(\Omega,\delta^{-\alpha} )$.

Now, consider the case $q \in \left[1+\frac{\alpha}{2s}, \frac{2N-2\alpha}{N-2s}\right)$. Assume that $\{u_n\}$ is a bounded sequence in $\Hbb(\Omega)$. Since $\Hbb(\Omega)$ is reflexive, there exists  $u \in \Hbb (\Omega)$ such that $u_n \rightharpoonup u$ in $\Hbb (\Omega)$. Also, since $\Hbb (\Omega) \hookrightarrow \hookrightarrow L^{\frac{2(qs-\alpha)}{2s-\alpha}}(\Omega)$ is compact (notice that $\frac{2(qs-\alpha)}{2s-\alpha} < \frac{2N}{N-2s}$), up to a subsequence, $u_n \to u$ in $L^{\frac{2(qs-\alpha)}{2s-\alpha}}(\Omega)$. Now we can apply \eqref{eq_interpolation} with $u_n - u$ to have
\begin{equation}\label{eq_embdd}
\begin{aligned}
\|u_n - u\|_{L^q(\Omega,\delta^{-\alpha})}^q  \lesssim \norm{u_n - u}_{\Hbb(\Omega)} ^{\frac{\alpha}{s}}  \norm{u_n-u}^{q-\frac{\alpha}{s}}_{L^{\frac{2(qs-\alpha)}{2s-\alpha}}(\Omega)} \lesssim 2M \norm{u_n-u}^{q-\frac{\alpha}{s}}_{L^{\frac{2(qs-\alpha)}{2s-\alpha}}(\Omega)}.
\end{aligned}
\end{equation}
From \eqref{eq_embdd}, we deduce that $u_n \to u$ in $L^q(\Omega,\delta^{-\alpha} )$. Therefore  $\Hbb (\Omega) \hookrightarrow \hookrightarrow L^q(\Omega,\delta^{-\alpha}  )$. 
\end{proof}

\begin{proposition}\label{prop_compactembd} Assume that $a \in L^r(\Omega,\delta^{\gamma} )$ is a positive function with $r > \frac{N+\gamma}{2s}$. Then $\Hbb(\Omega) \hookrightarrow \hookrightarrow L^2(\Omega,a )$.
\end{proposition}

\begin{proof} Put $\alpha = \frac{\gamma}{r-1}$. Since $r > \frac{N+\gamma}{2s} \ge 1+\frac{\gamma}{2s}$, it follows that
$$ 1+\frac{\alpha}{2s} \le  2r'  = \frac{2r}{r-1} < \frac{2(N-\alpha)}{N-2s} \text{ and }\alpha < 2s. $$
By Lemma \ref{lem_interpolation}, $\Hbb(\Omega) \hookrightarrow \hookrightarrow L^{2r'} (\Omega, \delta^{-\alpha} )$ and  estimate \eqref{embed1} holds with $q=2r'$, namely
\begin{equation} \label{embed1a} \| u \|_{L^{2r'}(\Omega,\delta^{-\alpha})}  \lesssim \norm{u}_{\Hbb(\Omega)}, \quad \forall \,u \in \Hbb(\Omega).
\end{equation}
Using H\"older's inequality, one has
$$ 
\int_{\Omega} a |u|^2  dx = \int_{\Omega} a\delta^{\frac{\gamma}{r}} |u|^2\delta^{-\frac{\gamma}{r}} dx  \leq \left( \int_{\Omega} a^r \delta^{\gamma} dx\right)^{\frac{1}{r}} \left( \int_{\Omega} |u|^{2r'}\delta^{-\alpha} dx \right)^{\frac{1}{r'}},
$$
for any $u \in L^{2r'}(\Omega,\delta^{-\alpha})$. Equivalently,
\begin{equation} \label{eq-weight-1}
\| u \|_{L^2(\Omega,a)} \leq \| a \|_{L^r(\Omega,\delta^\gamma )}^{\frac{1}{2}} \| u \|_{L^{2r'}(\Omega,\delta^{-\alpha})} .
\end{equation}
Combining \eqref{embed1a} and \eqref{eq-weight-1} yields
\begin{equation} \label{embed-2} \| u \|_{L^2(\Omega,a)} \lesssim \| a \|_{L^r(\Omega,\delta^\gamma )}^{\frac{1}{2}} \| u \|_{\Hbb(\Omega)}, \quad \forall \, u \in \Hbb(\Omega).
\end{equation}
Moreover, from Lemma \ref{lem_interpolation} and \eqref{eq-weight-1}, we obtain $\Hbb(\Omega) \hookrightarrow \hookrightarrow L^{2r'} (\Omega, \delta^{-\alpha} ) \hookrightarrow L^2(\Omega, a )$. This implies that $\Hbb(\Omega) \hookrightarrow \hookrightarrow L^2(\Omega,a )$. We complete the proof.
\end{proof}

\subsection{Eigenproblems} We consider the weighted eigenproblem
\begin{equation}\label{eq_eigenvalue} 
\left\{ \begin{aligned} 
\L u &= \sigma a u &&\text{ in }\Omega, \\
u  &= 0 &&\text{ on }\partial \Omega, \\
u &=0 &&\text{ in }\R^N \backslash \overline{\Omega} \text{ (if applicable)},
\end{aligned} \right. 
\end{equation}
where $a \in L^r(\Omega,\delta^{\gamma} )$, $r > \frac{N+\gamma}{2s}$, is a positive function.
This problem has been studied in \cite{ChaGomVaz_2019_2020} in the case $a = 1$. In this case, one has $\varphi \in \delta^{\gamma}L^{\infty}(\Omega)$ \cite[Proposition 3.7]{ChaGomVaz_2019_2020}, where $\varphi$ is an arbitrary eigenfunction of $\L$. Furthermore,  the first eigenvalue and the corresponding eigenfunction of $\L$ are characterized by (see \cite[Proposition 5.1]{BonFigVaz_2018})
$$0 < \lambda_1 = \inf_{ u \in L^2(\Omega) \setminus \{0\} } \frac{\displaystyle \int_{\Omega}u^2 dx}{\displaystyle  \int_{\Omega} u\Gbb^{\Omega}[u]dx} = \frac{\displaystyle  \int_{\Omega} \varphi_1^2 dx }{\displaystyle \int_{\Omega} \varphi_1\Gbb^{\Omega}[\varphi_1]dx}$$
with $\varphi_1 \sim \delta^{\gamma}$ in $\Omega$.

In our present work, we consider the eigenproblem in a framework involving spaces $\Hbb(\Omega)$ (cf. space $\Hbb^1_{\mathcal{L}}(\Omega)$  in \cite{ChaGomVaz_2019_2020}). More precisely, a nonzero function $\varphi \in \Hbb(\Omega)$ is an eigenfunction associated to an eigenvalue $\sigma$ if
\begin{equation}\label{eq_variationaleigen0}
\inner{\varphi,\xi}_{\Hbb(\Omega)} = \sigma \int_{\Omega} a \varphi \xi dx, \quad \forall \xi \in \Hbb(\Omega).
\end{equation}

\begin{remark}\label{rmk_welldefined}
We note that the right-hand side of \eqref{eq_variationaleigen0} is well-defined by Proposition \ref{prop_compactembd}. Furthermore, it can easily be seen that if $\varphi$ satisfies \eqref{eq_variationaleigen0}, then $\varphi = \sigma \Gbb^{\Omega}[ a\varphi]$. Indeed, for a function $\xi \in C^{\infty}_c(\Omega)$, choosing $\Gbb^{\Omega} [\xi] \in \Hbb(\Omega)$ as a test function in \eqref{eq_variationaleigen0}, one has
$$\inner{\varphi, \Gbb^{\Omega} [\xi]}_{\Hbb(\Omega)} = \sigma \int_{\Omega} a \varphi \Gbb^{\Omega} [\xi] dx.$$
By noticing that
$$\inner{\varphi, \Gbb^{\Omega} [\xi]}_{\Hbb(\Omega)}  = \int_{\Omega} \varphi \xi dx \text{ and } \int_{\Omega} a \varphi \Gbb^{\Omega} [\xi] dx = \int_{\Omega} \Gbb^{\Omega}[a\varphi] \xi dx,$$
we conclude that
$$ \int_{\Omega} \left(\varphi - \sigma \Gbb^{\Omega}[a\varphi] \right)\xi dx,\quad \forall \xi \in C^{\infty}_c(\Omega),$$
from which we have $\varphi = \sigma \Gbb^{\Omega}[a\varphi]$.
\end{remark}

\begin{proposition}\label{prop_eigenproblem} Let $0< a \in L^r(\Omega,\delta^{\gamma} )$ with $r > \frac{N+\gamma}{2s}$. Then the eigenproblem
\begin{equation} \label{weighteigen} \sigma_1^a:  = \inf_{ \phi \in \Hbb(\Omega) \setminus \{0\}} \frac{\norm{\phi}^2_{\Hbb(\Omega)}}{\norm{\phi}^2_{L^2(\Omega, a )}}
\end{equation}
 admits a positive minimizer $\varphi_1^a$ in $\Hbb(\Omega)$ with normalization $\| \varphi_1^a \|_{L^2(\Omega,a)}=1$. Therefore $\sigma_1^a$ is the first eigenvalue of $\L$ in $\Hbb(\Omega)$ with respect to the weight $a$ and $\varphi_1^a$ is a corresponding eigenfunction. The eigenpair $(\sigma_1^a,\varphi_1^a)$ solves \eqref{eq_eigenvalue} in the sense 
\begin{equation} \label{weighteigen-1}
\langle \varphi_1^a, \xi \rangle_{\Hbb(\Omega)} = \sigma_1^a \int_{\Omega} a \varphi_1^a \xi dx, \quad \forall \xi \in \Hbb(\Omega).
\end{equation}
\end{proposition}

\begin{proof}The proof is standard and similar to that in \cite[Proposition 9]{SerVal_2013}. It is easy to see that $0 \leq \sigma_1^a < \infty$. Moreover, from \eqref{embed-2}, there exists a constant $C > 0$ depending on $\|a\|_{L^r(\Omega,\delta^\gamma)}$, such that
$$
\|u\|_{L^2(\Omega,a)}^2 \leq C \| u \|_{\Hbb(\Omega)}^2 \quad \forall \, u \in \Hbb(\Omega).
$$
Without loss of generality, we may assume that there exists a sequence $\{ u_n \} \subset \Hbb(\Omega)$ such that $\norm{u_n}_{L^2(\Omega,a )} = 1$ and $\norm{u_n}_{\Hbb(\Omega)} \to \sigma_1^a$. Since $\Hbb(\Omega)$ is reflexive, there exists a function $\varphi_1^a \in \Hbb(\Omega)$ such that $u_n \rightharpoonup \varphi_1^a$ in $\Hbb(\Omega)$. Since the embedding $\Hbb(\Omega) \hookrightarrow L^2(\Omega, a )$ is compact (Proposition \ref{prop_compactembd}), $u_n \to \varphi_1^a$ in $L^2(\Omega, a )$. Thus,
\begin{align*}
\sigma_1^a \leq \| \varphi_1^a \|_{\Hbb(\Omega)} \leq \liminf_{n \to \infty}\| u_n \|_{\Hbb(\Omega)} = \sigma_1^a,
\end{align*}
which, together with the fact that $\| \varphi_1^a \|_{L^2(\Omega,a)}=1$, implies that $\varphi_1^a$ is a minimizer of \eqref{weighteigen}. By a standard argument, we can show that $\varphi_1^a$ can be chosen to be nonnegative and $\varphi_1^a$ solves \eqref{eq_eigenvalue} in the sense of \eqref{weighteigen-1}. The proof is complete.
\end{proof}

\begin{lemma}\label{lem_uniquesolution} Let $0 \leq a \in L^r(\Omega,\delta^{\gamma} )$ with $r > \frac{N+\gamma}{2s}$ and $\sigma_1^a$ be the first eigenvalue defined by \eqref{weighteigen}. Assume that $\sigma_1^a >1$.  Then for any $f \in \Hbb^{-1}(\Omega)$, the problem 
\begin{equation}\label{eq_eigenvalue3}
\left\{ \begin{aligned}
 \L u &= a u + f &&\text{ in }\Omega, \\
u  &= 0 &&\text{ on }\partial \Omega, \\
u &=0 &&\text{ in }\R^N \backslash \overline{\Omega} \text{ (if applicable)},
\end{aligned} \right. 
\end{equation}
has a unique variational solution $u \in \Hbb(\Omega)$.
\end{lemma}
\begin{proof}We define the bilinear form $A: \Hbb(\Omega) \times \Hbb(\Omega) \to \R$ by
$$A(u,v):= \inner{u,v}_{\Hbb(\Omega)}  - \int_{\Omega} a uv dx,\quad \forall u,v \in \Hbb(\Omega).$$
It can be seen that $A$ is continuous since
\begin{align*}
|A(u,v)|  \le |\inner{u,v}_{\Hbb(\Omega)}| + \int_{\Omega} a |uv| dx  
 \le (1+ C^2\norm{a}_{L^r(\Omega,\delta^{\gamma} )})\norm{u}_{\Hbb(\Omega)} \norm{v}_{\Hbb(\Omega)},
\end{align*}
by using H\"older's inequality and estimate \eqref{embed-2}. Furthermore, $A$ is coercive. Indeed, from Proposition \ref{prop_eigenproblem}, we see that
\begin{equation} \label{weightHar}
\sigma_1^a \| \xi \|_{L^2(\Omega,a)}^2 \leq \| \xi \|_{\Hbb(\Omega)}^2, \quad \forall \xi \in \Hbb(\Omega).
\end{equation}	
By \eqref{weightHar} and assumption $\sigma_1^a >1$, we have
$$A(u,u) = \norm{u}_{\Hbb(\Omega)}^2 - \int_{\Omega} au^2 dx \ge \left( 1 - \frac{1}{\sigma_1^a}\right)\norm{u}^2_{\Hbb(\Omega)}, \quad \forall u \in \Hbb(\Omega).$$
Invoking Lax--Milgram theorem, we conclude that for every $f \in \Hbb^{-1}(\Omega)$, there exists a unique $u \in \Hbb(\Omega)$ such that 
$$A(u,\xi ) = \inner{ f, \xi },\quad \forall \xi \in \Hbb(\Omega).$$
This implies that \eqref{eq_eigenvalue3} has a unique solution $u \in \Hbb(\Omega)$. We complete the proof.
\end{proof}

\subsection{Key regularity results}
\begin{lemma} \label{lem_bootstrap1}
Assume $0 \leq a \in L^{r}(\Omega,\delta^\gamma )$ with $r>\frac{N+\gamma}{2s}$. Then there exists a constant $\sigma = \sigma(N,s,\gamma,r) > 1$ such that, for every $\ell>1$, the linear operator 
$$ \begin{aligned} \T: L^\ell(\Omega,a \delta^{\gamma} ) &\longrightarrow L^{\ell \sigma} (\Omega,a \delta^{\gamma} ),   \\
f &\mapsto \Gbb^{\Omega}[af]
\end{aligned}
$$
is continuous. Moreover, there exists a constant $C = C(\Omega,N,s,\gamma,r) > 0$ such that
\begin{equation} \label{T-cont}  \|\T[f] \|_{L^{\ell\sigma}(\Omega,a\delta^\gamma )} \leq C\norm{a}_{L^r(\Omega,\delta^{\gamma} )}^{1-\frac{1}{\ell}+\frac{1}{\ell\sigma}} \| f \|_{L^\ell(\Omega,a\delta^\gamma )}, \quad \forall f \in L^\ell(\Omega,a\delta^\gamma ).
\end{equation}
\end{lemma}
The following H\"older's inequality will be used several times in the proof of the lemma.
\begin{equation}\label{eq:Holder}
\left(\int_{\Omega} |v| d\tau \right)^q \le \left( \int_{\Omega} |v|^q d\tau \right)\left(\int_{\Omega} d\tau \right)^{q-1},
\end{equation}
for $q \ge 1$ and $\tau$ is a positive bounded measure.

\begin{proof} We will prove \eqref{T-cont}, which implies the continuity of the linear operator $\T$. Since $r  > \frac{N+\gamma}{2s}$, $r' < \frac{N+\gamma}{N+\gamma-2s} = p^*$. Choose $\sigma \in \left(1, \frac{p^*}{r'}\right)$. Recall from \eqref{eq_GLinfty} that
\begin{equation}\label{eq_Gaest}
\Gbb^{\Omega}[a](x) =\int_{\Omega} G^{\Omega}(x,y)a(y) dy < C_1(\Omega,N,s,\gamma)\norm{a}_{L^r(\Omega,\delta^{\gamma} )}.
\end{equation}
Using H\"older's inequality \eqref{eq:Holder} with $q=\ell$ and $d\tau = G^{\Omega}(x,y)a(y)dy$,  and \eqref{eq_Gaest}, we have
\begin{align*}
 |\T[f](x)|^{\ell\sigma} & \leq \left(\int_{\Omega} G^{\Omega}(x,y)|f(y)|a(y) dy\right)^{\ell\sigma}  \\
&\le C_2 \left( \int_{\Omega} G^{\Omega}(x,y)a(y) dy \right)^{(\ell-1)\sigma}\left(\int_{\Omega} G^{\Omega}(x,y) |f(y)|^\ell a(y) dy\right)^{\sigma} \\
&\le C_2\norm{a}_{L^r(\Omega,\delta^{\gamma} )}^{\sigma(\ell-1)} \left(\int_{\Omega} G^{\Omega}(x,y) |f(y)|^\ell a(y) dy\right)^{\sigma},
\end{align*}
where $C_2 = C_1^{\sigma(\ell-1)}$ with $C_1$ being the constant in \eqref{eq_Gaest}. Then applying again H\"older's inequality \eqref{eq:Holder} with $q=\sigma$ and $d\tau  = |f(y)|^\ell a(y)\delta(y)^{\gamma}dy$, we obtain
\begin{equation}\label{eq:keyestimate}
\begin{aligned}
 |\T[f](x)|^{\ell\sigma} &\le C_2\norm{a}_{L^r(\Omega,\delta^{\gamma} )}^{\sigma(\ell-1)} \left(\int_{\Omega} \dfrac{G^{\Omega}(x,y)}{\delta(y)^{\gamma}} |f(y)|^\ell a(y)\delta(y)^{\gamma} dy\right)^{\sigma} \\
&\le C_2 \norm{a}_{L^r(\Omega,\delta^{\gamma} )}^{\sigma(\ell-1)} \| f \|_{L^\ell(\Omega,a\delta^\gamma )}^{(\sigma-1)\ell}
\int_{\Omega}  \left(\dfrac{G^{\Omega}(x,y)}{\delta(y)^{\gamma}}\right)^{\sigma} |f(y)|^\ell a(y)\delta(y)^{\gamma} dy.
\end{aligned}
\end{equation}
Set
$$M_{a,f}:= \norm{a}_{L^r(\Omega,\delta^{\gamma} )}^{\sigma(\ell-1)} \| f \|_{L^\ell(\Omega,a\delta^\gamma )}^{(\sigma-1)\ell}.$$
By Fubini's theorem and \eqref{eq:keyestimate}, we have
\begin{align*}
\int_{\Omega} |\T[f](x)|^{\ell\sigma} a(x) \delta(x)^{\gamma} dx &\le C_2 M_{a,f} \int_{\Omega}  \int_{\Omega}  \left(\dfrac{G^{\Omega}(x,y)}{\delta(y)^{\gamma}}  \right)^{\sigma} |f(y)|^\ell a(y)\delta(y)^{\gamma}dy \, a(x) \delta(x)^{\gamma} dx \\
&=C_2 M_{a,f}\int_{\Omega} \left(\int_{\Omega} \left(\dfrac{G^{\Omega}(x,y)}{\delta(y)^{\gamma}}  \right)^{\sigma} a(x)\delta(x)^{\gamma} dx \right) |f(y)|^\ell a(y)\delta(y)^{\gamma} dy.
\end{align*}
Using H\"older's inequality again, we see that
\begin{equation}\label{eq_bootstrap2}
I(y) := \int_{\Omega} \left(\dfrac{G^{\Omega}(x,y)}{\delta(y)^{\gamma}}\right)^{\sigma} a(x)\delta(x)^{\gamma} dx \le  \norm{a}_{L^r(\Omega,\delta^{\gamma} )} \left(\int_{\Omega} \left(\dfrac{G^{\Omega}(x,y)}{\delta(y)^{\gamma}}\right)^{\sigma r'} \delta(x)^{\gamma} dx \right)^{\frac{1}{r'}}. 
\end{equation}
This estimate and Corollary \ref{G/delta} (notice that $\sigma r' < p^*$) yield 
\begin{equation} \label{Iest} I(y)  \le  \norm{a}_{L^r(\Omega,\delta^{\gamma} )} \left(\int_{\Omega} \left(\dfrac{G^{\Omega}(x,y)}{\delta(y)^{\gamma}}\right)^{\sigma r'} \delta(x)^{\gamma} dx \right)^{\frac{1}{r'}} \le C_3\norm{a}_{L^r(\Omega,\delta^{\gamma} )}.
\end{equation}
Using estimate \eqref{Iest} and taking into account the definition of $M_{a,f}$, we derive
$$  \int_{\Omega} |\T[f](x)|^{\ell\sigma} a(x) \delta(x)^{\gamma} dx \le C_2C_3 \norm{a}_{L^r(\Omega,\delta^{\gamma} )}^{\sigma(\ell-1)+1} \| f \|_{L^\ell(\Omega,a\delta^\gamma )}^{\ell\sigma}.$$
This implies \eqref{T-cont} with the constant $C=C(N,\Omega,s,\gamma,r)$.
\end{proof}

\begin{proposition}\label{prop_bootstrapsequence} Assume  $0 \leq a \in L^{r}(\Omega,\delta^\gamma )$ with $r>\frac{N+\gamma}{2s}$ and  $u \in L^\ell(\Omega,a\delta^{\gamma} )$ with $\ell > 1$. Let $\{u_n\}$ be the sequences given by
$$ u_0 =u, \quad u_{n+1} = \Gbb^{\Omega}[a u_n], \quad  \forall n\in \mathbb{N}.$$
Then there exists $n_0 \in \N$ such that $u_{n_0} \in L^{\infty}(\Omega)$. Moreover, 
$$ \|u_{n_0}\|_{L^\infty(\Omega)} \leq C_{n_0}\norm{a}_{L^r(\Omega,\delta^{\gamma} )}^{n_0 - \frac{1}{\ell}} \| u \|_{L^\ell(\Omega,a\delta^\gamma )}$$
for some constant $C_{n_0} = C_{n_0}(\Omega,N,s,\gamma,r) > 0$.
\end{proposition}

\begin{proof} Let $\{\ell_n\}$ be a sequence given by $\ell_0 = \ell$ and $\ell_{n+1} = \ell_n \sigma$ for  $n \in \N$.	
By the argument as in Lemma \ref{lem_bootstrap1}, $u_n \in L^{q_n}(\Omega,a \delta^{\gamma} )$ for any $n \in \N$ and
$$ \| u_{n} \|_{L^{\ell_{n}}(\Omega,a\delta^\gamma )} \le \tilde C\norm{a}_{L^r(\Omega,\delta^{\gamma} )}^{1-\frac{1}{\ell_{n-1}}+\frac{1}{\ell_{n}}} \| u_{n-1} \|_{L^{\ell_{n-1}}(\Omega,a\delta^\gamma )} $$
for some $\tilde C = \tilde C(\Omega,N,s,\gamma,r)$. This implies
\begin{equation}
\| u_{n} \|_{L^{\ell_{n}}(\Omega,a\delta^\gamma )} \le {\tilde C}^{n}\norm{a}_{L^r(\Omega,\delta^{\gamma} )}^{n-\frac{1}{\ell}+\frac{1}{\ell_{n}}} \| u \|_{L^\ell(\Omega,a\delta^\gamma )}.
\end{equation}
Using H\"older's inequality, we have
\begin{equation} \label{conqn0} \begin{aligned}
|u_{n+1}(x)| &\leq \int_{\Omega}G^{\Omega}(x,y)a(y)|u_n(y)|dy \\ 
&\leq \| u_{n} \|_{L^{\ell_{n}}(\Omega,a\delta^\gamma )} \left( \int_{\Omega} G^{\Omega}(x,y)^{\ell_n'}a(y)\delta(y)^{-\frac{\ell_n' \gamma}{\ell_n}} dy  \right)^{\frac{1}{\ell_n'}} \\ 
&\leq {\tilde C}^{n}\norm{a}_{L^r(\Omega,\delta^{\gamma} )}^{n-\frac{1}{\ell}+\frac{1}{\ell_{n}}} \| u \|_{L^\ell(\Omega,a\delta^\gamma )} \left( \int_{\Omega} G^{\Omega}(x,y)^{\ell_n'}a(y) \delta(y)^{-\frac{\ell_n' \gamma}{\ell_n}} dy  \right)^{\frac{1}{\ell_n'}} \\
&\leq {\tilde C}^{n}\norm{a}_{L^r(\Omega,\delta^{\gamma} )}^{n+1 -\frac{1}{\ell}} \| u \|_{L^\ell(\Omega,a\delta^\gamma )} \left( \int_{\Omega} G^{\Omega}(x,y)^{\ell_n' r'} \delta(y)^{-\gamma r'(\frac{1}{r} + \frac{\ell_n'}{\ell_n})} dy  \right)^{\frac{1}{\ell_n'r'}}.
\end{aligned} \end{equation}

\textbf{Claim.} Let $\alpha_1,\alpha_2 \ge 0$ be such that $\alpha_1 < \frac{N-\alpha_2}{N-2s}$ and $\alpha_2 \le \gamma \alpha_1$. There holds
$$\int_{\Omega} G^{\Omega}(x,y)^{\alpha_1}\delta(y)^{-\alpha_2}dy < C(\Omega,N,s,\gamma, \alpha_1,\alpha_2).$$
Indeed, one has
$$G^{\Omega}(x,y) \le C\frac{\delta(y)^{\frac{\alpha_2}{\alpha_1}}}{|x-y|^{N-2s+\frac{\alpha_2}{\alpha_1}}},  \quad \forall x,y \in \Omega, x \neq y,$$
which, together with the fact that $\alpha_1(N-2s) + \alpha_2 < N$ due to $\alpha_1 < \frac{N-\alpha_2}{N-2s}$, implies
$$\int_{\Omega} G^{\Omega}(x,y)^{\alpha_1} \delta(y)^{-\alpha_2} dy \le C \int_{\Omega} \dfrac{1}{|x-y|^{\alpha_1(N-2s)+\alpha_2}}dy < C.$$

Using the above claim with $\alpha_1=\ell_n'r'$ and $\alpha_2=\gamma r'(\frac{1}{r} + \frac{\ell_n'}{\ell_n})$, to show that
$$\int_{\Omega} G^{\Omega} (x,y)^{\ell_n' r'} \delta(y)^{-\gamma r'(\frac{1}{r} + \frac{\ell_n'}{\ell_n})} dy  < \infty,$$
it is sufficient to check
\begin{equation} \label{conqn1} 
r' \left(\frac{1}{r} + \frac{\ell_n'}{\ell_n}\right) \le \ell_n' r' < \frac{N}{N-2s} - \frac{\gamma r'}{N-2s} \left(\frac{1}{r} + \frac{\ell_n'}{\ell_n}\right)
\end{equation}
for some $n$ large enough.
The inequalities are equivalent to
\begin{equation} \label{conqn2}
r > 1\quad \text{ and } \quad \frac{2sr - (N+\gamma)}{r-1} \ell_n > \gamma\left(r-\frac{1}{r-1}\right),
\end{equation}
respectively. Since $r > \frac{N+\gamma}{2s}$ and $\ell_n \to \infty$ as $n \to \infty$, there exists $n_0 \in \N$ large enough such that \eqref{conqn2} holds for $n=n_0$. This means \eqref{conqn1} is fulfilled for $n=n_0$. 

Using \eqref{conqn0} with $n=n_0$, we conclude that $u_{n_0+1} \in L^{\infty}(\Omega)$ with
$$|u_{n_0+1}(x)| \le C\tilde C^{n_0}\norm{a}_{L^r(\Omega,\delta^{\gamma} )}^{n_0 + 1-\frac{1}{\ell}} \| u \|_{L^\ell(\Omega,a\delta^\gamma )}.$$
We complete the proof.
\end{proof}

\begin{proposition}\label{prop_bootstrap2}
Let $0 \leq a \in L^{r}(\Omega,\delta^\gamma )$ with $r>\frac{N+\gamma}{2s}$ and $f \in L^m(\Omega,\delta^\gamma )$ with $m>\frac{N+\gamma}{2s}$. Assume $u$ is a weak-dual solution of \eqref{eq_eigenvalue3}.
Then $u \in L^\infty(\Omega)$. 
\end{proposition}
\begin{proof}Since $u$ is a weak-dual solution of \eqref{eq_eigenvalue3}, we can write
$$u = \Gbb^{\Omega} [au]+\Gbb^{\Omega}[f] = \Gbb^{\Omega}[a \Gbb^{\Omega} [au]] + \Gbb^{\Omega}[a\Gbb^{\Omega}[f]] + \Gbb^{\Omega}[f]= \cdots.$$
In particular, we can write
$$u = v_n + w_1+\cdots + w_n, \quad n \in \mathbb{N},$$
where $v_0 = u$, $v_{n+1} = \Gbb^{\Omega}[a v_n]$, $n \in \mathbb{N}$ and $w_1 = \Gbb^{\Omega}[f], w_{n+1} = \Gbb^{\Omega}[aw_n]$, $n \in \mathbb{N}^*$.
It can be seen from estimate \eqref{eq_GLinfty} that $w_n \in L^{\infty}(\Omega)$ for all $n \in \mathbb{N}^*$ since
$$\norm{\Gbb^{\Omega}[f]}_{L^{\infty}(\Omega)} \le C \norm{f}_{L^m(\Omega,\delta^{\gamma} )}, \quad \forall f \in L^m(\Omega,\delta^{\gamma} ).$$
By Proposition \ref{prop_bootstrapsequence} there exists an $n_0$ such that $v_{n_0} \in L^{\infty}(\Omega)$. Thus $u \in L^{\infty}(\Omega)$.
\end{proof}

The next result provides a uniform upper bound for weak-dual solutions of \eqref{eq_nonlinear_source}.

\begin{proposition}\label{prop:boundonu} Assume that $u$ is a nonnegative weak-dual solution of \eqref{eq_nonlinear_source} for $\lambda > 0$. Then there exists a constant $C = C(\Omega,N,s,\gamma, \lambda_1)$ such that
\begin{equation}
\lambda \Gbb^{\Omega}[\mu] \le u \le C(\Gbb^{\Omega}[\mu] + 1) \text{ a.e. in }\Omega.
\end{equation}
\end{proposition}

\begin{proof}
Without loss of generality, assume that $\lambda = 1$ and  $u$ is a nonnegative weak-dual solution of \eqref{eq_nonlinear_source} with $\lambda=1$. By the representation $u = \Gbb^{\Omega} [\mu] + \Gbb^{\Omega} [ u^p ]$, it is easy to see that $u \geq \Gbb^{\Omega} [\mu]$ a.e. in $\Omega$. Using again the representation and the inequality $(a+b)^p \le 2^p (a^p + b^p) $, one has
\begin{align*}
u   = \Gbb^{\Omega}[\mu] + \Gbb^{\Omega} \left[ (\Gbb^{\Omega} [\mu] + \Gbb^{\Omega} [ u^p ])^p \right] 
 \le \Gbb^{\Omega}[\mu] + 2^p\left( \Gbb^{\Omega}[\Gbb^{\Omega}[\mu]^p] + \Gbb^{\Omega}[\Gbb^{\Omega}[ u^p ]^p]\right).
\end{align*}
Since $\Gbb^{\Omega}[\Gbb^{\Omega}[\mu]] \le c_p \Gbb^{\Omega}[\mu]$, one has
$u \le C \Gbb^{\Omega}[\mu] + 2^p\Gbb^{\Omega}[\Gbb^{\Omega}[u^p ]^p]$ where $C=C(\Omega,N,s,\gamma,p)$. 
Let $\{u_n\}$ be a sequence given by 
$$u_0 = u, \quad u_{n+1} = \Gbb^{\Omega}[ (u_n)^p], \quad n \in \mathbb{N}.$$
By induction, we obtain
\begin{equation} \label{reg-un} u \leq C_n(\Gbb^{\Omega}[\mu]+u_n), \quad \forall n \in \N,
\end{equation}
where $C_n$ depends on $\Omega,N,s,\gamma,p,n$. It is enough to prove that there exists an $n_0$ such that $u_{n_0}$ is bounded. To that end, we define a dominant sequence $\{v_n\}$ by
$$ v_0 = u, \quad v_{n+1} = \Gbb^{\Omega} [ u^{p-1} v_n], \quad \forall n \in \N.
$$
We observe that  $v_n \ge u_n$ for every $n \in \mathbb{N}$ since $u \ge u_n$ for every $n \in \mathbb{N}$. 

Fix $\ell \in (1,p^*-p+1)$. Applying Proposition \ref{prop_bootstrapsequence} for $r=p' > (p^*)'$ (due to $p<p^*$), $a=u^{p-1} \in L^{p'}(\Omega,\delta^\gamma )$ (since $u \in L^p(\Omega,\delta^\gamma )$) and the sequence $\{v_n\}$, one deduce that there exists $n_0 \in \N$ large enough such that
\begin{equation} \label{reg-vn}\| v_{n_0}\|_{L^{\infty}(\Omega)} \le C_{n_0}(\Omega,N,s,\gamma,p) \| u \|_{L^{p}(\Omega,\delta^\gamma )}^{(n_0-\frac{1}{\ell})(p-1)} \left( \int_{\Omega}|u|^{\ell+p-1}\delta^\gamma dx \right)^{\frac{1}{\ell}} < C.
\end{equation}
The second estimate follows from \eqref{univer-2} since for any $q \in (1,p^*)$, one has $u \in L^q(\Omega,\delta^\gamma )$ and $\| u \|_{L^q(\Omega,\delta^\gamma )}$ is bounded from above by a constant $C=C(N,\Omega,s,\gamma,\lambda_1,p,q)$.
Combining \eqref{reg-un}, \eqref{reg-vn} and the inequality $u_n \leq v_n$, we obtain $ u \leq C_{n_0} (\Gbb^{\Omega}[\mu]+1)$.
Thus we complete the proof.
\end{proof}

\section{Semilinear equations: uniqueness and multiplicity} \label{sec:uni+multi}
In this section, we assume that \eqref{eq_L1}--\eqref{eq_L3} and \eqref{eq_G1}--\eqref{eq_G3} hold. 

\subsection{Stability} We will prove the (semi-)stability of the minimal solution $\underline{u}_{\lambda}$ of \eqref{eq_nonlinear_source}. 

Let us recall the definition of stable solutions (see, for instance \cite[Definition 1.1.2]{Dup_2011}).
\begin{definition}A solution of $u$ of \eqref{eq_nonlinear_source} is stable (respectively, semistable) if
$$\norm{\xi}^2_{\Hbb(\Omega)} > (\text{resp.} \geq) \,\, p\int_{\Omega} u^{p-1} \xi^2 dx, \,  \quad \forall \xi \in \Hbb(\Omega) \setminus \{0\}.$$
\end{definition}

Recall that for every $\lambda \in (0,\lambda^*)$, \eqref{eq_nonlinear_source} admits a minimal nonnegative solution $\underline{u}_{\lambda} \in L^p (\Omega,\delta^{\gamma} )$, $p \in [1,p^*)$ (Theorem \ref{th:main1}). 

\begin{proposition}\label{prop_stability}For every $\lambda \in (0,\lambda^*)$, the minimal nonnegative solution $\underline{u}_{\lambda}$ of \eqref{eq_nonlinear_source} is stable.
\end{proposition}
\begin{proof} We first prove the stability of minimal solutions $\underline{u}_{\lambda}$, provided that $\lambda > 0$ is small enough. By Lemma \ref{lem_technicaleq} and Theorem \ref{th:main1}, one can see that $0 < \underline{u}_{\lambda} \le T(\lambda) \Gbb^{\Omega}[\mu]$
for all $\lambda > 0$ small enough, where $T$ is increasing and $T(\lambda) \to 0$ as $\lambda \to 0$. Hence, for $\xi \in \Hbb(\Omega) \setminus \{0\}$, we have
$$p\int_{\Omega} \xi^2 (\underline{u}_{\lambda})^{p-1} dx \le p\int_{\Omega} \xi^2 \left(T(\lambda) \Gbb^{\Omega}[\mu]\right)^{p-1} dx \le pT(\lambda)^{p-1}\int_{\Omega} \xi^2 \Gbb^{\Omega}[\mu]^{p-1}dx.$$
Since $\Gbb^{\Omega}[\mu] \in L^q(\Omega,\delta^{\gamma} )$ for any $q \in [1, p^*)$, one has $\Gbb^{\Omega}[\mu]^{p-1} \in L^r(\Omega,\delta^{\gamma} )$ for some $r > \frac{N+\gamma}{2s}$. Thus, the embedding $\Hbb(\Omega) \hookrightarrow L^2(\Omega,\Gbb^{\Omega}[\mu]^{p-1} )$ is continuous and also compact (see Proposition \ref{prop_compactembd}). This implies
$$p\int_{\Omega} \xi^2 (\underline{u}_{\lambda})^{p-1} dx \le pT(\lambda)^{p-1}\int_{\Omega} \xi^2 \Gbb^{\Omega}[\mu]^{p-1}dx \le CT(\lambda)^{p-1}\norm{\xi}_{\Hbb(\Omega)}^2 <\norm{\xi}_{\Hbb(\Omega)}^2,$$
if $\lambda > 0$ is small enough. We conclude that in this case, $\underline{u}_{\lambda}$ is stable.

Suppose there exists a $\lambda \in (0,\lambda^*)$ such that $\underline{u}_{\lambda}$ is not stable. Then one has
$$\sigma(\lambda) := \inf_{\xi \in \Hbb(\Omega), \xi \neq 0} \dfrac{\norm{\xi}^2_{\Hbb(\Omega)} }{ \int_{\Omega} p(\underline{u}_{\lambda})^{p-1} \xi^2dx } \le 1.$$
By applying Proposition \ref{prop_eigenproblem} with $a=(\underline{u}_{\lambda})^{p-1} \in L^{p'}(\Omega,\delta^\gamma )$, we deduce that there exists a minimizer $0 \le \xi_1 \in \Hbb(\Omega)$, which means 
$$\sigma(\lambda) = \dfrac{\norm{\xi_1}^2_{\Hbb(\Omega)} }{ \int_{\Omega} p(\underline{u}_{\lambda})^{p-1} \xi_1^2dx }.$$
Moreover, there holds
\begin{equation}\label{eq_stable1}
\inner{\xi_1, \zeta}_{\Hbb(\Omega)} = \sigma (\lambda)  \int_{\Omega} p(u_\lambda)^{p-1} \xi_1 \zeta dx, \quad \forall \zeta \in \Hbb(\Omega).
\end{equation}
Choose $\widetilde{\lambda} \in (\lambda, \lambda^*)$ and denote by $\underline{u}_{\widetilde{\lambda}}$ the  minimal solution of \eqref{eq_nonlinear_source} for parameter $\tilde \lambda$. Set $w := \underline{u}_{\widetilde{\lambda}} - \underline{u}_{\lambda} \ge 0$. Due to Proposition \ref{prop_priorestimate}, one has
\begin{align*}
w = \Gbb^{\Omega}[(\underline{u}_{\widetilde{\lambda}})^p - (\underline{u}_{\lambda})^p] + (\widetilde{\lambda} - \lambda)\Gbb^{\Omega} [\mu] \ge \Gbb^{\Omega} [p(\underline{u}_{\lambda})^{p-1} w]+ c(\widetilde{\lambda} - \lambda)\Gbb^{\Omega} [\Gbb^{\Omega}[\mu]].
\end{align*}
By noticing that $\mu \in \M^+(\Omega,\delta^{\gamma} )$ with $\norm{\mu}_{\M(\Omega,\delta^{\gamma} )} = 1$, one has
$$\Gbb^{\Omega}[\mu](x) = \int_{\Omega} G^{\Omega} (x,y) d\mu(y) \gtrsim \int_{\Omega} \delta(x)^{\gamma} \delta(y)^{\gamma} d\mu(y) = \delta(x)^{\gamma}, \quad x\in \Omega,$$
which implies $\Gbb^{\Omega}[\Gbb^{\Omega}[\mu]] \gtrsim \Gbb^{\Omega}[\delta^{\gamma}]$ a.e. in $\Omega$. This gives $w \ge \Gbb^{\Omega} [p(\underline{u}_{\lambda})^{p-1} w] + \tilde c \, \Gbb^{\Omega} [\delta^{\gamma}]$
for some positive constant $\tilde c$. 
On the other hand, the function $v = \tilde c \, \Gbb^{\Omega} [\delta^{\gamma}]$ satisfying $v \leq w$ and obviously, $v \le \Gbb^{\Omega} [p(\underline{u}_{\lambda})^{p-1} v] + \tilde c \, \Gbb^{\Omega} [\delta^{\gamma}]$.
We deduce from Proposition \ref{prop_existence.nonlinear} that there exists a weak-dual solution $\tilde u \in L^1(\Omega,\delta^{\gamma} )$ of the problem
\begin{equation}\label{eq_stable2}
\left\{ \begin{aligned}
\L \tilde u &= p(\underline{u}_{\lambda})^{p-1} \tilde u + \tilde c \, \delta^{\gamma} &&\text{ in }\Omega, \\ 
u  &= 0 &&\text{ on }\partial \Omega, \\
u &=0 &&\text{ in }\R^N \backslash \overline{\Omega} \text{ (if applicable)}.
\end{aligned} \right. 
\end{equation}
Since $0 \leq \tilde u \leq w$, it follows that $\tilde u \in L^{q}(\Omega,\delta^{\gamma} )$ for any $q < p^*$. Furthermore, since $\underline{u}_{\lambda} \in L^q(\Omega,\delta^{\gamma} )$ for any $q < p^*$ by Remark \ref{rmk:uniformbound}, we have $(\underline{u}_{\lambda})^{p-1} \in L^r(\Omega,\delta^{\gamma} )$ for some $r > \frac{N+\gamma}{2s}$. Corollary \ref{prop_bootstrap2} shows that  $\tilde u \in L^{\infty}(\Omega)$.

Under assumption \eqref{eq_G3}, i.e. $N \geq N_{s,\gamma}$, we prove that $\tilde u$ belongs to $\Hbb(\Omega)$. It is enough to demonstrate that the right-hand side of the equation in \eqref{eq_stable2} belongs to $L^2(\Omega)$.  Since $\tilde u,\, \delta^\gamma \in L^\infty(\Omega)$, we need to show that $(\underline{u}_{\lambda})^{p-1} \in L^2(\Omega)$.  To this purpose, we will consider two cases.
 
\textbf{Case 1:} $\gamma \geq 2s$. One has $N_{s,\gamma}=4s(\gamma+1)-\gamma$. Since $N  \geq N_{s,\gamma}$ (which implies $0<\frac{2\gamma}{r-2}<1$) and  $(\underline{u}_{\lambda})^{p-1} \in L^r(\Omega,\delta^{\gamma} )$, it follows that
 \begin{align*}
\int_{\Omega}|\underline{u}_{\lambda}|^2 dx = \int_{\Omega} |\underline{u}_{\lambda}|^2\delta^{\frac{2\gamma}{r}} \cdot \delta^{-\frac{2\gamma}{r}}dx \leq \left(\int_{\Omega} |\underline{u}_{\lambda}|^r \delta^{\gamma}dx \right)^{\frac{2}{r}} \left(\int_{\Omega} \delta^{-\frac{2\gamma}{r-2}} dx\right)^{1-\frac{2}{r}} < +\infty.
\end{align*}
Hence $(\underline{u}_{\lambda})^{p-1} \in L^2(\Omega)$. 

\textbf{Case 2:} $\gamma < 2s$. One has $N_{s,\gamma}=4s$. By Remark \ref{gam<2s}, $\underline{u}_{\lambda} \in L^q(\Omega)$ for $q < \frac{N}{N-2s+\gamma}$, which implies  $(\underline{u}_{\lambda})^{p-1} \in L^r(\Omega)$ for some $r > \frac{N}{2s}$. Since $N \geq N_{s,\gamma}=4s$, we derive $(\underline{u}_{\lambda})^{p-1} \in L^2(\Omega)$. 

In both cases, when $N \geq N_{s,\gamma}$, we have $\Gbb^{\Omega}[ (\underline{u}_{\lambda})^{p-1}\tilde u] \in \Hbb(\Omega)$ and hence $\tilde u \in \Hbb(\Omega)$. Choosing $\xi_1 \in \Hbb(\Omega)$ as a test function in \eqref{eq_stable2} and $\tilde u$ as a test function in \eqref{eq_stable1}, one has
\begin{align*}
\int_{\Omega} p(\underline{u}_{\lambda})^{p-1} \tilde u \xi_1 dx + \tilde c \int_{\Omega} \xi_1 \delta^\gamma dx = \langle \tilde u, \xi_1 \rangle_{\Hbb(\Omega)} =  \sigma(\lambda) \int_{\Omega} p(\underline{u}_{\lambda})^{p-1} \xi_1 \tilde u dx,
\end{align*}
which is a contradiction since we assumed that $\sigma(\lambda) \le 1$. We conclude that $\underline{u}_{\lambda}$ is stable for all $\lambda \in (0,\lambda^*)$.
\end{proof}

\subsection{Uniqueness of the extremal solution}\label{sec:uniqueness}

 We study the  uniqueness of the extremal solution of \eqref{eq_nonlinear_source}, which is the minimal solution in the case $\lambda = \lambda^*$.

\begin{lemma}\label{lem_solutionHs} Assume that \eqref{eq_nonlinear_source} has a nonnegative weak-dual solution $u_{\lambda}$, $\lambda \in (0,\lambda^*]$. Denote by $\left\{ u_{\lambda,n} \right\}$, $n \in \mathbb{N}$, $\lambda > 0$ the sequence as in \eqref{eq_sequence_minimal}. Then there exists an $n_0$ such that for all $n \ge n_0$, one has
\begin{align*}
u_{\lambda} - u_{\lambda,n} &\in L^{\infty}(\Omega), \quad \forall \lambda \in (0,\lambda^*],\\
u_{\lambda,n+1} - u_{\lambda,n} &\in L^{\infty}(\Omega), \quad \forall \lambda > 0.
\end{align*}
Furthermore, if $N \geq N_{s,\gamma}$, then 
\begin{align*}
u_{\lambda} - u_{\lambda,n} &\in \Hbb(\Omega), \quad \forall \lambda \in (0,\lambda^*],\\
u_{\lambda,n+1} - u_{\lambda,n} &\in \Hbb(\Omega), \quad \forall \lambda > 0.
\end{align*}
\end{lemma}

\begin{proof} For $\lambda \le \lambda^*$, it can easily be seen that
$$u_{\lambda} - u_{\lambda,{n+1}} =  \Gbb^{\Omega} \left[ (u_{\lambda} )^p - (u_{n,\lambda})^p\right] \le \Gbb^{\Omega} \left[ p(u_{\lambda})^{p-1} \left(u_{\lambda} - u_{n,\lambda}\right)\right],\quad \forall n \in \mathbb{N}.$$
Consider the sequence  $w_{\lambda,0} = u_{\lambda} - u_{\lambda,0} $ and $w_{\lambda,n+1} = \Gbb^{\Omega} [p(u_{\lambda})^{p-1} w_{\lambda,n} ]$ for  $n \in \mathbb{N}$. 
It can be seen that $w_{\lambda,0} \in L^q(\Omega,\delta^{\gamma} )$ for all $q \in \left[1,p^*\right)$ and $(u_{\lambda})^{p-1} \in L^r(\Omega,\delta^{\gamma}), r > \frac{N+\gamma}{2s}$. By Proposition \ref{prop_bootstrapsequence}, there exists an $n_0$ such that $w_{\lambda,n} \in L^{\infty}(\Omega)$ for all $n \ge n_0$. Since $0 \le u_{\lambda} - u_{\lambda,n}  \le w_{\lambda,n}$ for all $n \in \mathbb{N}$, we conclude that $u_{\lambda} - u_{\lambda,n} \in L^{\infty}(\Omega)$ for all $n \ge n_0$. 
The second statement can be proceeded similarly by noticing that
$$u_{\lambda,n+1} - u_{\lambda,n} = \Gbb^{\Omega} [ (u_{\lambda,n})^p - (u_{\lambda, n -1})^p],\quad \forall n \in \mathbb{N}.$$
Finally, if $N \geq N_{s,\gamma}$, then by the same argument as in the proof of Proposition \ref{prop_stability} we have $u_{\lambda} - u_{\lambda,n}  \in \Hbb(\Omega)$ and $u_{\lambda,n+1} - u_{\lambda,n} \in \Hbb(\Omega)$.
The proof is complete.
\end{proof}

%

\begin{proposition}\label{prop_eigenvalue1} The minimal positive weak-dual solution $\underline{u}_{\lambda^*}$ of \eqref{eq_nonlinear_source} with $\lambda = \lambda^*$  is semi-stable. Moreover, $\sigma(\lambda^*)=1$ where
\begin{equation} \label{sigma}
\sigma(\lambda^*) := \inf_{\xi \in \Hbb(\Omega) \setminus \{0\}} \dfrac{\norm{\xi}^2_{\Hbb(\Omega)} }{ \int_{\Omega} p(\underline{u}_{\lambda^*})^{p-1} \xi^2dx }.
\end{equation}
\end{proposition}

\begin{proof}

\textbf{Step 1.} We prove the semi-stability of $\underline{u}_{\lambda^*}$.  Since $\underline{u}_{\lambda}$ is stable for all $\lambda \in (0,\lambda^*)$ (Proposition \ref{prop_stability}), we know that
\begin{equation}\label{eq_extremal_2}
\norm{\xi}^2_{\Hbb(\Omega)} \ge \sigma({\lambda}) \int_{\Omega} p(\underline{u}_{\lambda})^{p-1} \xi^2 dx > \int_{\Omega} p(\underline{u}_{\lambda})^{p-1} \xi^2 dx,\quad \forall \xi \in \Hbb(\Omega) \backslash \{ 0 \}.
\end{equation}
Letting $\lambda \nearrow \lambda^*$ in \eqref{eq_extremal_2}, one has
\begin{equation} \label{semistable} \norm{\xi}^2_{\Hbb(\Omega)} \ge  \int_{\Omega} p(\underline{u}_{\lambda^*})^{p-1} \xi^2 dx, \quad \forall \xi \in \Hbb(\Omega) \backslash \{0\}, 
\end{equation}
which means $\underline{u}_{\lambda^*}$ is semi-stable. 

\textbf{Step 2.} Let $\sigma(\lambda^*)$ be defined in \eqref{sigma}. We infer from \eqref{semistable} that $\sigma(\lambda^*) \geq 1$. We will show that $\sigma(\lambda^*)=1$. Suppose by contradiction that $\sigma(\lambda^*) > 1$. Let $u_{\lambda,n_0}$ be the function given in Lemma \ref{lem_solutionHs}. Consider the problem 
\begin{equation}\label{eq_secondsolution}
\left\{ \begin{aligned}
\L w &= (w+ u_{\lambda,n_0+1})^p - (u_{\lambda,n_0})^p &&\text{ in }\Omega, \\  
u  &= 0 &&\text{ on }\partial \Omega, \\
u &=0 &&\text{ in }\R^N \backslash \overline{\Omega} \text{ (if applicable)},
\end{aligned} \right. 
\end{equation}
and the map 
$$F(w,\lambda) := \L w - (w+u_{\lambda,n_0+1})^p + (u_{\lambda,n_0})^p,\quad w \in \Hbb(\Omega),\lambda > 0.$$

It can be seen that $F: (0,\infty) \times \Hbb(\Omega) \to \Hbb^{-1}(\Omega)$. Indeed, since $N \geq N_{s,\gamma}$, one has $u_{\lambda, n_0+1} - u_{\lambda,n_0} \in \Hbb(\Omega)$ by Lemma \ref{lem_solutionHs}. Furthermore, $w+u_{\lambda,n_0+1} \in L^r(\Omega,\delta^{\gamma} )$, $r > \frac{N+\gamma}{2s}$. Thus, since
$$0 \le (w+u_{\lambda,n_0+1})^p - (u_{\lambda,n_0})^p \le p(w+u_{\lambda,n_0+1})^{p-1}(w+u_{\lambda,n_0+1}-u_{\lambda,n_0})$$
and $ p(w+u_{\lambda,n_0+1})^{p-1}(w+u_{\lambda,n_0+1}-u_{\lambda,n_0}) \in \Hbb^{-1}(\Omega)$ by the embedding in Proposition \ref{prop_compactembd}, we conclude that $(w+u_{\lambda,n_0+1})^p - (u_{\lambda,n_0})^p \in \Hbb^{-1}(\Omega)$.

One can see that $F$ is a differentiable map with 
\begin{equation}\label{eq_derivative}
\begin{aligned}
&\dfrac{\partial F}{\partial w}(w,\lambda) :\Hbb(\Omega) \times (0,\infty) \to \Hbb^{-1}(\Omega)\\
&\dfrac{\partial F}{\partial w}(w,\lambda)v = \L v - p(w+ u_{\lambda,n_0+1})^{p-1}v, \quad v \in \Hbb(\Omega),\lambda > 0.
\end{aligned}
\end{equation}
In \eqref{eq_derivative}, choose $\lambda:= \lambda^*$ and $w^*:= \underline{u}_{\lambda^*} - u_{\lambda^*, n_0 +1} \in \Hbb(\Omega)$, by Lemma \ref{lem_solutionHs} and the fact that $N \geq N_{s,\gamma}$. Since $\sigma(\lambda^*) > 1$,  using Lemma \ref{lem_uniquesolution}, it can be seen that
$$\dfrac{\partial F}{\partial w}(w^*,\lambda^*)v = \L v - p(\underline{u}_{\lambda^*})^{p-1}v, \quad v \in \Hbb(\Omega)$$
realizes an isomorphism from $\Hbb(\Omega)$ to $\Hbb^{-1}(\Omega)$. By the Implicit Function Theorem, there exists an $\varepsilon > 0$ such that \eqref{eq_secondsolution} has a solution $w_{\lambda} \in \Hbb(\Omega)$ for any $\lambda \in (\lambda^*-\varepsilon,\lambda^*+\varepsilon)$. This solution can be written as $w_{\lambda} = \Gbb^{\Omega} [ (w_{\lambda}+ u_{\lambda,n_0+1})^p] - \Gbb^{\Omega}[ (u_{\lambda,n_0})^p]$ by Remark \ref{rmk_welldefined}. By the continuity of the map $\lambda \mapsto w_{\lambda}$, one has $w_{\lambda}  \ge 0$ for any $\lambda \in ( \lambda^* - \varepsilon,\lambda^*+\varepsilon)$. Furthermore,
\begin{align*}
w_{\lambda} + u_{\lambda,n_0+1} &= \Gbb^{\Omega} [ (w_{\lambda}+ u_{\lambda,n_0+1})^p] + u_{\lambda,n_0+1}- \Gbb^{\Omega}[ (u_{\lambda,n_0})^p]  \\
& = \Gbb^{\Omega} [ (w_{\lambda}+ u_{\lambda,n_0+1})^p] + \lambda\Gbb^{\Omega}[\mu],
\end{align*}
for some $\lambda > \lambda^*$, which is a contradiction since \eqref{eq_nonlinear_source} has no positive solution for $\lambda > \lambda^*$. We conclude that $\sigma(\lambda^*) = 1$.
\end{proof}

\begin{proof}[\sc Proof of Theorem \ref{th:main2}(1)] Let $u \ge \underline{u}_{\lambda^*}$ be a nonnegative solution of \eqref{eq_nonlinear_source} with $\lambda = \lambda^*$. We will show that $u \equiv \underline{u}_{\lambda^*}$. Suppose by contradiction that $u > \underline{u}_{\lambda^*}$ in a set of positive measure. Using Proposition \ref{lem_solutionHs}, $u - \underline{u}_{\lambda^*} \in \Hbb(\Omega)$ and
$$u  - \underline{u}_{\lambda^*}   = \Gbb^{\Omega} [u^p - (\underline{u}_{\lambda^*})^p] > \Gbb^{\Omega} [ p(\underline{u}_{\lambda^*})^{p-1}(u - \underline{u}_{\lambda^*})],$$
by the maximal principle (see Lemma \ref{lem_maximum}). Thus,
$$\int_{\Omega}p(\underline{u}_{\lambda^*})^{p-1}\xi_1(u - \underline{u}_{\lambda^*}) dx = \int_{\Omega} (u - \underline{u}_{\lambda^*}) \xi_1 dx > \int_{\Omega} p(\underline{u}_{\lambda^*})^{p-1}\xi_1 (u - \underline{u}_{\lambda^*})dx, $$
where $\xi_1$ is the solution to \eqref{eq_stable1}, which is a contradiction. We conclude that \eqref{eq_nonlinear_source} has a unique solution for $\lambda = \lambda^*$.
\end{proof}

\subsection{Multiplicity}\label{sec:secondsolution}
In this subsection, we look for a second solution of \eqref{eq_nonlinear_source} for $\lambda \in (0,\lambda^*)$. 
To this end, we prove the existence of nontrivial solutions of the problem 
\begin{equation}\label{eq_mountainpass}
\left\{ \begin{aligned}
\L v &= (\underline{u}_{\lambda} + v^+)^p - (\underline{u}_{\lambda})^p := h(\underline{u}_{\lambda},v) &&\text{ in }\Omega,\\
u  &= 0 &&\text{ on }\partial \Omega, \\
u &=0 &&\text{ in }\R^N \backslash \overline{\Omega} \text{ (if applicable)}.
\end{aligned} \right. 
\end{equation}

Recall that a function $v \in \Hbb(\Omega)$ is a variational solution of \eqref{eq_mountainpass} if \begin{equation}\label{eq_variational-source}
\begin{aligned}
\inner{v,\xi}_{\Hbb(\Omega)} = \int_{\Omega} h(\underline{u}_{\lambda},v)  \xi dx, \quad \forall \xi \in \Hbb(\Omega).
\end{aligned}
\end{equation}
  
It can easily be seen that a variational solution of \eqref{eq_mountainpass} is a critical point of the natural functional $J: \Hbb(\Omega) \to \R$ associated to \eqref{eq_variational-source} which is given by
$$J(v) := \frac{1}{2} \norm{v}_{\Hbb(\Omega)}^2  - \int_{\Omega} H(\underline{u}_{\lambda},v) dx, \quad \forall v \in \Hbb(\Omega),$$
where
\begin{align*}
H(a,b) := \int_0^b h(a,t)dt = \frac{1}{p+1}\left[(a+b^+)^{p+1} - a^{p+1}\right] - a^p b^+,\quad a\ge 0, b \in \R.
\end{align*}

Let us recall a technical lemma in \cite{NaiSat_2007} which will be used in the sequel.
\begin{lemma}[{\cite[Lemma C.2]{NaiSat_2007}}]\label{lem_naitosato}\text{}
\begin{enumerate}
\item For all $\varepsilon > 0$, there exists a $c_{\varepsilon} > 0$ such that
\begin{equation}\label{eq_estimateonH}
H(a,b) \le \left(\frac{1+\varepsilon}{2}\right) pa^{p-1} b^2 + c_{\varepsilon} b^{p+1}, \quad \forall a,b \ge 0.
\end{equation}
\item There holds
\begin{equation}\label{eq_estimateonH_1}
\frac{1}{p+1}a^{p+1} \le H(a,b) \le \frac{1}{2}h(a,b)a, \quad \forall a,b \ge 0.
\end{equation}
\item There exists a constant $c_p>0$ such that
\begin{equation}\label{eq_estimateonH_2}
h(a,b)b - (2+c_p)H(a,b) \ge -\dfrac{c_p p}{2}a^{p-1}b^2, \quad \forall a,b \ge 0.
\end{equation}
\end{enumerate}
\end{lemma}

To prove the existence of a nontrivial solution of \eqref{eq_mountainpass}, we invoke the Mountain Pass Theorem (see \cite[Theorem 6.1]{Str_2010}, \cite{AmbRab_1973}). 

\begin{theorem}\label{theo_mountainpas} Problem \eqref{eq_mountainpass} admits a nontrivial variational solution.
\end{theorem}

\begin{proof}\textbf{Step 1.} We first prove that $J$ has the mountain pass geometry.

We observe that $J(0) = 0$ and
\begin{align*}
J(v)  &= \frac{1}{2} \left[ \norm{v}_{\Hbb(\Omega)}^2 - \int_{\Omega} p(\underline{u}_{\lambda})^{p-1} v^2 dx \right] - \left[\int_{\Omega} H(\underline{u}_{\lambda}, v)dx - \frac{1}{2}\int_{\Omega} p(\underline{u}_{\lambda})^{p-1} v^2 dx \right]\\
& \ge \frac{1}{2} \left(1 - \frac{1}{\sigma(\lambda)} \right)\norm{v}^2_{\Hbb(\Omega)} - \left(\dfrac{\varepsilon }{2} \int_{\Omega} p(\underline{u}_{\lambda})^{p-1} v^2 dx  + c_{\varepsilon} \int_{\Omega} |v|^{p+1} dx \right)\\
& \ge \frac{1}{2} \left(1 - \frac{1}{\sigma(\lambda)} - \varepsilon \right)\norm{v}^2_{\Hbb(\Omega)} - c_{\varepsilon} \norm{v}_{L^{p+1}(\Omega)}^{p+1} \\
& \ge \frac{1}{2} \left(1 - \frac{1}{\sigma(\lambda)} - \varepsilon \right)\norm{v}^2_{\Hbb(\Omega)} - c_{\varepsilon} \norm{v}_{\Hbb(\Omega)}^{p+1},
\end{align*}
where we have used \eqref{eq_estimateonH} and the stability of $\underline{u}_{\lambda}$ in the second line, and the Sobolev inequality in the last line. Choosing a fixed $\varepsilon > 0$ small enough and noticing that $\sigma(\lambda) > 1$ since $\underline{u}_{\lambda}$ is stable, we have
$$J(v) \ge c_1 \norm{v}^2_{\Hbb(\Omega)} - c_2\norm{v}_{\Hbb(\Omega)}^{p+1}$$
for some $c_1,c_2 > 0$. Fix $\norm{v}_{\Hbb(\Omega)} = 1$. Since $p+1 > 2$, we know that there exists $t_0$ such that $J(t_0 v) := \beta > 0$. Next, let $v_0 \in \Hbb(\Omega)$, $v_0 \ge 0$ such that $\norm{v_0}_{\Hbb(\Omega)} = 1$. Using \eqref{eq_estimateonH_1}, we have
\begin{align*}
J(tv_0) < \frac{t^2}{2} \norm{v_0}^2_{\Hbb(\Omega)} - \frac{t^{p+1}}{p+1} \norm{v_0}^{p+1}_{L^{p+1}(\Omega)}= \frac{t^2}{2} - \frac{t^{p+1}}{p+1}\norm{v_0}^{p+1}_{L^{p+1}(\Omega)}.
\end{align*}
If we choose $t_0 > 0$ large enough then $J(t v_0) \leq 0$ for all $t \ge t_0$.

\textbf{Step 2.} We prove that $J$ satisfies (PS) condition. Recall that a functional $J$ satisfies (PS) condition if for any sequence $\{v_n\}$ in $\Hbb(\Omega)$ satisfying $J(v_n) \to c$ and $J'(v_n) \to 0$ as $n \to \infty$, there exists a convergent subsequence of $\{ v_n\}$.

Let $\{v_n\} \subset \Hbb(\Omega)$ satisfying $J(v_n) \to c$ and $J'(v_n) \to 0$ as $n \to \infty$. Then there exist an $n_0 \in \mathbb{N}$ and $c_3,c_4 > 0$ such that
\begin{equation}\label{eq_mpt1}
\frac{1}{2}\norm{v_n}^2_{\Hbb(\Omega)} - \int_{\Omega} H(\underline{u}_{\lambda}, v_n) dx = J(v_n) \le c_3,
\end{equation}
and
\begin{equation}\label{eq_mpt2}
\inner{v_n,w}_{\Hbb(\Omega)} - \int_{\Omega} h(\underline{u}_{\lambda}, v_n) w dx = \inner{J'(v_n),w}_{\Hbb^{-1}(\Omega),\Hbb(\Omega)}\le c_4 \norm{w}_{\Hbb(\Omega)}
\end{equation}
 for all $n \ge n_0$ and $w \in \Hbb(\Omega)$. Multiplying both sides of \eqref{eq_mpt1} by $(2+c_p)$, where $c_p$ is given in Lemma \ref{lem_naitosato}(3), and choosing $w= -v_n \in \Hbb(\Omega)$ in \eqref{eq_mpt2}, one has
\begin{equation}\label{eq_mpt3bis}
\begin{aligned}
\frac{1}{2}(2+c_p)\norm{v_n}^2_{\Hbb(\Omega)} - (2+c_p)\int_{\Omega} H(\underline{u}_{\lambda}, v_n) dx &\le c_3(2+c_p),\\
-\inner{v_n,v_n}_{\Hbb(\Omega)} + \int_{\Omega} h(\underline{u}_{\lambda}, v_n) v_n dx &= -\inner{v_n,v_n}_{\Hbb(\Omega)} + \int_{\Omega} h(\underline{u}_{\lambda}, v_n) (v_n)^+ dx  \\ &\le c_4 \norm{v_n}_{\Hbb(\Omega)}.
\end{aligned}
\end{equation}
Combining \eqref{eq_mpt3bis} with Lemma \ref{lem_naitosato}(3) yields
\begin{equation}\label{eq_mpt3.2}
\frac{1}{2}(2+c_p)\norm{v_n}^2_{\Hbb(\Omega)} -\inner{v_n,v_n}_{\Hbb(\Omega)} -\dfrac{c_p}{2} \int_{\Omega}p(\underline{u}_{\lambda})^{p-1} v_n^2 dx \le c_4 \norm{v_n}_{\Hbb(\Omega)} + c_3(2+c_p).
\end{equation}
Since $\underline{u}_{\lambda}$ is stable, we have
\begin{equation}\label{eq_mpt3.3}
\dfrac{c_p}{2}\left(1 - \frac{1}{\sigma(\lambda)}\right) \norm{v_n}_{\Hbb(\Omega)}^2 \le \dfrac{c_p}{2} \left[ \norm{v_n}^2_{\Hbb(\Omega)} - \int_{\Omega} p(\underline{u}_{\lambda})^{p-1} v_n^2 dx \right].
\end{equation}
By \eqref{eq_mpt3.2} and \eqref{eq_mpt3.3}, we deduce that
$$\dfrac{c_p}{2}\left(1 - \frac{1}{\sigma(\lambda)}\right) \norm{v_n}_{\Hbb(\Omega)}^2  \le c_4 \norm{v_n}_{\Hbb(\Omega)} + c_3(2+c_p).$$
This, together with the fact that $\sigma(\lambda) > 1$,  implies $\{v_n\}$ is bounded in $\Hbb(\Omega)$. Thus, since $2< p+1 < \frac{2N}{N-2s}$, there exist a subsequence of $\{v_n\}$, still denoted by $\{v_n\}$, and a function $v \in \Hbb(\Omega)$ such that
\begin{align*}
v_n \rightharpoonup v &\text{ in }\Hbb(\Omega),\\
v_n \rightarrow v &\text{ in }L^2(\Omega,(\underline{u}_{\lambda})^{p-1} ) \text{ and  in }L^{p+1}(\Omega), \\
v_n \to v &\text{ a.e. in }\Omega.
\end{align*} 
We will prove that
\begin{equation}\label{eq_mpt4}
\norm{v_n}_{\Hbb(\Omega)} \to \norm{v}_{\Hbb(\Omega)},
\end{equation} 
from which we conclude that $v_n \to v$ in $\Hbb(\Omega)$. One has
\begin{equation}\label{eq_mpt5}
\inner{ J'(v_n), v_n - v }_{\Hbb^{-1}(\Omega),\Hbb(\Omega)} = \inner{v_n, v_n -v}_{\Hbb(\Omega)} - \int_{\Omega} h(\underline{u}_{\lambda},v_n)(v_n - v) dx.
\end{equation}
On the one hand, since $J'(v_n) \to 0$ and $v_n \rightharpoonup v$ in $\Hbb(\Omega)$ as $n \to \infty$, we derive
\begin{equation} \label{J'} \lim_{n \to \infty} \langle J'(v_n), v_n-v \rangle_{\Hbb^{-1}(\Omega),\Hbb(\Omega)} = 0. 
\end{equation}
On the other hand, we infer from the inequality $h(s,t) \le C(s^{p-1}|t| + |t|^p)$ that
\begin{align*}
\left|\int_{\Omega} h(\underline{u}_{\lambda},v_n)(v_n - v) dx \right| &\le C \int_{\Omega} \left[ (\underline{u}_{\lambda})^{p-1} |v_n| + |v_n|^p \right] |v_n - v|dx \\
&\le C \left(\int_{\Omega} |v_n - v|^2 (\underline{u}_{\lambda})^{p-1} dx \right)^{\frac{1}{2}} \left( \int_{\Omega} |v_n|^2 (\underline{u}_{\lambda})^{p-1} dx \right)^{\frac{1}{2}}    \\& + C \left(\int_{\Omega} |v_n|^{p+1} dx \right)^{\frac{p}{p+1}} \left(\int_{\Omega} |v_n - v|^{p+1} dx \right)^{\frac{1}{p+1}}.
\end{align*}
The right hand side tends to $0$ as $n \to \infty$ since  $v_n \to v$ in $L^{p+1}(\Omega)$ and also in $L^2(\Omega, (\underline{u}_{\lambda})^{p-1} )$, which gives
\begin{equation}\label{eq_mpt6}
\int_{\Omega} h(\underline{u}_{\lambda},v_n)(v_n - v) dx \to 0 \text{ as }n \to \infty.
\end{equation}
Combining \eqref{eq_mpt5}--\eqref{eq_mpt6} leads to \eqref{eq_mpt4}. As a result, since $\Hbb(\Omega)$ is a Hilbert space, $u_n \to u$ in $\Hbb(\Omega)$. Thus, $J$ satisfies (PS) condition.

By the Mountain Pass Theorem, we derive the existence of a nontrivial critical point $v \in \Hbb(\Omega)$ of $J$, which is a nontrivial variational solution of \eqref{eq_mountainpass}. We complete the proof.
\end{proof}

\begin{proof}[\sc Proof of Theorem \ref{th:main2}(2)] Theorem \ref{theo_mountainpas} shows that there exists a nontrivial nonnegative solution $v$ of \eqref{eq_mountainpass}, which means $v = \Gbb^{\Omega}_{s} [(\underline{u}_{\lambda} + v)^p] - \Gbb^{\Omega} [(\underline{u}_{\lambda})^p] > 0$
by Lemma \ref{lem_maximum}. This implies $\underline{u}_{\lambda} + v = \Gbb^{\Omega} [(\underline{u}_{\lambda} + v)^p] + \lambda \Gbb^{\Omega} [\mu]$, therefore $\underline{u}_{\lambda}+v$ is a solution which is strictly greater than $\underline{u}_{\lambda}$. The bounds on $u_{\lambda}$ are obtained in Proposition \ref{prop:boundonu}. We complete the proof.
\end{proof}


\section{Other examples} \label{sec:other-examples} We present below further examples to which our theory can be applied.

\noindent \textbf{Laplacian with a Hardy potential.} We consider an operator created by perturbing the classical Laplacian by a Hardy potential which is strongly singular on the boundary
$$\L =L_\kappa := -\Delta  - \dfrac{\kappa}{\delta^{2}},$$
where $\kappa$ is a parameter.  This operator is a special case of Schr\"odinger operators and has been investigated extensively in the literature (see, for instance, \cite{BanMor_2008, GkiVer_2015, MarNgu_2017, GkiNgu_2019}).  The best constant in Hardy's inequality, i.e. the quantity
\begin{equation}\label{eq:hardypotential}
C_H(\Omega):= \inf_{u \in H^1_0(\Omega) \setminus \{0\} } \dfrac{\displaystyle \int_{\Omega}|\nabla u|^2 dx}{ \displaystyle \int_{\Omega} (u/\delta)^2 dx},
\end{equation}
is deeply involved in the study of $L_\kappa$. It is classical that $C_H(\Omega) \in (0,\frac{1}{4}]$. Our theory is applicable to operator $\L=L_\kappa$ with $\kappa \in (0,C_H(\Omega))$.  Indeed, we have
\begin{equation} \label{Hardy-1} 
\inner{\L u,v}_{L^2(\Omega)} = \int_{\Omega} \nabla u \cdot \nabla v dx - \kappa \int_{\Omega} \dfrac{uv}{\delta^2} dx = \inner{\L v,u}_{L^2(\Omega)}, \quad \forall u,v \in C^{\infty}_c(\Omega).
\end{equation}
Moreover, by \eqref{eq:hardypotential} and the condition $\kappa  \in (0,C_H(\Omega))$, we have, for any $u \in C^{\infty}_c(\Omega)$
\begin{equation} \label{Hardy-2} 
\inner{\L u,u}_{L^2(\Omega)} \geq \left(1- \frac{\kappa}{C_H(\Omega)}\right)\int_{\Omega}|\nabla u|^2 dx \geq (C_H(\Omega) - \kappa) \int_{\Omega} \dfrac{|u|^2}{\delta^2} dx \gtrsim \norm{u}^2_{L^2(\Omega)}.
\end{equation}
Combining \eqref{Hardy-1} and \eqref{Hardy-2} yields \eqref{eq_L1}--\eqref{eq_L3}.

Next,  \eqref{eq_G1} holds by \cite[The paragraph before Theorem 4.11]{FilMosTer_2007} and Lemma \ref{lem:infinitesimal} (see also \cite{GkiVer_2015,MarNgu_2017}) and \eqref{eq_G2} holds with $s=1$ and
$\gamma = \frac{1}{2}+\sqrt{\frac{1}{4}-\kappa}$ by \cite[Theorem 4.11]{FilMosTer_2007}.  Finally, assumption \eqref{eq_G3} reads as $N > 3$. \medskip

\noindent \textbf{Restricted relativistic Schr\"odinger operators.} We consider a class of relativistic Schr\"odinger operators with mass $m>0$ of the form
$$\L  = L_m^s:= (-\Delta  + m^2 I)^s - m^{2s} I,$$
restricted to functions that are zero outside $\Omega$, where $s \in (0,1)$ and $I$ is the identity operator.  By \cite[(1.3) and (6.7)]{FalFel_2015}, $L_m^s$ can be alternatively expressed by
$$\L u(x) = c_{N,s} m^{\frac{N+2s}{2}} P.V \int_{\R^N} \dfrac{u(x)-u(y)}{|x-y|^{\frac{N+2s}{2}}}K_{\frac{N+2s}{2}}(m|x-y|)dy,\quad x \in \Omega,$$
where, for $\nu>0$, 
$$K_\nu(r) \sim \frac{\Gamma(\nu)}{2}\left(\frac{r}{2}\right)^{-\nu}, \quad r>0,$$
with $\Gamma$ being the usual Gamma function.  The study of this operator and its variants has recently attracted a lot of attention from numerous mathematicians because of its significant interest to many research areas in mathematics and physics (see \cite{Ryz_2003,FalFel_2015,Amb_2016}). Furthermore, by \cite[(1.3) and (6.7)]{FalFel_2015}, $L_m^s$ can be alternatively expressed by
$$\L u(x) = c_{N,s} m^{\frac{N+2s}{2}} P.V \int_{\R^N} \dfrac{u(x)-u(y)}{|x-y|^{\frac{N+2s}{2}}}K_{\frac{N+2s}{2}}(m|x-y|)dy,\quad x \in \Omega,$$
where, for $\nu>0$, 
$$K_\nu(r) \sim \frac{\Gamma(\nu)}{2}\left(\frac{r}{2}\right)^{-\nu}, \quad r>0,$$
with $\Gamma$ being the usual Gamma function. 

Let $u \in C_c^\infty(\Omega)$, we write
\begin{align*}
\L u(x) &= c_{N,s} m^{\frac{N+2s}{2}} P.V. \int_{\R^N\backslash \{ |x-y| \le 1\}} \dfrac{u(x)-u(y)}{|x-y|^{\frac{N+2s}{2}}}K_{\frac{N+2s}{2}}(m|x-y|)dy \\ &+c_{N,s} m^{\frac{N+2s}{2}} P.V. \int_{\{ |x-y| \le 1\}} \dfrac{u(x)-u(y)- \nabla u(x)\cdot (x-y)}{|x-y|^{\frac{N+2s}{2}}}K_{\frac{N+2s}{2}}(m|x-y|)dy\\ &=: I_1(x) + I_2(x).
\end{align*}
It can be seen that
\begin{align*}
|I_1(x)| \lesssim \norm{u}_{L^{\infty}(\Omega)} \int_{\R^N \backslash \{ |x-y| \le 1 \}} \dfrac{1}{|x-y|^{N+2s}} dy \lesssim \norm{u}_{L^{\infty}(\Omega)},
\end{align*}
while
\begin{align*}
|I_2(x)| \lesssim \int_{\{ |x-y| \le 1\} } \dfrac{ \norm{D^2 u}_{L^{\infty}(\Omega)}}{|x-y|^{N-2-2s}} dy \lesssim \norm{D^2 u}_{L^{\infty}(\Omega)}.
\end{align*}
We conclude that $\L u \in L^{\infty}(\Omega) \subset L^2(\Omega)$. Hence, by proceeding as in the case of the RFL in Subsection \ref{sec:someexamples}, we conclude that \eqref{eq_L1}--\eqref{eq_L3} hold.

Next, \eqref{eq_G1} follows from \cite[Theorem 1]{Ryz_2003} and Lemma \ref{lem:infinitesimal}. It can be seen that \eqref{eq_G2} is satisfied with $\gamma =s$ (see \cite[Theorem 4]{Ryz_2003,FalFel_2015}). Finally, \eqref{eq_G3} becomes $N \geq 4s$. \medskip

\noindent \textbf{Sum of two RFLs.} Let $0 < \beta < \alpha < 1$, we consider the sum of $\alpha$-RFL and $\beta$-RFL given by
$$\L u(x) = (-\Delta)^{\alpha}_{\RFL}u(x) + (-\Delta)^{\beta}_{\RFL} u(x) := P.V \int_{\R^N} J_{\alpha,\beta}(x,y) (u(x) - u(y)) dy,$$
where
$$J_{\alpha,\beta}(x,y) := \frac{1}{|x-y|^N} \left(\dfrac{c_{N,\alpha}}{|x-y|^{2\alpha}} + \dfrac{c_{N,\beta}}{|x-y|^{2\beta}}\right).$$
This type of operators has been introduced in \cite{CheKimSon_2010}. It can be seen that $\Hbb(\Omega) = H^{\alpha}_{00}(\Omega)$, and therefore, by proceeding as in the case of RFL in Subsection \ref{sec:someexamples} with minor modification, we can show that \eqref{eq_L1}--\eqref{eq_L3} are also satisfied.  

Next, \eqref{eq_G1} is satisfied due to \cite{CheKimSon_2010} and Lemma \ref{lem:infinitesimal}. By \cite[Corollary 1.2]{CheKimSon_2010}, \eqref{eq_G2} holds with $s=\gamma = \alpha$. Finally, \eqref{eq_G3} becomes $N \geq 4\alpha$. \medskip

\noindent \textbf{An interpolation of the RFL and the SFL.} For $\sigma_1, \sigma_2 \in (0,1]$, we consider the $\sigma_2$-spectral decomposition of the $\sigma_1$-RFL in a domain $\Omega$, namely $\left[(-\Delta)^{\sigma_1}_{\RFL}\right]^{\sigma_2}_{\SFL}$. This type of interpolation operator has been recently studied  in \cite{KimSonVon_2020-1} by using a probabilistic approach. 

Let $Z$ be a rotationally invariant $\sigma_1-$stable process in $\R^N$ (for instance, one can consider a Brownian motion $X$ in $\R^N$ subordinated by an independent $\sigma_1-$stable subordinator), and $Z^{\Omega}$ be the subprocess of $Z$ that is killed upon existing $\Omega$. Then we subordinate $Z^{\Omega}$ by an independent $\sigma_2-$stable subordinator $T$ to obtain a process $Y^{\Omega}$, i.e. $(Y^{\Omega})_t  = (Z^{\Omega})_{T_t}$. Denote by $(R^{\Omega}_t)$ the semigroup of $Y^{\Omega}$. The infinitesimal generator ${\mathcal L}^{\sigma_1,\sigma_2}$ of the semigroup $(R^{\Omega}_t)$ can be written as
$$ \L = {\mathcal L}^{\sigma_1,\sigma_2} := \left[(-\Delta)^{\sigma_1}_{\RFL}\right]^{\sigma_2}_{\SFL}.$$
In particular, one has from \cite[(2.20) and (3.4)]{KimSonVon_2020-1} that
\begin{equation}\label{eq_interpolation1}
{\mathcal L}^{\sigma_1,\sigma_2} u(x) = \int_{\Omega} J_{\sigma_1,\sigma_2} (x,y) (u(x) - u(y)) dy + \kappa_{\sigma_1,\sigma_2}(x)u(x),\quad u \in C^{\infty}_c(\Omega),
\end{equation}
where
\begin{equation}\label{eq_JB}
\begin{aligned}
J_{\sigma_1,\sigma_2}(x,y) &:= \int_0^{\infty} p_{\sigma_1}(t,x,y)\nu(dt), \\
\kappa_{\sigma_1,\sigma_2}(x) &:= \int_0^{\infty} \left( 1- \int_{\Omega} p_{\sigma_1}(t,x,y)dy\right)\nu(dt).
\end{aligned}
\end{equation}
In \eqref{eq_JB}, $$\nu(dt) = \dfrac{\sigma_2}{\Gamma(1-\sigma_2)} t^{-\sigma_2-1} dt$$ is the L\'evy measure of the $\sigma_2$-stable subordinator $T$ and $p_{\sigma_1}(t,x,y)$ denotes the heat kernel of the  $\sigma_1$-RFL in the domain $\Omega$. By  \cite[Proposition 3.4]{KimSonVon_2020}, one has ${\mathcal L}^{\sigma_1,\sigma_2}: C^{\infty}_c(\Omega) \subset L^2(\Omega) \to L^2(\Omega)$ and the quadratic form of ${\mathcal L}^{\sigma_1,\sigma_2}$ is
\begin{equation}\label{eq_interpolation1B}
\begin{aligned}
\inner{ {\mathcal L}^{\sigma_1,\sigma_2} u, v}_{L^2(\Omega)} &= \inner{ u, {\mathcal L}^{\sigma_1,\sigma_2}  v}_{L^2(\Omega)} \\ &= \dfrac{1}{2} \int_{\Omega} \int_{\Omega} J_{\sigma_1,\sigma_2} (x,y) (u(x) - u(y))(v(x)-v(y))dx dy \\ &\quad \quad + \int_{\Omega} \kappa_{\sigma_1,\sigma_2}(x)u(x)v(x) dx
\end{aligned}
\end{equation}
for every $u,v \in C^{\infty}_c(\Omega)$. Therefore ${\mathcal L}^{\sigma_1,\sigma_2}$ is positive and symmetric on $C_c^\infty(\Omega)$. As explained in Section \ref{sec:preliminaries}, it can be extended to a positive self-adjoint operator, still denoted by ${\mathcal L}^{\sigma_1,\sigma_2}$, whose quadratic form domain is denoted by $\Hbb(\Omega)$.  Note that $\Hbb(\Omega)$ is a Hilbert space  and  $C^{\infty}_c(\Omega)$ is dense in $\Hbb(\Omega)$. The reader is referred to \cite{KimSon_2011,KimSonVon_2020-1} for more details.

Alternatively, as in the case of the spectral fractional Laplacian $(-\Delta)^s_{\SFL}$, one can define the interpolation operator $\left[(-\Delta)^{\sigma_1}_{\RFL}\right]^{\sigma_2}_{\SFL}$ in term of the spectral decomposition of $(-\Delta)^{\sigma_1}_{\RFL}$. 
We follow the strategy in \cite{AbaDup_2017,CapDav_2016}. Consider the eigenproblem
$$\va{(-\Delta)^{\sigma_1}_{\RFL} u  &= \lambda u &\text{ in }\Omega,\\ u &= 0 &\text{ in }\mathbb{R}\backslash \Omega.}$$
Let $G^{\Omega}_{\sigma_1}$ be the Green kernel of $(-\Delta)^{\sigma_1}_{\RFL}$ in $\Omega$. Since $\Gbb^{\Omega}_{\sigma_1}:L^2(\Omega) \to L^2(\Omega)$ is compact, $(-\Delta)^{\sigma_1}_{\RFL}$ admits a discrete spectrum $\{ \lambda_{\sigma_1,n}\}$ such that $0<\lambda_{\sigma_1,1} < \lambda_{\sigma_1,2} < \cdots \lambda_{\sigma_1,n} \nearrow +\infty$. Let $\{\varphi_{\sigma_1,n}\}$ be the corresponding eigenfunctions. Therefore, there exists an orthonormal basis of $L^2(\Omega)$ consisting of eigenfunctions $\{\varphi_{\sigma_1,n}\}$ (see \cite{BonFigVaz_2018,SerVal_2014}). 

For a function $u \in C^{\infty}_c(\Omega)$, one has $u = \sum_{n=1}^{\infty} \widehat{u_{\sigma_1,n}}\varphi_{\sigma_1,n}$, where $\widehat{u_{\sigma_1,n}} := \int_{\Omega} u\varphi_{\sigma_1,n} dx$. Thus, one can define the $\sigma_2-$spectral decomposition of $(-\Delta)^{\sigma_1}_{\RFL}$ by
\begin{equation}\label{eq_interpolation2}
Lu(x):= \sum_{n=1}^{\infty} (\lambda_{\sigma_1,n})^{\sigma_2} \widehat{u_{\sigma_1,n}}\varphi_{\sigma_1,n}(x), \quad u \in C_c^\infty(\Omega).
\end{equation}
By density, this operator can be extended to the Hilbert space
\begin{equation} \label{spaceH} H: = \left\{ u \in L^2(\Omega): \norm{u}_{H}^2 := \sum_{n=1}^{\infty} (\lambda_{\sigma_1,n})^{\sigma_2} |\widehat{u_{\sigma_1,n}}|^2  < \infty \right\}.
\end{equation}
Thus, if $u \in H$, then $L u \in H^{-1}$, the topological dual of $H$, and
\begin{equation}\label{eq_interpolation2B}
\inner{L u, v}_{H^{-1},H} = \inner{u,v}_{H} = \sum_{n=1}^{\infty} (\lambda_{\sigma_1,n})^{\sigma_2}\widehat{u_{\sigma_1,n}}\widehat{v_{\sigma_1,n}},\quad u,v \in H.
\end{equation}

We will show that \eqref{eq_interpolation1B} and \eqref{eq_interpolation2B} are in fact equivalent up to a constant. Recall that by \cite[Theorem 3.8]{CheSon_1997}, the heat kernel for $(-\Delta)^{\sigma_1}_{\RFL}$ in $\Omega$ can be written as 
\begin{equation}
p_{\sigma_1}(t,x,y) = \sum_{n=1}^{\infty} e^{-t \lambda_{\sigma_1,n}}\varphi_{\sigma_1,
n}(x)\varphi_{\sigma_1,n}(y), \quad t > 0, x,y \in \Omega.
\end{equation}
For $u,v \in C^{\infty}_c(\Omega)$, one has
\begin{align*}
\inner{ L u, v}_{H^{-1},H} &= \sum_{n=1}^{\infty} (\lambda_{\sigma_1,n})^{\sigma_2} \widehat{u_{\sigma_1,n}}\widehat{v_{\sigma_1,n}} \\
& = \dfrac{\sigma_2}{\Gamma(1-\sigma_2)}\sum_{n=1}^{\infty} \int_0^{\infty} \left(  \widehat{u_{\sigma_1,n}}\widehat{v_{\sigma_1,n}} - e^{-t \lambda_{\sigma_1,n}} \widehat{u_{\sigma_1,n}}\widehat{v_{\sigma_1,n}}\right)\dfrac{dt}{t^{1+\sigma_2}}\\
&=\dfrac{\sigma_2}{\Gamma(1-\sigma_2)}\int_0^{\infty} \left(\sum_{n=1}^{\infty} \widehat{u_{\sigma_1,n}}\widehat{v_{\sigma_1,n}} - \sum_{n=1}^{\infty} e^{-t \lambda_{\sigma_1,n}} \widehat{u_{\sigma_1,n}}\widehat{v_{\sigma_1,n}} \right)\dfrac{dt}{t^{1+\sigma_2}}\\
&=\dfrac{\sigma_2}{\Gamma(1-\sigma_2)}\int_0^{\infty} \left[\int_{\Omega} u(y)v(y)dy - \int_{\Omega} \int_{\Omega} p_{\sigma_1}(t,x,y)u(x)v(y)dx dy \right] \dfrac{dt}{t^{1+\sigma_2}}.
\end{align*}
Here we have used the formula
$$\lambda^{\sigma} = \dfrac{\sigma}{\Gamma(1-\sigma)} \int_0^{\infty} (1-e^{-t\lambda })\dfrac{dt}{t^{1+\sigma}}$$ in the second equality and Fubini's theorem in the third equality. Thus,
\begin{equation}\label{eq_interpoeq1}
\begin{aligned}
\inner{ L u, v}_{H^{-1},H}=\,\,&\dfrac{\sigma_2}{\Gamma(1-\sigma_2)} \int_0^{\infty} \int_{\Omega} \left[ \int_{\Omega}p_{\sigma_1}(t,x,y)\left(u(y) - u(x)\right)v(y)dx\right]dy \dfrac{dt}{t^{1+\sigma_2}}\\
\quad + &\dfrac{\sigma_2}{\Gamma(1-\sigma_2)}\int_0^{\infty}\int_{\Omega} u(y)v(y)\left(1- \int_{\Omega}p_{\sigma_1}(t,x,y)dx \right) dy \dfrac{dt}{t^{1+\sigma_2}} \\
=:\,\,&I_1(u,v) + I_2(u,v).
\end{aligned}
\end{equation}
We will treat the terms $I_1(u,v)$ and $I_2(u,v)$ successively. As for $I_1$, by taking $u=v$ and by interchanging $x$ and $y$ in the formula, we obtain
$$I_1(u,u) = \dfrac{\sigma_2}{2\Gamma(1-\sigma_2)} \int_{\Omega}\int_{\Omega} \left[\int_0^{\infty} p_{\sigma_1}(t,x,y) \dfrac{dt}{t^{1+\sigma_2}} \right]\left(u(x) - u(y)\right)^2dx dy,$$
hence $0 \leq I_1(u,u) < +\infty$. Next, by interchanging $x$ and $y$ in the formula, we get
\begin{equation}\label{eq_interpoeq3}
\begin{aligned}
I_1(u,v) &= \dfrac{\sigma_2}{2\Gamma(1-\sigma_2)} \int_0^{\infty} \int_{\Omega} \left[ \int_{\Omega}p_{\sigma_1}(t,x,y)\left(u(x) - u(y)\right)\left(v(x)-v(y)\right)dx\right]dy \dfrac{dt}{t^{1+\sigma_2}}\\ 
&= \dfrac{\sigma_2}{2\Gamma(1-\sigma_2)} \int_{\Omega}\int_{\Omega} \left[\int_0^{\infty} p_{\sigma_1}(t,x,y) \dfrac{dt}{t^{1+\sigma_2}} \right]\left(u(x) - u(y)\right)\left(v(x)-v(y)\right)dx dy \\
&= \frac{1}{2}\int_{\Omega}\int_{\Omega} J_{\sigma_1,\sigma_2} (x,y)\left(u(x) - u(y)\right)\left(v(x)-v(y)\right) dx dy.
\end{aligned}
\end{equation}
Next, we see that $I_2(u,u) \geq 0$ by the fact that $0\le \int_{\Omega} p_{\sigma_1}(t,x,y) dx \le 1$. Furthermore, we have
$$I_2(u,u) = \dfrac{\sigma_2}{\Gamma(1-\sigma_2)}\int_0^{\infty}\int_{\Omega} |u(y)|^2\left(1- \int_{\Omega}p_{\sigma_1}(t,x,y)dx \right) dy \dfrac{dt}{t^{1+\sigma_2}} \le \inner{u,u}_{\Hbb(\Omega)} < +\infty.$$
Thus using Fubini's theorem and H\"older's inequality, we have
\begin{equation}\label{eq_interpoeq2}
\begin{aligned}
I_2(u,v)  &= \dfrac{\sigma_2}{\Gamma(1-\sigma_2)}\int_{\Omega} \int_0^{\infty}\left(1- \int_{\Omega}p_{\sigma_1}(t,x,y)dx \right) \dfrac{dt}{t^{1+\sigma_2}}  u(y)v(y)dy \\
&= \int_{\Omega} \kappa_{\sigma_1,\sigma_2}(y)u(y)v(y) dy.
\end{aligned}
\end{equation}
Combining \eqref{eq_interpoeq1}, \eqref{eq_interpoeq2} and \eqref{eq_interpoeq3}, we conclude that
$$\inner{Lu,v}_{H^{-1},H} =  \inner{ {\mathcal L}^{\sigma_1,\sigma_2} u,v}_{\Hbb^{-1}(\Omega),\Hbb(\Omega)}, \quad \forall u,v \in C^{\infty}_c(\Omega).$$

We verify that ${\mathcal L}^{\sigma_1,\sigma_2}$ satisfies all the assumptions in Section \ref{sec:preliminaries}. By the above spectral decomposition and the fact that the first eigenvalue of ${\mathcal L}^{\sigma_1,\sigma_2}$ is $\lambda_{\sigma_1,1}^{\sigma_2} > 0$, \eqref{eq_L1} and \eqref{eq_L2} hold.


Next we prove that \eqref{eq_L3} also holds, i.e.  $\Hbb(\Omega) = H^s_{00}(\Omega)$, provided that $\Omega$ has smooth boundary. To this end, we invoke the discrete $J-$method in \cite[Theorem 8.2]{BonSirVaz_2015} and adapt the argument in \cite[pages 5734--5735]{BonSirVaz_2015}. We include it here for the sake of convenience.

By the spectral decomposition, one has $u = \sum_{n=1}^{\infty} u_{\sigma_1,n}$, where $u_{\sigma_1,n} := \widehat{u_{\sigma_1,n}} \varphi_{\sigma_1,n}$. It can easily be seen that $\norm{u_{\sigma_1,n}}_{L^2(\Omega)} = |\widehat{u_{\sigma_1,n}}|$. On the other hand, if $2\sigma_1 \neq \frac{1}{2}$, one has
\begin{align*}
\norm{u_{\sigma_1,n}}^2_{H^{2\sigma_1}_0(\Omega)} &= |\widehat{u_{\sigma_1,n}}|^2 \norm{\varphi_{\sigma_1,n}}_{H^{2\sigma_1}_0(\Omega)} = |\widehat{u_{\sigma_1,n}}|^2 \norm{\varphi_{\sigma_1,n}}_{H^{2\sigma_1}(\R^N)} \\
&= |\widehat{u_{\sigma_1,n}}|^2\norm{(-\Delta)^{\sigma_1} \varphi_{\sigma_1,n}}^2_{L^2(\R^N)}.
\end{align*}
By noticing that
\begin{align*}
\int_{\R^N} |(-\Delta)^{\sigma_1} \varphi_{\sigma_1,n}|^2 dx = \int_{\Omega}|(-\Delta)^{\sigma_1} \varphi_{\sigma_1,n}|^2 dx + \int_{\R^N \backslash \Omega} \varphi_{\sigma_1,n}(-\Delta)^{2\sigma_1} \varphi_{\sigma_1,n} dx = \lambda_{\sigma_1,n}^2|\widehat{u_{\sigma_1,n}}|^2,
\end{align*}
we have
$$\norm{u_{\sigma_1,n}}^2_{H^{2\sigma_1}_0(\Omega)} = \lambda_{\sigma_1,n}^2|\widehat{u_{\sigma_1,n}}|^2.$$

We next check two conditions in the $J-$Method. Put
$$U_n :=\lambda_{\sigma_1,n}^{-\theta} J(\lambda_{\sigma_1,n},u_{\sigma_1,n}) = \lambda_{\sigma_1,n}^{-\theta} \max\{ \norm{u_{\sigma_1,n}}_{H^{2\sigma_1}_0(\Omega)}, \lambda_{\sigma_1,n} \norm{u_{\sigma_1,n}}_{L^2(\Omega)} \} = \lambda_{\sigma_1,n}^{1-\theta} |\widehat{u_{\sigma_1,n}}|.$$
From \eqref{spaceH}, we see that $U_n \in l^2(\mathbb{N})$ if and only if $\theta = 1 - \sigma_2/2$.

It can be easily  seen that the eigenvalues of ${\mathcal L}^{\sigma_1,\sigma_2}$ are $\{\lambda_{\sigma_1,n}^{\sigma_2} \}$, where $\{ \lambda_{\sigma_1,n} \}$ are the eigenvalues of $(-\Delta)^{\sigma_1}_{\RFL}$. Since $0< \lambda_{\sigma_1,n+1}/ \lambda_{\sigma_1,n} < \Gamma_0$ for some $\Gamma_0>0$,  it follows that $0 < \lambda_{\sigma_1,n+1}^{\sigma_2}/\lambda_{\sigma_1,n}^{\sigma_2} < \Gamma_0^{\sigma_2} $ for every $n \in \N$. We conclude from the discrete version of the $J-$Method that
 $$F = \left[H^{2\sigma_1}_0(\Omega),L^2(\Omega) \right]_{1-\frac{\sigma_2}{2}} = H^s_{00}(\Omega),$$
since $\sigma_1\sigma_2 = s$. The case $2\sigma_1 = 1$ can be proceeded similarly. Note that if $\sigma_2 = 1$, one has the RFL $(-\Delta)^s_{\RFL}$, while $\sigma_1 = 1$ gives the SFL $(-\Delta)^s_{\SFL}$. We derive that \eqref{eq_L3} holds. 

Next, \eqref{eq_G1} follows from \cite{KimSonVon_2020-1} and Lemma \ref{lem:infinitesimal}. By \cite[Theorem 6.4]{KimSonVon_2020-1},  assumption \eqref{eq_G2} is satisfied with $s=\sigma_1\sigma_2$ and $\gamma = \sigma_1$.  This operator is interesting since both cases $\gamma<2s$ and $\gamma \geq 2s$ may occur. Indeed, if $\sigma_2>\frac{1}{2}$ then $\gamma<2s$ and hence \eqref{eq_G3} becomes $N \geq 4\sigma_1\sigma_2$. If $\sigma_2 \leq \frac{1}{2}$ then $\gamma \geq 2s$ and hence \eqref{eq_G3} becomes $N \geq 4\sigma_1 \sigma_2 (1+\sigma_1) - \sigma_1$.

We conclude that assumptions, \eqref{eq_L1}--\eqref{eq_L3}, \eqref{eq_G1}--\eqref{eq_G3} hold and thus our theory can be applied to this type of operators.

\begin{appendices}
\section{Appendix}\label{sec:appendixx}

In this appendix, we provide a result which can be used to verify assumption \eqref{eq_G1} in the examples considered in Subsection \ref{sec:someexamples} and Section \ref{sec:other-examples}.

Assume $\L: \dom(\L) \subset L^2(\Omega) \to L^2(\Omega)$ is a positive, self-adjoint operator. It is well known that the spectrum $\sigma(\L)$ of $\L$ satisfies $\sigma(\L) \subset [0,\infty)$. By Hille--Yosida theorem for nonpositive self-adjoint operators, $-\L$ is the infinitesimal generator of the contraction semigroup $\{P_t = e^{-t \L}\}_{t \ge 0}$ in $L^2(\Omega)$. Furthermore, by the spectral theorem (see \cite[Theorem A.4.2]{DomIvan_2016} or \cite[Theorem 12.4]{SchiSongVodracek_2012}), there exists a unique projection-valued measure $E(\cdot)$ with support on the spectrum $\sigma(\L)$ of $\L$ such that
$$ \L = \int_{\sigma(\L)}\lambda dE(\lambda),
$$ 
and the semigroup $\{ P_t \}_{t \ge 0}$ can be written as (see e.g. \cite[page 480]{DomIvan_2016}) 
\begin{equation} \label{pvm}
P_t =  \int_{\sigma(\L)} e^{-\lambda t} dE(\lambda).
\end{equation} 

Assume that the semigroup $\{ P_t \}_{t \ge 0}$ admits a density kernel (also called heat kernel), namely there exists for every $t > 0$ a positive measurable function $p_{\Omega}(t,x,y)$ defined almost everywhere on $\Omega \times \Omega$  such that, for any $f \in L^2(\Omega)$,
$$P_t f(x) = \int_{\Omega} p_{\Omega}(t,x,y)f(y) dy,\quad x \in \Omega.$$
The Green function and Green operator associated to $\L$ are defined respectively by
\begin{align} \label{Gp1}
&G^{\Omega}(x,y)= \int_0^{\infty} p_{\Omega}(t,x,y) dt, \quad x,y \in \Omega,\\ \label{Gp2}
&\Gbb^{\Omega}[f](x) = \int_{\Omega} G^{\Omega}(x,y) f(y) dy,\quad  x \in \Omega, \quad  f\in L^2(\Omega).
\end{align}
The following result asserts that, under some additional conditions, $\Gbb^{\Omega}$ is in fact the right inverse operator of $\L$.

\begin{lemma}\label{lem:infinitesimal} Let $\L$ be a positive, self-adjoint operator. 
	Assume that $\inf \sigma(\L)>0$ and the semigroup $\{P_t\}_{t \ge 0}$ generated by $\L$ admits a heat kernel $p_{\Omega}(t,x,y)$. Then $\L^{-1}$ exists and $\L^{-1}f = \Gbb^{\Omega}[f]$ for every $f \in L^2(\Omega)$.
\end{lemma}

\begin{proof}
From \eqref{Gp1} and \eqref{Gp2}, we have, for any $f \in L^2(\Omega)$, 
\begin{equation}\label{eq:lemmaPG}
\begin{aligned}
\Gbb^{\Omega}[f](x) &= \int_{\Omega} \left( \int_0^{\infty} p_{\Omega}(t,x,y) dt \right) f(y) dy \\
&= \int_0^{\infty}  \left( \int_{\Omega} p_{\Omega}(t,x,y) f(y) dy \right) dt =\int_0^{\infty} P_t f(x) dt,\quad x \in \Omega.
\end{aligned}
\end{equation}	
This, together with \eqref{pvm} and assumption $\inf \sigma(\L)>0$, implies
	\begin{align*}
	\Gbb^{\Omega}[f] &=\int_0^{\infty} \left( \int_{\sigma(\L)} e^{-t \lambda} dE(\lambda)f \right) dt\\
	&= \int_{\sigma(\L)} \left(\int_0^{\infty} e^{-t \lambda} dt \right) dE(\lambda) f =\int_{\sigma(\L)} \lambda^{-1} d E(\lambda) f.
	\end{align*}
	Recall that for a Borel measurable function $f$ on $\sigma(\L)$, one may define the operator $f(\L)$ as
	$$f(\L) = \int_{\sigma(\L)} f(\lambda) dE(\lambda).$$
	In particular, since $\inf \sigma(\L)>0$, the inverse of $\L$ exists and
	$$\L^{-1} f  = \int_{\sigma(\L)} \lambda^{-1} d E(\lambda) f = \Gbb^{\Omega}[f],$$
	which is the desired result.
\end{proof}

\begin{remark} The assumption on the existence of the heat kernel $p_{\Omega}(t,x,y)$  in the statement of Lemma \ref{lem:infinitesimal} does not hold in general and is closely related to the ultracontractivity property of the semigroup $\{P_t\}_{t \geq 0}$. More precisely, if $\{P_t\}_{t \geq 0}$ is ultracontractive, namely for each $t > 0$, $P_t$ is a bounded operator from $L^2(\Omega)$ to $L^{\infty}(\Omega)$, then $\{P_t\}_{t \geq 0}$ has a heat kernel $p_\Omega(t,x,y)$ such that, for any $t>0$, $0 \le  p_{\Omega}(t,x,y) \le c_t$ a.e. on $\Omega \times \Omega$ (see, for instance, \cite[page 247]{SchiSongVodracek_2012}). It is also worth mentioning that, under certain conditions on the Green kernel and the semigroup, the existence and estimates of a heat kernel has been recently established by using the eigen-decomposition of $\L$ (see \cite[Sections 3 and 6]{chan2020singular} for more details).
\end{remark}

\begin{remark} We note that Lemma \ref{lem:infinitesimal} can be used to verify assumption \eqref{eq_G1} for all operators considered in Subsection \ref{sec:someexamples} and Section \ref{sec:other-examples} thanks to their probabilistic characteristic. More precisely,  properties of the Green function of $(-\Delta)^s_{\text{RFL}}$, $(-\Delta)^s_{\text{SFL}}$ and $(-\Delta)^s_{\text{CFL}}$ are given in \cite{CheSon_1998}, \cite{KimSon_2011} and \cite{CheKimSon_2009} respectively.  For other operators in Section \ref{sec:other-examples} such as the Laplacian with a Hardy potential, the restricted relativistic Schr\"odinger operator, the sum of two RFLs and the interpolation operator, the characteristic of their Green function is given in \cite{FilMosTer_2007}, \cite{Ryz_2003}, \cite{CheKimSon_2010} and \cite{KimSonVon_2020-1} respectively.
\end{remark}
\end{appendices}

\bibliographystyle{siam}

\begin{thebibliography}{99}
	
\bibitem{Aba-PhD} N.~Abatangelo, {\em  Large solutions for fractional Laplacian operators}, Ph.D. dissertation (2015).

\bibitem{Aba_2015} N.~Abatangelo, {\em Large $s$-harmonic functions and boundary blow-up solutions for the fractional Laplacian}, Discrete Contin. Dyn. Syst. {\bf 35}(2015), 5555--5607.

\bibitem{AbaDup_2017} N.~Abatangelo and L.~Dupaigne, {\em Nonhomogeneous boundary conditions for the spectral fractional Laplacian}, Ann. Inst. H. Poincar\'e Anal. Non Lin\'eaire {\bf 34} (2017), 439--467.

\bibitem{Aba_2019} N.~Abatangelo, D.~G\'omez-Castro and J.~L. V\'azquez, {\em Singular boundary behaviour and large solutions for fractional elliptic equations} (2019), arXiv:1910.00366.

\bibitem{Amb_2016} V. Ambrosio, {\em Ground states solutions for a non-linear equation involving a pseudo-relativistic Schr\"odinger operator}, J. Math. Phys. {\bf 57} (2016), no. 5, 051502, 18 pp.

\bibitem{AmbRab_1973} A.~Ambrosetti and P.~H. Rabinowitz, {\em Dual variational methods in critical point theory and applications}, J. Funct. Anal.  \textbf{14} (1973), 349--381.

\bibitem{BanMor_2008} C. Bandle, V. Moroz and W. Reichel, {\em Boundary blowup type sub-solutions to semilinear elliptic equations with Hardy potential}, J. Lond. Math. Soc. \textbf{2} (2008), 503--523.

\bibitem{Bha_2012} P.~K. Bhattacharyya, {\em Distributions: generalized functions with applications in Sobolev spaces},  De Gruyter 2012.

\bibitem{BidViv_2000} M.-F.~ Bidaut-V\'eron and L.~Vivier,  {\em An elliptic semilinear equation with source term involving boundary measures: the subcritical case}, 
Rev. Mat. Iberoamericana \textbf{16} (2000), 477--513.

\bibitem{BidYar_2002} M.-F.~ Bidaut-V\'eron and C.~Yarur,  {\em Semilinear elliptic equations and systems with measure data: existence and a priori estimates}, Adv. Differential Equations \textbf{7} (2002), 257--296.

\bibitem{BogBurChe_2003} K.~Bogdan, K.~Burdzy and Z.-Q. Chen, {\em Censored stable processes}, Probab. Theory Related Fields \textbf{127} (2003), 89--152.

\bibitem{BonFigVaz_2018} M.~Bonforte, A.~Figalli and J.~L. V\'azquez, {\em Sharp boundary behaviour of solutions to semilinear nonlocal elliptic equations}, Calc. Var. Partial Differential Equations \textbf{57} (2018), 34 pp.

\bibitem{BonSirVaz_2015}
M.~Bonforte, Y.~Sire and J.~L.~ V\'azquez, {\em Existence, uniqueness and asymptotic behaviour for fractional porous medium equations on bounded domains}, Discrete Contin. Dyn. Syst. \textbf{35} (2015), 5725--5767.

\bibitem{BonVaz_2014}
M.~Bonforte and J.~L. V\'azquez, {\em Quantitative local and global a priori estimates for fractional nonlinear diffusion equations}, Adv. Math. \textbf{250} (2014), 242--284.

\bibitem{BonVaz_2015} M.~Bonforte and J.~L. V\'azquez, {\em A priori estimates for fractional nonlinear degenerate diffusion equations on bounded domains}, 
Arch. Ration. Mech. Anal. \textbf{218} (2015), 317--362.

\bibitem{BraColPabSan_2013}
C.~Br\"andle, E.~Colorado, A.~D. Pablo, and U.~S\'anchez, {\em A concave-convex elliptic problem involving the fractional Laplacian}, Proc. Roy. Soc. Edinburgh Sect. A \textbf{143} (2013), 39--71.

\bibitem{CafSti_2016}
L.~A.~ Caffarelli and P.~R.~Stinga, {\em Fractional elliptic equations, Caccioppoli estimates and regularity}, Ann. Inst. H. Poincar\'e Anal. Non Lin\'eaire \textbf{33} (2016), 767--807.

\bibitem{CapDav_2016}
A.~Capella, J.~D\'avila, L. Dupaigne and Y.~Sire, {\em Regularity of radial extremal solutions for some non-local semilinear equations}, Comm. Partial Differential Equations \textbf{36} (2011), 1353--1384.

\bibitem{ChaGomVaz_2019_2020}
H.~Chan, D.~G\'omez-Castro, and J.~L.~V\'azquez, {\em Blow-up phenomena in nonlocal eigenvalue problems: when theories of $L^1$ and $L^2$ meet}, J. Funct. Anal. (2020), https://doi.org/10.1016/j.jfa.2020.108845.

\bibitem{chan2020singular}
H.~Chan, D.~G\'omez-Castro and J.~L.~V\'azquez, {\em Singular solutions for fractional parabolic boundary value problems} (2020), arXiv:2007.13391.

\bibitem{Che_2018}
H.~Chen, {\em The Dirichlet elliptic problem involving regional fractional Laplacian}, J. Math. Phys. \textbf{59} (2018), 071504, 19 pp.

\bibitem{CheVerFel_2014}
H.~Chen, P.~Felmer and L.~V\'eron,  {\em Elliptic equations involving general subcritical source nonlinearity and measures} (2014), arXiv:1409.3067.

\bibitem{CheSon_1997}
Z.-Q.~Chen and R.~Song, {\em Intrinsic ultracontractivity and conditional gauge for symmetric stable processes}, 
J. Funct. Anal. \textbf{150} (1997), 204--239.

\bibitem{CheSon_2003}
Z.-Q.~Chen and R.~Song, {\em Hardy inequality for censored stable processes}, 
Tohoku Math. J. (2) \textbf{55} (2003), 439--450.

\bibitem{CheKimSon_2009}
Z.-Q.~Chen, P.~Kim and R.~Song, {\em Two-sided heat kernel estimates for censored stable-like processes}, 
Probab. Theory Related Fields. \textbf{44} (2010), 361--399.

\bibitem{CheKimSon_2010}
Z.-Q.~Chen, P.~Kim and R.~Song, {\em Dirichlet heat kernel estimates for $(-\Delta)^{\alpha/2} + (-\Delta)^{\beta/2}$},
Illinois J. Math. \textbf{54} (2012), 1357--1392.

\bibitem{CheQua_2018} H.~Chen and A.~Quaas, {\em Classification of isolated singularities of nonnegative solutions to fractional semi-linear elliptic equations and the existence results}, J. Lond. Math. Soc. \textbf{97} (2018), 196--221.

\bibitem{CheVer_2014}
H.~Chen and L.~V\'eron, {\em Semilinear fractional elliptic equations involving measures}, J. Differential Equations \textbf{257} (2014), 1457--1486.

\bibitem{CheSon_1998}
Z.-Q. Chen and R.~Song, {\em Estimates on Green functions and Poisson kernels for symmetric stable processes}, Math. Ann. \textbf{312} (1998), 465--501.

\bibitem{CheSon_2002}
Z.-Q. Chen and R.~Song, {\em General gauge and conditional gauge theorems}, Ann. Probab., \textbf{30} (2002), 1313--1339.

\bibitem{Dav_1980}
E.B. Davies, {\em One-parameter semigroups}, Academic Press, London, 1980.

\bibitem{DhiMaaZri_2011}
A.~Dhifli, H.~M\^aagli, and M.~Zribi,  {\em On the subordinate killed b.m in bounded domains and existence results for nonlinear fractional Dirichlet problems},
Math. Ann. \textbf{352} (2011), 259--291.

\bibitem{DipMedPerVal_2017} S.~Dipierro, M.~Medina, I.~Peral and E.~Valdinoci, {\em Bifurcation results for a fractional elliptic equation with critical exponent in $\R^n$}, Manuscripta Math. {\bf 153} (2017), 183--230.

\bibitem{DomIvan_2016} 
B. Dominique, G. Ivan and L. Michel,
{\em Analysis and geometry of Markov diffusion operators},
Springer,  2016.

\bibitem{Dup_2011}
L.~Dupaigne, {\em Stable solutions of elliptic partial differential equations},  Chapman \& Hall/CRC, 2011.


\bibitem{Fal_2020}
M.~M. Fall, {\em Regional fractional Laplacians: Boundary regularity} (2020),  arXiv:2007.04808.

\bibitem{FalFel_2015}
M.~M. Fall and V. Felli, {\em Unique continuation properties for relativistic Schr\"odinger operators with a singular potential}, Discrete Contin. Dyn. Syst. \textbf{35}(2015), 5827--5867.

\bibitem{FerSac_2006} A.~Ferrero and C.~Saccon, {\em Existence and multiplicity results for semilinear equations with measure data}, Topol. Methods Nonlinear Anal. {\bf 28} (2006), 285--318.

\bibitem{FilMosTer_2007}
S. Filippas, L. Moschini, A. Tertikas, {\em Sharp two-sided heat kernel estimates for critical Schrodinger operators on bounded domains}, Commun. Math. Phys. \textbf{273} (2007), 237--281.

\bibitem{Garofalo_2019}
N. Garofalo, {\em Fractional Thoughts,} New Developments in the Analysis of Nonlocal Operators, Con- temp. Math., \textbf{723}, Amer. Math. Soc., Providence, RI, 2019, 1--135.

\bibitem{GatHes_2014}
P.~Gatto and J.~S. Hesthaven, {\em Numerical approximation of the fractional Laplacian via $H^p$, with an application to image denoising}, J. Sci. Comput. \textbf{65} (2014), 249--270.

\bibitem{GidSpr_1981} B.~Gidas and J.~Spruck, {\em  Global and local behavior of positive solutions of nonlinear elliptic equations}, Comm. Pure Appl. Math. {\bf 34} (1981), 525--598.

\bibitem{GkiVer_2015} K.~T.~Gkikas and L.~V\'eron, {\em Boundary singularities of solutions of semilinear elliptic equations with critical Hardy potentials}, Nonlinear Anal. {\bf 121} (2015), 469--540.

\bibitem{GkiNgu_2019} K.~T.~Gkikas and P.~-T.~Nguyen, {\em On the existence of weak solutions of semilinear elliptic equations and systems with Hardy potentials}, Journal of Differential Equations  {\bf 266}  (2019), 833--875.

\bibitem{GomVaz_2019} D. G\'omez-Castro and  J.~L.~V\'azquez, {\em The fractional Schr\"odinger equation with singular potential and measure data}, 
Discrete Contin. Dyn. Syst. {\bf 39} (2019), 7113--7139.

\bibitem{Gra_2009}
L.~Grafakos, {\em Classical fourier analysis}, Graduate Texts in Mathematics, 2009.


\bibitem{Gris_2011}
P.~Grisvard, {\em Elliptic problems in nonsmooth domains},
Society for Industrial and Applied Mathematics, 2011.

\bibitem{KimSon_2011}
P. Kim, R. Song and Z. Vondra\v cek, {\em Potential theory of subordinate Brownian motions revisited} (2011), arXiv:1102.1369.

\bibitem{KimSonVon_2019} P. Kim, R. Song and Z. Vondra\v cek, {\em Potential theory of subordinate killed Brownian motion}, Trans. Amer. Math. Soc. {\bf 371} (2019), 3917--3969.

\bibitem{KimSonVon_2020-1}
P. Kim, R. Song and Z. Vondra\v cek, {\em On the boundary theory of subordinate killed L\'evy processes}, Potential Anal \textbf{53} (2020), 131--181.

\bibitem{KimSonVon_2020}
P. Kim, R. Song, and Z. Vondra\v cek, {\em On potential theory of Markov processes with jump kernels decaying at the boundary} (2020), arXiv:1910.10961.


\bibitem{LioMag_1972}
J.-L. Lions and E.~Magenes, {\em Non-homogeneous boundary value problems and applications}, Springer-Verlag, 1972.

\bibitem{Lio_1980}
P.~Lions, {\em Isolated singularities in semilinear problems},
J. Differential Equations \textbf{38} (1980), 441--450.

\bibitem{MarVer_2014}
M.~Marcus and L.  V\'eron, {\em Nonlinear second order elliptic equations involving measures},  Walter de Gruyter GmbH \& Co. KG, 2014.

\bibitem{MarNgu_2017}
M.~Marcus and P.-T. Nguyen, {\em Moderate solutions of semilinear elliptic equations with Hardy potential}, Ann. Inst. H. Poincar\'e Anal. Non Lin\'eaire \textbf{34} (2017):69--88.

\bibitem{MonPon_2008}
M.~Montenegro and A.~C. Ponce, {\em The sub-supersolution method for weak solutions}, Proc. Amer. Math. Soc. \textbf{136} (2008), 2429--2438.

\bibitem{NaiSat_2007}
Y.~Naito and T.~Sato, {\em Positive solutions for semilinear elliptic equations with singular forcing terms}, J. Differential Equations \textbf{235} (2007), 439--483.

\bibitem{NezPalVal_2012}
E.~D. Nezza, G.~Palatucci and E.~Valdinoci, {\em Hitchhiker's guide to the fractional Sobolev spaces}, Bull. Sci. Math. \textbf{136} (2012), 521--573.

\bibitem{RosSer_2014}
X.~Ros-Oton and J.~Serra, {\em The Dirichlet problem for the fractional Laplacian: Regularity up to the boundary}, J. Math. Pures Appl. \textbf{101} (2014), 275--302.

\bibitem{Ryz_2003}
M. Ryznar, {\em Estimates of Green function for relativistic $\alpha$-stable process}, Potential Anal. \textbf{17} (2003), 1--23.

\bibitem{SchiSongVodracek_2012}
R. L. Schilling, R. Song and Z. Vondra\v{c}ek, {\em Bernstein functions theory and applications}, De Gruyter, 2012.



\bibitem{Ser_2014} R. Servadei, {\em A critical fractional Laplace equation in the resonant case}, Topol. Methods Nonlinear Anal. {\bf 43} (2014), 251--267.

\bibitem{SerVal_2012}
R.~Servadei and E.~Valdinoci, {\em Mountain pass solutions for non-local elliptic operators}, J. Math. Anal. Appl. \textbf{389} (2012), 887--898.

\bibitem{SerVal_2013}
R.~Servadei and E.~Valdinoci, {\em Variational methods for non-local operators of elliptic type}, Discrete Contin. Dyn. Syst. \textbf{33} (2013), 2105--2137.

\bibitem{SerVal_2014}
R.~Servadei and E.~Valdinoci, {\em On the spectrum of two different fractional operators}, Proc. Roy. Soc. Edinburgh Sect. A \textbf{144} (2014), 831--855.

\bibitem{SonVon_2003}
R.~Song and Z.~Vondra\v cek, {\em Potential theory of subordinate killed Brownian motion in a domain}, Probab. Theory Related Fields \textbf{125} (2003), 578--592.

\bibitem{Str_2010}
M.~Struwe, {\em Variational methods: Applications to nonlinear partial differential equations and Hamiltonian systems}, Springer, 2010.

\bibitem{Tri_1978}
H. Triebel, {\em Interpolation Theory, Function Spaces, Differential Operators}, North-Holland, Amsterdam, 1978.


\bibitem{Ver-handbook} L. V\'eron, {\em Elliptic equations involving measures. Stationary partial differential equations}, Vol. I, 593--712, Handb. Differ. Equ., North-Holland, Amsterdam, 2004.
\end{thebibliography}

\end{document}